\documentclass[11pt]{amsart}

\makeatletter
\def\@settitle{\begin{center}%
  \baselineskip14\p@\relax
    \normalfont\LARGE

  \@title
  \end{center}%
}
\makeatother
\usepackage{epigraph}
\usepackage[dvipsnames]{xcolor}
\usepackage[hyperfootnotes=false]{hyperref}
\usepackage{enumerate}
\usepackage{tikz}
\usepackage{tikz-cd}
\usepackage{hyperref}
\usepackage[bbgreekl]{mathbbol}
\usepackage{amsthm,amsfonts,amsmath,amscd,amssymb}
\usepackage{xcolor,float}
\usepackage[all]{xy}
\usepackage{mathrsfs}
\usepackage{graphics,graphicx}
\usepackage{caption}
\usepackage{comment}
\usepackage{array}
\usepackage{enumitem}
\usepackage{stmaryrd}
\usepackage{tabularx}

\usepackage{xspace}

\usepackage{adjustbox}

\newcolumntype{P}[1]{>{\centering\arraybackslash}p{#1}}
\newcolumntype{M}[1]{>{\centering\arraybackslash}m{#1}}

\usepackage{epigraph}
\usepackage[capitalize, noabbrev, nosort]{cleveref}

\usepackage{bm}

\let\oldmarginpar\marginpar
\renewcommand\marginpar[1]{\-\oldmarginpar[\raggedleft\footnotesize #1]%
	{\raggedright\footnotesize #1}}
\theoremstyle{plain}

\newtheorem{thm}{Theorem}[section]
\newtheorem{lemma}[thm]{Lemma}
\newtheorem*{theorem*}{Theorem}
\newtheorem*{corollary*}{Corollary}
\newtheorem{prop}[thm]{Proposition}
\newtheorem{proposition}[thm]{Proposition}

\newtheorem{cor}[thm]{Corollary}

\theoremstyle{definition}

\newtheorem{definition}[thm]{Definition}

\newtheorem{example}[thm]{Example}

\newtheorem{remark}[thm]{Remark}

\numberwithin{equation}{section}

\usepackage{mdframed}
\usepackage{lipsum}
\newtheorem*{theo}{Main Theorem}
\newenvironment{ftheo}
  {\begin{mdframed}[innertopmargin = 3pt, innerbottommargin=3pt,skipabove=5pt,skipbelow=5pt,linewidth=0.25pt,nobreak=true,align=center]\begin{theo}}
  {\end{theo}\end{mdframed}}
\def\maintheoref{the \hyperref[thm:main_result_equality_DO_NOT_CREF]{Main Theorem}\xspace}

\newcommand{\Z}{\mathbb{Z}}

\newcommand{\R}{\mathbb{R}}
\newcommand{\C}{\mathbb{C}}

\newcommand{\la}{\lambda}

\renewcommand{\d}{\delta}
\newcommand{\e}{\varepsilon}

\newcommand{\ip}{\, \lrcorner \,}

\newcommand{\sse}{\subset}
\newcommand{\lr}{\longrightarrow}

\newcommand{\x}{\times}

\newcommand{\SL}{\operatorname{SL}}

\newcommand{\Br}{\operatorname{Br}}

\newcommand{\sign}{\operatorname{sign}}

\newcommand{\Spec}{\operatorname{Spec}}

\newcounter{daggerfootnote}

\newcommand{\bG}{\mathbb{G}}

\newcommand{\B}{\mathsf{B}}

\def\prop{\lambda}

\newcommand{\n}{\mathfrak{n}}
\renewcommand{\l}{\mathfrak{l}}

\newcommand{\cc}{\mathfrak{c}}

\newcommand{\ww}{\mathfrak{w}}

\newcommand{\Hom}{\mathrm{Hom}}
\usepackage{dsfont}
\allowdisplaybreaks

\newcommand{\id}{\mathrm{id}}
\def \vertbar [#1](#2,#3,#4){
    \draw [#1] (#2,#3) -- (#2,#4);
    \draw [fill=white] (#2,#3) circle [radius=0.1];
    \draw [fill=black] (#2,#4) circle [radius=0.1];
}

\usepackage{mathdots}
\usepackage{afterpage}
\makeatletter
\providecommand{\leftsquigarrow}{%
  \mathrel{\mathpalette\reflect@squig\relax}%
}
\newcommand{\reflect@squig}[2]{%
  \reflectbox{$\m@th#1\rightsquigarrow$}%
}
\makeatother

\usepackage{subcaption}

\makeatletter
\def\Ddots{\mathinner{\mkern1mu\raise\p@
\vbox{\kern7\p@\hbox{.}}\mkern2mu
\raise4\p@\hbox{.}\mkern2mu\raise7\p@\hbox{.}\mkern1mu}}
\makeatother

\def \horline [#1](#2,#3,#4){
    \draw [#1] (#2,#4) -- (#3,#4);
    \draw [fill=white] (#2,#4) circle [radius=0.1];
    \draw [fill=black] (#3,#4) circle [radius=0.1];
}

\def \crossing (#1,#2)(#3,#4){
\draw (#1,#2) -- (#3,#4);
\draw (#1,#4) -- (#3,#2);
}

\DeclareFontFamily{U}{mathb}{}
\DeclareFontShape{U}{mathb}{m}{n}{
  <-5.5> mathb5
  <5.5-6.5> mathb6
  <6.5-7.5> mathb7
  <7.5-8.5> mathb8
  <8.5-9.5> mathb9
  <9.5-11.5> mathb10
  <11.5-> mathbb12
}{}

\newcommand{\MSB}[1]{\textcolor{red}{[MSB: \footnotesize #1]}}

\newcommand{\PG}[1]{\textcolor{green!70!black}{[PG: \footnotesize #1]}}
\newcommand{\MG}[1]{\textcolor{orange}{[MG: \footnotesize #1]}}

\newcommand{\rw}{\overrightarrow{\ww}}
\newcommand{\lw}{\overleftarrow{\ww}}

\def\dbr{\bbbeta}
\def\negind{J_{-}}
\def\posind{J_{+}}
\def\dstr{\ddot{\mathbf{s}}}
\def\wo{w_\circ}

\def\br{\beta}
\def\Graph{\Gamma}

\def\DW{\ddot\ww}
\def\borel{\mathsf{B}}
\def\G{\mathsf{G}}
\def\U{\mathsf{U}}

\makeatletter
\DeclareRobustCommand{\cev}[1]{%
  \mathpalette\do@cev{#1}%
}
\newcommand{\do@cev}[2]{%
  \fix@cev{#1}{+}%
  \reflectbox{$\m@th#1\vec{\reflectbox{$\fix@cev{#1}{-}\m@th#1#2\fix@cev{#1}{+}$}}$}%
  \fix@cev{#1}{-}%
}
\newcommand{\fix@cev}[2]{%
  \ifx#1\displaystyle
    \mkern#23mu
  \else
    \ifx#1\textstyle
      \mkern#23mu
    \else
      \ifx#1\scriptstyle
        \mkern#22mu
      \else
        \mkern#22mu
      \fi
    \fi
  \fi
}

\makeatother

\def\Xbul{X_{\bullet}}
\def\Ybul{Y_{\bullet}}
\def\Fbul{F_{\bullet}}

\def\bmp{\brWe}
\def\brWe{\dbr^{(-|+)}}

\def\Rbr3D{R}
\def\RbrWe{X}

\def\Bz_#1(#2){B_{#1}(#2)}
\def\BZ_#1(#2){B_{#1}(#2)}
\def\BZop{B}
\def\c{c}
\def\cp{d}

\def\z{z}
\def\ip{i'}
\def\zp{z'}

\def\iso{\varphi}

\def\Reg(#1){\operatorname{Reg}(#1)}
\def\bs{\backslash}

\def\Ptop{P^\uparrow}
\def\Pbot{P^\downarrow}

\def\gPathX(#1){Z(#1)}
\def\wPathX(#1){w(#1)}

\def\gPtop_#1{\gPathX(\Ptop_{#1})}
\def\wPtop_#1{\wPathX(\Ptop_{#1})}
\def\gPbot_#1{\gPathX(\Pbot_{#1})}
\def\wPbot_#1{\wPathX(\Pbot_{#1})}

\def\dw{\dot w}

\def\hp{h^+}

\def\H{H}

\def\uc{u^{(\c)}}

\def\jc{j^{(\c)}}
\def\bj{\mathbf{j}}
\def\bi{\mathbf{i}}
\def\bjc{\mathbf{j}^{(\c)}}
\def\ds{\dot s}
\def\ap{v^{(e)}}
\def\bigR{\mathcal{Y}}
\def\bigRo{\mathcal{Y}^{\circ}}
\def\bigRofix{\mathcal{Z}^{\circ}}
\def\bigRfix{\mathcal{Z}}
\def\tV{\tilde{V}}
\def\tT{\tilde{T}}

\newcommand{\smat}[1]{\left[\begin{smallmatrix}
      #1
    \end{smallmatrix}\right]}

\def\<{{\langle}}
\def\>{{\rangle}}

\def\bgamma{\bm{\gamma}}

\def\nudc{\nu^{(\c)}_{\e}}
\def\fdc{f^{(\c)}_{\e}}

\def\curve{\gamma}
\def\fug{f_{\bgamma,\bj}}
\def\fugjp{f_{\bgamma,\bj'}}
\def\fuggp{f_{\bgamma',\bj}}
\def\fugmap{\phi}
\def\chip(#1){\chi^+(#1)}
\def\l{\ell}
\def\charof{\chi}
\def\Dual{Dual\xspace}
\def\dual{dual\xspace}

\def\etop{e^{\top}}
\def\propagR{\vec\rho} 

\def\propag{\cev\rho} 

\def\PD{\Gamma}
\def\d{d}
\def\torus{H}

\def\etop{e^{\operatorname{top}}}

\def\doublestring{\dstr} 

\def\bijX{\phi_X}
\def\bijY{\phi_Y}
\def\nm{{n+m}}
\def\threeD{{\operatorname{D}}}
\def\We{\operatorname{W}}
\def\x3D{x^{\threeD}}
\def\xWe{x^{\We}}
\def\bx3D{{\bm{x}}^{\threeD}}
\def\bxWe{{\bm{x}}^{\We}}

\def\J3D{J^{\threeD}_{\dbr}}
\def\JWe{J^{\We}_{\brWe}}

\def\grid{\Delta}

\def\ord{\operatorname{ord}}
\def\ordxx_#1#2{\ord_{V_{#1}}#2}

\def\gamchp{\gamma^+}
\def\gamchm{\gamma^-}

\def\e{e}

\def\weave{\ww}
\def\lcyc{\nu}
\def\edge{\mathtt{e}} 
\def\gamWe{\gamma^{\We}}
\def\vertex{\mathtt{v}}
\def\ne{\mathtt{ne}}

\def\k{{k}}
\def\triangle{\vartriangle} 
\def\invc{{\overline{c}}}

\def\inve{{\overline{e}}}
\def\f{f}
\def\invf{\overline{f}}
\def\Om3D{\Omega^\threeD}
\def\OmWe{\Omega^{\We}}
\def\minor{\Delta}

\def\grid{\Delta}
\DeclareMathOperator{\dlog}{dlog}

\newcommand{\Rrel}[1]{\stackrel{#1}{\longrightarrow}}
\newcommand{\Rwrel}[1]{\stackrel{#1}{\Longrightarrow}}

\newcommand{\Lwrel}[1]{\stackrel{#1}{\Longleftarrow}}

\usetikzlibrary{decorations.markings, arrows}
\tikzset{tangent/.style={decoration={markings,mark=at position #1 with {
      \coordinate (tangent point-\pgfkeysvalueof{/pgf/decoration/mark info/sequence number}) at (0pt,0pt);
      \coordinate (tangent unit vector-\pgfkeysvalueof{/pgf/decoration/mark info/sequence number}) at (1,0pt);
      \coordinate (tangent orthogonal unit vector-\pgfkeysvalueof{/pgf/decoration/mark info/sequence number}) at (0pt,1);
      }},postaction=decorate},
    use tangent/.style={
        shift=(tangent point-#1),
        x=(tangent unit vector-#1),
        y=(tangent orthogonal unit vector-#1)
    },
    use tangent/.default=1
    }
\DeclareMathOperator{\pre}{pre}
\DeclareMathOperator{\Lec}{Lec}
\DeclareMathOperator{\Ing}{Ing}
\DeclareMathOperator{\Men}{M}

\def\figref#1(#2){Figure~\hyperref[#1]{\ref*{#1}(#2)}}

\begin{document}

	\title{Comparing cluster algebras on braid varieties}
	

	\author{Roger Casals}
	\address{University of California Davis, Dept. of Mathematics, USA}
	\email{casals@ucdavis.edu}

\author{Pavel Galashin}
	\address{University of California Los Angeles, Dept. of Mathematics, USA}
	\email{galashin@math.ucla.edu}

  \author{Mikhail Gorsky}
\address{Universit\"at Hamburg, Fachbereich Mathematik, Bundesstraße 55, 20146 Hamburg, Germany}
\email{mikhail.gorskii@univie.ac.at}

 \author{Linhui Shen}
\address{Michigan State University, Dept.~of Mathematics, USA}
\email{linhui@math.msu.edu}

\author{Melissa Sherman-Bennett}
\address{University of California Davis, Dept. of Mathematics, USA}
\email{mshermanbennett@ucdavis.edu}

 \author{Jos\'e Simental}
\address{Instituto de Matem\'aticas, Universidad Nacional Aut\'onoma de M\'exico, M\'exico}
\email{simental@im.unam.mx}

\begin{abstract}
Braid varieties parametrize linear configurations of flags with transversality conditions dictated by positive braids. They include and generalize reduced double Bruhat cells, positroid varieties, open Bott-Samelson varieties, and Richardson varieties, among others. Recently, two cluster algebra structures were independently constructed in the coordinate rings of braid varieties: one using weaves and the other using Deodhar geometry. The main result of the article is that these two cluster algebras coincide. More generally, our comparative study matches the different concepts and results from each approach to the other, both on the combinatorial and algebraic geometric aspects.
\end{abstract}

\maketitle
\setcounter{tocdepth}{1}
\vspace{-1.1cm}
\tableofcontents
\vspace{-1.1cm}


\section{Introduction}

The coordinate rings of braid varieties have recently been proven to be cluster algebras. Two independent constructions appeared: using weaves and using Deodhar geometry. The object of this article is to compare these two cluster algebras. The main contribution of the manuscript is that, appropriately understood, these two cluster algebras coincide.

En route, the article provides a detailed comparative study of each of the two constructions, explaining how to develop the necessary concepts and results on each side through the lens of the other. On the geometric side, this includes building explicit isomorphisms between braid varieties and double braid varieties, between weave tori and Deodhar tori, comparing the weave 2-form and the Deodhar 2-form, and the weave and Deodhar cluster variables. On the combinatorial side, this includes matching the combinatorics of double braid words with the combinatorics of double strings and double inductive weaves, and the orders of vanishing of certain chamber minors on Deodhar hypersurfaces with the tropical Lusztig rules on weaves.


\subsection{Scientific context} Recently, the two articles \cite{CGGLSS} and \cite{GLSB} have shown that the coordinate rings of braid varieties are cluster algebras for any simple Lie group type $\G$. Braid varieties were introduced in \cite{CGGS}, and note that \cite{GLSB} was preceded by \cite{GLSBS}, which focused on the case $\G=\SL$ and 3D plabic graphs.  In brief, braid varieties are irreducible smooth affine varieties that parametrize linear configurations of flags with transversality conditions dictated by a positive braid. They generalize known classes of varieties appearing in Lie theory and algebraic combinatorics, such as positroid and Richardson varieties, and appear prominently in low-dimensional contact and symplectic topology, cf.~e.g.~\cite{CasalsHonghao,CasalsGao24,casals_cbms}. Succinctly, the two approaches of \cite{CGGLSS} and \cite{GLSB} can be described as follows:

\begin{enumerate}
    \item The main result of \cite{CGGLSS} is the construction of a cluster algebra structure on $\C[X(\beta)]$, where $X(\beta)$ is the braid variety associated to a positive braid word $\beta$. The key technique in \cite{CGGLSS} is weave calculus, as introduced in \cite{CasalsZaslow}, and see also \cite{CGGS}. At core, a weave $\ww:\beta\to\beta'$ is a combinatorial object, given by a diagram in the plane, which encodes a way to relate two braids words $\beta,\beta'$ via braid moves and the moves $\sigma_i^2\to\sigma_i$, where $\sigma_i$ is an Artin generator for the braid group. The cluster algebra structure on $\C[X(\beta)]$ built in \cite{CGGLSS} is such that a weave $\ww:\beta\to\delta(\beta)$ from a braid $\beta$ to its Demazure product $\delta(\beta)$ defines a cluster seed on $\C[X(\beta)]$. Specifically, such a weave defines a weave torus $T_\ww\sse X(\beta)$, a weave quiver $Q_\ww$, and a corresponding set of cluster variables $\bxWe\in \C[X(\beta)]$.\\

\item The main result of \cite{GLSB} is the construction of a cluster algebra structure on $\C[R(\dbr)]$, where $R(\dbr)$ is the double braid variety associated to a double braid word $\dbr$. The key technique in \cite{GLSB} is using a generalization of the Deodhar decomposition of Richardson varieties and certain rational expressions of generalized minors. The cluster algebra structure on $\C[R(\dbr)]$ built in \cite{GLSB} is such that the double braid word $\dbr$ itself defines a cluster seed on $\C[R(\dbr)]$. Specifically, it defines a Deodhar torus $T_\dbr\sse R(\dbr)$, a Deodhar quiver $Q_\dbr$, and a corresponding set of cluster variables $\bx3D\in \C[R(\dbr)]$.
\end{enumerate}

Each technique has its strengths and shortcomings. By developing a detailed comparison between these two approaches, as we do in this manuscript, the advantages of each technique explicitly translate to the other. For instance, the propagation algorithm for weaves, which naturally follows the tropical Lusztig rules, can now be also implemented as a propagation algorithm for the Deodhar cluster variables. We have summarized some of the key ingredients and results of our comparative study in Tables \ref{table:introduction} and \ref{table:summary_intro} below.


\subsection{Main result} The combinatorial and algebraic geometric objects appearing in \cite{CGGLSS} and \cite{GLSB}, the two works we are comparing, are qualitatively different. To wit, we summarize the key aspects of each work as follows. In \cite{CGGLSS}, braid varieties $X(\beta)$ are associated to a braid word $\beta$. Given a Demazure weave $\ww:\beta\to\delta(\beta)$, three objects are constructed:
\begin{enumerate}
    \item the weave torus $T_\ww\sse X(\beta)$, via the flag transversality conditions imposed by $\ww$,
    \item the weave cluster variables $\bxWe\in\C[X(\beta)]$, measuring transversality according to $\ww$,
    \item the weave quiver $Q^{\rm W}$, equivalently encoded in the weave 2-form $\OmWe\in\Omega^2X(\beta)$.
\end{enumerate}

The cluster algebra structure on $\C[X(\beta)]$ constructed in \cite{CGGLSS} is such that the triple $(T_\ww,\bxWe,\OmWe)$ associated to a given Demazure weave $\ww$ is a cluster seed for $\C[X(\beta)]$. Non-equivalent Demazure weaves yield different cluster seeds. We refer to this cluster algebra structure on $\C[X(\beta)]$, built in \cite{CGGLSS}, as the {\it weave cluster algebra}. In essence, the weave cluster algebra uses the combinatorics of weaves as a way to govern transversality conditions of tuples of flags, and each cluster seed is itself named in terms of a weave, via the tropical combinatorics of its Lusztig cycles.

In \cite{GLSB} double braid varieties $R(\dbr)$ are associated to double braid words $\dbr$. From a given double braid word $\dbr$, three objects are constructed:
\begin{enumerate}
    \item the Deodhar torus $T_\dbr\sse R(\dbr)$, with cocharacter lattice, via certain chamber minors,
    \item the Deodhar cluster variables $\bx3D\in\C[R(\dbr)]$, via the vanishing of certain characters,
    \item the Deodhar quiver $Q^{\rm D}$, equivalently encoded in the Deodhar 2-form $\Om3D\in\Omega^2R(\dbr)$.
\end{enumerate}

The cluster algebra structure on $\C[R(\dbr)]$ constructed in \cite{GLSB} is such that the triple $(T_\dbr,\bx3D,\Om3D)$ associated to a given double braid word $\dbr$ is a cluster seed for $\C[R(\dbr)]$. We refer to this cluster algebra structure on $\C[R(\dbr)]$, built in \cite{GLSB}, as the {\it Deodhar cluster algebra}. In a nutshell, the Deodhar cluster algebra uses a Deodhar-type decomposition of the double braid variety, whose strata are dictated by the flag relations as imposed by $\dbr$. Each cluster seed is itself named in terms of certain unique characters determined by their vanishing along codimension-1 strata: intuitively, for each codimension-1 stratum, a cluster variable is named as the unique character on the Deodhar torus solely vanishing along that hypersurface.

The core of this article is a detailed comparison of the weave and Deodhar cluster algebras. It occurs in two different layers, the combinatoral comparison and the geometric comparison:
\begin{itemize}
    \item[(i)] Combinatorially, we explain how to transition from a double braid word $\dbr$ to a certain braid word $\beta=\dbr^{(-|+)}$, establish a dictionary between double braid words $\dbr$, double strings $\dstr$ and double inductive weaves $\DW$, and relate the tropical Lusztig propagation rules for weaves to a type of propagation of Deodhar cluster variables in chamber minors.

    \item[(ii)] Geometrically, we establish isomorphisms between double braid varieties and braid varieties, and show that they map weave tori to Deodhar tori, weave cluster variables to Deodhar cluster variables and the weave 2-form to the Deodhar 2-form. See \cref{table:summary_intro}.
\end{itemize}

The qualitatively different nature of these two constructions, one based on the combinatorics of weaves and the other on properties of chamber minors and character lattices, leads to a comparison with substance. For instance, we prove a new formula expressing the Cartan element $\hp_c$ in \cite{GLSB} in terms of the $u$-variables for an edge labeling of a weave. While both $\hp_c$ and the $u$-variables are crucial in defining the respective cluster variables $\bx3D$ and $\bxWe$, such formula is of interest on its own. Another instance, the tropical Lusztig propagation rules, translated through our comparison, lead to a new propagating algorithm to compute the Deodhar cluster variables $\bx3D$, a desired procedure that was missing in \cite{GLSB}.

In summary, the main result of this article is that the weave and Deodhard cluster algebras are isomorphic. There are also new contributions in the techniques and results we develop to establish such isomorphism, which are provided throughout the article, cf.~e.g. \cref{table:summary_intro}. For reference, we state the main contribution, again emphasizing that the ingredients developed for its proof are by themselves new and of interest:

\phantomsection
\begin{ftheo}[\textcolor{blue}{Comparison of weave and Deodhar cluster algebras}]\label{thm:main_result_equality_DO_NOT_CREF}
Let $\dbr$ be a double braid word. Then, there exists an isomorphism of algebraic varieties
$$\iso: R(\dbr) \xrightarrow{\sim} X(\dbr^{(-|+)})$$
such that the following equalities hold:

\begin{enumerate}
    \item Equality of tori: $\iso(T_\dbr)=T_{\DW(\dstr(\dbr))}$,\\

    \item Equality of cluster variables: $\iso^{\ast}\bxWe=\bx3D$,\\

    \item Equality of 2-forms: $\iso^{\ast}\OmWe = \Om3D$.
\end{enumerate}
This implies that the weave cluster algebra structure on $\C[X(\dbr^{(-|+)}]$ is isomorphic to the Deodhar cluster algebra structure on $\C[R(\dbr)]$, via the isomorphism $\iso^{\ast}$.
\end{ftheo}

In the statement above, $\dbr^{(-|+)}$ denotes a specific braid word associated to the given double braid word $\dbr$, and $\DW(\dstr(\dbr))$ denotes the double inductive weave of the double string associated to the double braid word $\dbr$. \cref{table:summary_intro} summarizes where these objects and constructions between them are established in the article.

Beyond \maintheoref above, the manuscript contains a number of additional details and comparisons of independent value. To wit, \cref{ssec:weaves_and_3Dplabicgraphs} provides the combinatorial relation between the 3D plabic graphs from \cite{GLSBS} and weaves for Lie Type $A$, e.g.~$\G=\SL$, and \cref{ssec:compare-cocharacters-type-A} develops the relation between monotone multicurves, which are associated to 3D plabic graphs, and the tropical rules for Lusztig cycle propagation in weaves. In particular, \cref{sssec:propagation_Lusztig_data} contains an interesting connection between the behavior of certain cocharacters under the propagation rules and string operators. Appendices B and C contain two more comparisons: between the double braid moves for double braid words from \cite{GLSB} and the double string moves for double inductive weaves from \cite{CGGLSS}, cf.~\cref{table:move_comparison}, and between constructions of cluster algebra structures in the coordinate rings of open Richardson varieties, cf.~\cref{thm:richardson-compare}.

\begin{table}
\begin{center}
\caption{The key concepts being compared on the construction of cluster algebras from \cite{CGGLSS} and \cite{GLSB}. The results of this article allow us to translate between the corresponding two columns.}\label{table:introduction}
\small
\begin{tabular}{ |P{3.8cm}|P{5.5cm}|P{4.8cm}|  }
	\hline
	\multicolumn{3}{|c|}{ {\bf Schematic comparison in general Lie type $\G$}} \\
	\hline

& In {\bf \cite{CGGLSS}} & In {\bf \cite{GLSB}}\\
\hline

Input data	 & Braid word $\br$ & Double braid word $\dbr$ \\
\hline

Geometric space	 & Braid variety $X(\br)$ & Double braid variety $R(\dbr)$\\
\hline

Braid move &  $X(\br) \cong X(\br')$ if $\br, \br'$ & $R(\dbr)\cong R(\dbr')$ if $\dbr, \dbr'$  \\
	isomorphisms &   related by braid move & related by double braid move \\
\hline

Commutative \textcolor{purple}{{\bf algebra}} & Regular functions $\textcolor{purple}{\C[X(\br)]}$ & Regular functions  $\textcolor{purple}{\C[R(\dbr)]}$\\
\hline

Extra choice of & A Demazure weave $\ww:\br\to\delta(\br)$, & \rule[-.4\baselineskip]{0.25\textwidth}{.5pt}\\
combinatorial data & which determines Lusztig cycles & \\
\hline

Open {\bf \textcolor{teal}{torus charts}} & Many tori $\textcolor{teal}{T_\ww}\sse X(\br)$, & One torus $\textcolor{teal}{T_{\dbr}} \subset R(\dbr)$,
\\
in geometric space& indexed by weave classes $\ww$ & the Deodhar torus\\
\hline

Effect of braid move  & $\textcolor{teal}{T_{\ww'}}\sse X(\br')$ maps to & $\textcolor{teal}{T_{\dbr'}} \subset R(\dbr')$ maps to
\\
isomorphisms on {\bf \textcolor{teal}{tori}} & a weave torus $\textcolor{teal}{T_\ww}\sse X(\br)$& $\textcolor{teal}{T_{\dbr}} \subset R(\dbr)$ or different torus \\
\hline

{\bf \textcolor{blue}{Quiver}} for each & {\bf \textcolor{blue}{Intersection of Lusztig cycles}}, & Coeff.~of {\bf \textcolor{blue}{2-form $\omega_\dbr$ on}} $R(\dbr)$, \\
torus chart & as sum of local contributions & from generalized minors\\
\hline

Cluster variables & Measure flag transversality, & Characters determined by
\\
(in theory) & imposed by each Lusztig cycle $\ww$ &  Deodhar hypersurfaces\\
\hline

{\bf \textcolor{orange}{Cluster variables}} & Regular functions on $T_\ww$ & Irreducible factors
\\
(in practice) & by {\bf \textcolor{orange}{scanning weave $\ww$}} & of {\bf \textcolor{orange}{chamber minors}}\\
\hline
\end{tabular}
\normalsize
\end{center}
\end{table}

\begin{table}[h!]
\begin{center}
\caption{A schematic reference for some of the new key concepts, constructions and results in the comparison. The core results are marked in {\bf bold font}. Objects of a combinatorial or Lie-theoretic nature are highlighted in \textcolor{blue}{blue}, objects of a more algebraic geometric nature are highlighted in \textcolor{purple}{red}.}\label{table:summary_intro}
\begin{tabular}{ |P{3cm}|P{10cm}|  }
	\hline
	\multicolumn{2}{|c|}{ {\bf Some specific ingredients in the comparison}} \\
	\hline

\cref{def:pos-neg-ind} & double braid word \textcolor{blue}{$\dbr$} $\leadsto$ braid word \textcolor{blue}{$\dbr^{(-|+)}$} \\
\hline

\cref{def:braid-variety} & braid word \textcolor{blue}{$\br$} $\leadsto$ braid variety \textcolor{purple}{$X(\br)$}\\
\hline

\cref{def:dbl-braid-variety} & double braid word \textcolor{blue}{$\dbr$} $\leadsto$ double braid variety \textcolor{purple}{$R(\dbr)$} \\
\hline

{\bf\cref{lem:iso-glsbs-cgglss}} & {\bf Isomorphism \textcolor{purple}{$\iso: R(\dbr) \xrightarrow{\sim} X(\dbr^{(-|+)})$}}\\
\hline

\cref{def:dbl-inductive-weave} & double string \textcolor{blue}{$\dstr$} $\leadsto$ double inductive weave \textcolor{blue}{$\DW(\dstr)$}\\
\hline

\cref{def:double_string_from_double_braid_word} & double braid word \textcolor{blue}{$\dbr$} $\leadsto$ double string \textcolor{blue}{$\dstr(\dbr)$}\\
\hline

\cref{sssec:weavetori_doubleinductive} & double inductive weave \textcolor{blue}{$\DW$} $\leadsto$ weave torus \textcolor{purple}{$T_{\DW}\sse X(\beta)$}\\
\hline

\cref{sec:deodhar} & double braid word \textcolor{blue}{$\dbr$} $\leadsto$ Deodhar torus \textcolor{purple}{$T_{\dbr}\sse R(\dbr)$}\\
\hline

{\bf\cref{lem:weave-tori-vs-deodhar-tori}} & {\bf Equality of tori \textcolor{purple}{$\iso(T_\dbr)=T_{\DW}$} for \textcolor{blue}{$\DW:=\DW(\dstr(\dbr))$}}\\
\hline

Sections \ref{sec:weave-cluster-variables1} \& \ref{sec:weave-cluster-variables2} & Weave cluster variables \textcolor{purple}{$\bxWe$}\\
\hline

Sections \ref{ssec:deodhar_variables1} \& \ref{ssec:deodhar_variables2} & Deodhar cluster variables \textcolor{purple}{$\bx3D$}\\
\hline

\cref{lem:hc-in-terms-of-uc} & \textcolor{blue}{$\hp_c$} in terms of \textcolor{blue}{$u$-variables}\\
\hline

{\bf \cref{thm:coincidence-cluster-variables}} & {\bf Equality of cluster variables \textcolor{purple}{$\iso^{\ast}\bxWe=\bx3D$}}\\
\hline

\cref{def:lusztig-data-from-dbl-weave} & Lusztig datum \textcolor{blue}{$[\bjc, \nudc]$} from weaves\\
\hline

\cref{ssec:consequence_equality_variables} & Deodhar cocharacter \textcolor{blue}{$\gamchp_{\dbr, \c, \e}$} and \textcolor{blue}{$[\bjc, \nudc]$} coweight\\
\hline

\cref{ssec:weave-2-form} & Weave 2-form \textcolor{blue}{$\OmWe$}, i.e.~exchange matrix for $T_\ww$\\
\hline

\cref{ssec:Deodhar-2-form} & Deodhar 2-form \textcolor{blue}{$\Om3D$}, i.e.~exchange matrix for $T_\dbr$\\
\hline

{\bf\cref{thm:coincidence-forms}} & {\bf Equality of 2-forms \textcolor{blue}{$\iso^{\ast}\OmWe = \Om3D$}}\\
\hline

\cref{ssec:proof_main_theorem} & Proof of \maintheoref.\\
\hline

\end{tabular}
\end{center}
\end{table}

{\bf Acknowledgements}. We are grateful to the American Institute of Mathematics for funding and hosting the workshop ``Cluster algebras and braid varieties'' (January 2023), where this project started. R.~Casals is supported by the National Science Foundation under grants DMS-2505760 and DMS-1942363 and a Sloan Research Fellowship. P.~Galashin is supported by the Sloan Fellowship and by the National Science Foundation under Grants No.~DMS-1954121 and No.~DMS-2046915.
 M.~Gorsky is supported by the Deutsche Forschungsgemeinschaft SFB 1624 ``Higher structures, moduli spaces and integrability'' (506632645), and the project ``Refined invariants in combinatorics, low-dimensional topology and geometry of moduli spaces (REFINV)'' of the ERC grant No.\ 101001159 under the European Union’s Horizon 2020 research and innovation programme. 
 L. Shen is supported by the National Science Foundation under DMS-2200738.
 M. Sherman-Bennett is supported by the National Science Foundation under DMS-2103282 and DMS-2444020. J. Simental is supported by UNAM's PAPIIT Grant IA102124 and SECIHTI Project CF-2023-G-106.\qed

\section{Comparison of algebraic varieties}\label{sec:braid_varieties}

\noindent This section focuses on comparing the algebraic varieties whose coordinate rings are shown to be cluster algebras in \cite{CGGLSS,GLSBS,GLSB}:
\begin{enumerate}
    \item In \cite{CGGLSS}, one considers braid varieties $X(\beta)$, associated to a braid word $\beta$;
    \item In \cite{GLSBS,GLSB} one considers double braid varieties $R(\dbr)$, associated to a double braid word $\dbr$.
\end{enumerate}
These varieties $X(\beta)$ and $R(\dbr)$ are introduced in \cref{ssec:braid-varieties}, after describing their combinatorial inputs $\beta$ and $\dbr$ in \cref{ssec:braid_words}, and their comparison is explained in \cref{ssec:comparison_varieties}.


\subsection{Braid words and double braid words}\label{ssec:braid_words}

In \cite{CGGLSS}, braid words $\beta$ use the alphabet $I$, whereas \cite{GLSBS,GLSB} employ braid words $\dbr$ in the alphabet $\pm I$. The latter type of words will be referred as double braid words and the former simply as braid words, to mark the difference. The first transitional step in comparing these approaches is the comparison of such braid words. Starting with $\dbr$, as used in \cite{GLSBS,GLSB}, the corresponding braid word $\beta$, as used in \cite{CGGLSS}, is described as follows:

\begin{definition}
    \label{def:pos-neg-ind}
    Let $\dbr= i_1 \cdots i_{n+m}$ be a double braid word in the alphabet $\pm I$ and consider
    $$\negind = \{a_1 < a_2 < \cdots < a_n\} \subseteq [n+m],\quad \posind = \{b_1 < b_2 < \cdots < b_m\} \subseteq [n+m],$$
    the sets of indices of the negative and positive letters of $\dbr$, respectively. By definition, the braid word $\dbr^{(-|+)}$ associated to $\dbr$ is
\begin{equation}\label{eq:def:pos-neg-ind}
      \dbr^{(-|+)}:=(-i_{a_1}^{\ast})\cdots (-i_{a_n}^{\ast})i_{b_m}\cdots i_{b_1}.
\end{equation}
By construction, $\dbr^{(-|+)}$ is a braid word in the alphabet $I$.\hfill$\Box$
\end{definition}

In \cref{def:pos-neg-ind}, we denoted by $-^{\ast}: I \to I$ the map defined by the property that conjugating by $\wo$ sends $s_i$ to $s_{i^*}$, and set $(-i)^{\ast} := -i^{\ast}$. A braid word  $\br = i_1\cdots i_l$ defines a word $s_{i_1} \cdots s_{i_l}$ in the generators of $W$. The \emph{Demazure product} $\delta(\br)$ of $\br$ is the maximal element $w \in W$ with respect to the Bruhat order such that $s_{i_1} \cdots s_{i_l}$ contains a reduced expression of $w$ as a (not necessarily consecutive) subexpression. The Demazure product can be equivalently defined inductively by the following rule:
\begin{equation}\label{def:demazure}
\delta(i):=s_i,\
\delta(\beta \cdot i):=\begin{cases}
\delta(\beta)s_i & \text{if}\  \ell(\delta(\beta)s_i)=\ell(\delta(\beta))+1\\
\delta(\beta) & \text{if}\ \ell(\delta(\beta)s_i)=\ell(\delta(\beta))-1.\\
\end{cases}
\end{equation}
We will also write $s_i \ast s_j$ for the Demazure product $\delta(i \cdot j)$. We define the Demazure product of a double braid word as follows.

\begin{definition}\label{def:si+-}
	For $i \in \pm I$, let 
	\[s_i^+:= \begin{cases}
		s_i & \text{if } i \in I\\
		\id & \text{if } i \in -I
	\end{cases}
	\qquad s_i^-:= \begin{cases}
		\id & \text{if } i \in I\\
		s_{|i|} & \text{if } i \in -I
	\end{cases}.\]
By definition, the Demazure product of a double braid word $\dbr=i_1 \cdots i_l$ is 
	\[\delta(\dbr) :=s_{i_1^{*}}^{-}*\cdots * s_{i_{l}^{*}}^{-}*s_{i_l}^{+}*\cdots * s_{i_1}^{+}
 \]
	where $*$ denotes the usual Demazure product.$\qed$
 \end{definition}

Note that $\delta(\dbr) = \delta(\brWe)$, where the left hand side of the equality is \cref{def:si+-} and the right hand side is the standard Demazure product from \cref{def:demazure}. \cref{fig:Front_Braidwords2} illustrates diagrams for (double) braid words in the case of Type A, i.e.~$\G=\SL$, where the number of strands is the rank of $\G$ plus one. In such braid diagrams, the strands are numbered increasingly starting from 1 at the highest strand, both for the strands at the bottom and at the top blocks. Note that a $\wo$ which is not part of $\dbr, \dbr^{(-|+)}$ is inserted in this diagram. 
The presence of $\wo$ is the reason for the appearance of the involution $*$ in the lower indices $s_{i^*_{a_j}}$, $j\in[1,n]$, and it is needed in order to introduce the different variants of braid varieties.

\begin{center}
	\begin{figure}[h!]
		\centering
		\includegraphics[scale=0.8]{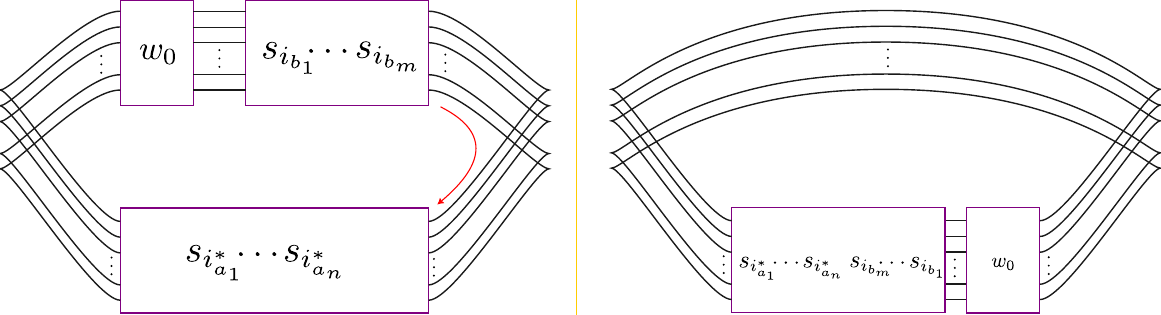}
		\caption{In Type A, $\G=\SL$, the braid words in \cref{def:pos-neg-ind} can be interpreted as geometric braids that close up to Legendrian links, whose fronts are depicted here. (Left) The front braid diagram associated to $\dbr$. (Right) The front associated to $\dbr^{(-|+)}$. The transition from left to right is given by passing, using Reidemeister moves, the two top boxes in the left diagram to the bottom through the {\it right} side, as indicated by the red arrow. This geometric transport explains why the word $s_{i_{b_1}}\ldots s_{i_{b_m}}$ at the top of the left diagram becomes $s_{i_{b_m}}\ldots s_{i_{b_1}}$ at the bottom of the right diagram. Here $w_0$ indicates a braid word for $\wo$, the appearance on the right should be understood as also reflecting (reading right-to-left) the choice of braid word for $\wo$ on the left.
        }
		\label{fig:Front_Braidwords2}
	\end{figure}
\end{center}

\begin{example}\label{ex:running}
Let $\G = \SL_3$, so that $I = \{1, 2\}$, and $\dbr = (-2, 1, 2, 1, -1, 1, 2)$. Then $\negind = \{1, 5\}$, $\posind = \{2, 3, 4, 6, 7\}$ and $\dbr^{(-|+)} = (2^{*}, 1^{*}, 2, 1, 1, 2, 1) = (1, 2, 2, 1, 1, 2, 1)$. The associated braid diagram is depicted in \cref{fig:Front_Braidwords4}. This is the running example in this manuscript.\hfill$\Box$
\end{example}

\begin{center}
	\begin{figure}[h!]
		\centering
		\includegraphics[scale=0.8]{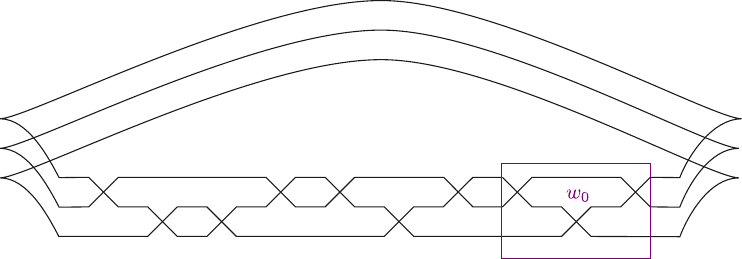}
		\caption{The braid diagram associated to the braid word $\dbr^{(-|+)}$ in \cref{ex:running}. Here crossings are all understood as positive Artin generators, i.e.~ overcrossings. Smoothly, the resulting link is a 3-component link.}
		\label{fig:Front_Braidwords4}
	\end{figure}
\end{center}


\subsection{Braid varieties and double braid varieties}\label{ssec:braid-varieties} Let us introduce braid varieties and double braid varieties in Definitions \ref{def:braid-variety} and \ref{def:dbl-braid-variety} below. The necessary preliminaries on flags, weighted flags, and their relative positions are summarized in Appendix A. The braid varieties studied in \cite{CGGLSS} are defined as follows:

\begin{definition}\label{def:braid-variety}
  Let $\br = i_1\cdots i_l$ be a braid word with the Demazure product $\delta(\br) = \wo \in W$. The \emph{braid variety} $X(\br)$ associated to $\beta$ is
$$X(\beta) := \{(\borel_0, \dots, \borel_{l}) \in (\G/\borel)^{l + 1} \mid \borel_0 = \borel, \borel_{k} \buildrel s_{i_{k+1}} \over \longrightarrow \borel_{k+1}, \borel_{l} = \wo \borel \},$$
That is, $X(\beta)$ parametrizes tuples of flags $\borel_\bullet=(\borel_0, \dots, \borel_l)$ satisfying the relative positions
	\begin{equation}\label{eq:CGGLSS-definition}
			\borel_+ = \borel_0 \Rrel{s_{i_1}}  \borel_1  \Rrel{s_{i_2}}   \cdots  \Rrel{s_{i_l}}  \borel_l = \wo \borel_{+},
	\end{equation}
dictated by $\beta$ and the boundary constraints $\borel_+ = \borel_0$ and $\borel_l = \wo \borel_{+}$.\hfill$\Box$
\end{definition}

\noindent It follows from \cite[Corollary 3.7]{CGGLSS} that $X(\beta)$ is a smooth irreducible affine variety of dimension $\ell(\beta)-\ell(\wo)$. The isomorphism type of $X(\beta)$ as an affine variety only depends on the braid $[\beta]$: it is independent of the particular braid word $\beta$ for $[\beta]$. Precisely, if $\beta$ and $\beta'$ are related by a braid move, there exists an explicit isomorphism, which depends on the braid move, between $X(\beta)$ and $X(\beta')$, cf.~\cite[Section 3]{CGGLSS}. 

\begin{remark}
In \cite{CGGLSS}, braid varieties are defined more generally without the assumption that $\delta(\br) = \wo$. However, by \cite[Lemma 3.4 (1)]{CGGLSS}, we can consider the case $\delta(\br) = \wo$ without loss of generality.\qed 
\end{remark}


\noindent There is an equivalent description of $X(\beta)$ in terms of weighted flags, cf.~\cite[Lemma 3.13]{CGGLSS}. This description is employed in Section \ref{sec:cluster_var} and also brings $X(\beta)$ closer to double braid varieties, which are defined in terms of weighted flags. Specifically, consider the algebraic variety $X_{\U}(\beta)$ defined by
\begin{equation}\label{eq:CGGLSS_weighted_dfn2}
X_{\U}(\beta):=\{(\U_0, \dots, \U_{l}) \in (\G/\U)^{l + 1} \mid \U_0 = \U_+, \U_0 \Rrel{s_{i_1}} \U_1  \Rrel{s_{i_2}}\cdots\Rrel{s_{i_l}}  \U_l,\pi(\U_l)=\wo B_+\},
\end{equation}
where $\U \Rrel{w} \U'$ denotes strong relative position $w$ between weighted flags, cf.~\cref{sec:appendix}. Then, the braid variety $X(\beta)$ is isomorphic to $X_{\U}(\beta)$:
\begin{equation}\label{eq:CGGLSS_weighted_dfn}
X(\beta)\cong X_{\U}(\beta).
\end{equation}
Thus, under this isomorphism, the braid variety $X(\beta)$ also parametrizes tuples of \emph{weighted} flags $\U_\bullet = (\U_0, \U_1, \dots, \U_l)$ satisfying:
	\begin{equation}\label{eq:CGGLSS_weighted_dfn2}
					\U_0 \Rrel{s_{i_1}} \U_1  \Rrel{s_{i_2}}   \cdots  \Rrel{s_{i_l}}  \U_l
	\end{equation}
and the boundary conditions $\U_0=\U_+$ and $\pi(\U_l)=\wo B_+$. By \cite[Lemma 3.13]{CGGLSS}, an explicit isomorphism witnessing \cref{eq:CGGLSS_weighted_dfn} is given by projecting each such tuple $\U_\bullet$ of weighted flags to the tuple $\borel_\bullet$ of underlying flags in $\G/\borel$. 
 In short, \cref{eq:CGGLSS_weighted_dfn} states that we can lift the first flag of a point in $X(\beta)$ to be the weighted flag $\U_+$ and that this lift propagates to the right in a unique manner, reading $\beta$ left-to-right. That is, once $\borel_+$ is lifted to $\U_+$ we can continue lifting each flag $\borel_i$ to a weighted flag $\U_i$ while preserving the relative positions.
\color{black}

The double braid variety studied in \cite{GLSBS,GLSB} is defined as follows.

\begin{definition}\label{def:dbl-braid-variety}
	Let $\dbr = i_1\cdots i_l$ be a double braid word with Demazure product $\delta(\dbr) = \wo \in W$. The \emph{double braid variety} $R(\dbr)$ associated to $\dbr$ is the variety of pairs of weighted flags $\{((X_{0},Y_{0}),\ldots,(X_{l}, Y_{l}))\}\sse (\G/\U\times \G/\U)^l$, satisfying the following relative position conditions
	\begin{equation}\label{eq:GLSBS-definition}
		\begin{tikzcd}
			X_0& \arrow[l,"{s_{i_1}^+}"'] X_1& \arrow[l,"{s_{i_2}^+}"'] \cdots& \arrow[l,"{s_{i_l}^+}"'] 
			X_l \\
			Y_0 \arrow[r,"{s_{i_1^\ast}^-}"'] \arrow[u, Rightarrow, "\wo"']& Y_1 \arrow[r,"{s_{i_2^\ast}^-}"'] & \cdots \arrow[r,"{s_{i_l^\ast}^-}"'] & Y_l. \ar[equal]{u}
		\end{tikzcd}
	\end{equation}
	modulo the simultaneous action of $\G$ on the left.\hfill$\Box$
\end{definition}

It follows from \cite[Section 4]{GLSB} that double braid moves applied to $\dbr$ preserve the isomorphism type of the double braid variety $R(\dbr)$, see loc.~cit.~for the definition of these double braid moves, see also \cref{rmk:compatibility-iso-moves}. That is, if $\dbr$ and $\dbr'$ are related by a double braid move, then there exists an explicit isomorphism, depending on the braid move, between $R(\dbr)$ and $R(\dbr')$. In particular, there is an isomorphism $R(\dbr) \cong R(\dbr')$ if $\dbr$ and $\dbr'$ are related by commuting a letter in $I$ with a letter in $-I$. Therefore, moving all the letters in $-I$ to the left and all the letters in $I$ to the right, we have that every double braid variety $R(\dbr)$ is naturally isomorphic to a double braid variety of the form $R(\dbr')$ where $\dbr' = (-i_1)\cdots (-i_m)j_1\cdots j_n$, with $l = m+n$ and $i_1, \dots, i_m, j_1, \dots, j_n \in I$. See \cref{fig:Front_Braidwords2} for a depiction of this.

\begin{remark}
The Deodhar torus $T_{\dbr} \subseteq R(\dbr)$ is an open torus chart that has a central role in construction of the cluster structures in \cite{GLSBS,GLSB}. It is defined in \cite[Definition 2.5]{GLSB}, cf.~also \cref{sec:deodhar} below. Such torus $T_{\dbr}$ depends crucially on the double braid word $\dbr$ in the following sense: if $\dbr$ and $\dbr'$ are related by a double braid move and $\psi:R(\dbr)\lr R(\dbr')$ is the isomorphism associated to such double braid move, then $\psi(T_{\dbr})$ is typically not the same as $T_{\dbr'}$.\hfill$\Box$
\end{remark}

Note that the symbols $s^\pm_i$ in \cref{def:si+-} are allowed to be the identity $\id \in W(\G)$ in some cases. More precisely, in \eqref{eq:GLSBS-definition} we have $Y_{j-1} = Y_j$, resp.~$X_{j-1} = X_j$, if $i_j \in I$, resp.~$i_j \in -I$. In contrast, in \cite{CGGLSS} no consecutive flags in a point of $X(\beta)$ are allowed to be equal. This is a technical difference between \cref{def:braid-variety} and \cref{def:dbl-braid-variety}.

\subsection{An isomorphism between $X(\beta)$ and $R(\dbr)$}\label{ssec:comparison_varieties} By \cref{def:pos-neg-ind}, each double braid word $\dbr$ defines a braid word $\dbr^{(-|+)}$. In this subsection, we show that the varieties $R(\dbr)$ and $X(\dbr^{(-|+)})$ are isomorphic. Intuitively, this isomorphism is obtained by reading \cref{eq:GLSBS-definition} counter-clockwise and starting at $Y_0$, while ignoring any arrows that are the identity. Specifically, the isomorphism of varieties is given as follows:


\begin{proposition}\label{lem:iso-glsbs-cgglss}
Let $\dbr = i_1\cdots i_{n+m}$ be a double braid word with $J_{-} = \{a_1 < \cdots < a_n\}$ and $J_{+} = \{b_1 < \cdots < b_m\}$. Given a $\G$-orbit $\G .(\Xbul', \Ybul') \in R(\dbr)$, let $(\Xbul, \Ybul)$ be the unique representative such that $Y_0 = \U_{+}$ and $\pi(X_0) = \wo\borel_{+}$. Consider the map
\begin{equation}\label{eq:iso-glsbs-cgglss}
\iso: R(\dbr) \lr X(\dbr^{(-|+)}),\quad \G.(\Xbul', \Ybul') \stackrel{\iso}{\longmapsto} (\borel_0, \dots, \borel_{n+m}) \in X(\dbr^{(-|+)}),
\end{equation}
where the flags $\borel_i$, $i\in[0,n+m]$ are defined by
\[
\begin{array}{l} \borel_0 = \pi(Y_0) = \borel_{+}, \borel_1 = \pi(Y_{a_1}), \dots, \borel_{n} = \pi(Y_{a_n}), \\ \borel_{n+1} = \pi(X_{b_m-1}), \borel_{n+2} = \pi(X_{b_{m-1}-1}), \dots, \borel_{n+m} = \pi(X_0) = \wo \borel_+.
\end{array}\]
Then the map $\iso:R(\dbr) \lr X(\dbr^{(-|+)})$ is an isomorphism of algebraic varieties.
\end{proposition}
\begin{proof}
Consider the map 
\[\iso': R(\dbr) \to X_{\U}(\dbr^{(-|+)}) \qquad \iso: \G.(\Xbul', \Ybul') \mapsto (Y_0,Y_{a_1},\ldots,Y_{a_n},X_{b_m-1},X_{b_{m-1}-1},\ldots,X_0),\]
where $(\Xbul, \Ybul)$ is the unique representative of $\G.(\Xbul', \Ybul')$ with $Y_0=\U_+$ and $\pi(X_0)= \wo \B_+$, as in the statement.
The morphism $\iso$ is $\iso'$ composed with the projection $\pi:\G/\U\lr\G/\borel$ applied to each flag. Since \cite[Lemma 3.13]{CGGLSS} implies that such projection $\pi$ induces an isomorphism between $X_{\U}(\beta)$ and $X(\beta)$, as in \cref{eq:CGGLSS_weighted_dfn} above, it suffices to show that the morphism $\iso'$ is an isomorphism.

Note that $\iso'$ is well-defined, as the existence of the special representative $(\Xbul, \Ybul)$ of $\G.(\Xbul', \Ybul')$ follows from the condition that $Y_0' \Rwrel{\wo} X_0'$, and its uniqueness follows from the fact that $\G$ acts freely on such pairs $(Y_0', X_0')$. Once this representative is fixed, well-definedness of $\iso'$ follows from the definition of $\dbr^{(-|+)}$. Indeed, the
linear sequence of relative positions in the definition of $X_{\U}(\dbr^{(-|+)})$ coincides with the sequence of relative positions in \cref{eq:GLSBS-definition} describing $R(\dbr)$, where we read counterclockwise starting at $Y_0$ and ignore instances of $s^\pm_i=\mbox{id}$. The inverse map $(\iso')^{-1}$ is constructed by repeating weighted flags for every instance of $s^\pm_i=\mbox{id}$ and then taking the $\G$-orbit: it follows that $\iso'$ is an isomorphism.
\end{proof}



The isomorphism $\iso$ in \cref{lem:iso-glsbs-cgglss} is essentially given by reflecting \eqref{eq:GLSBS-definition} across the $x=y$ line and reading (the projections of) the flags clockwise starting at $Y_0=U_+$, contracting all arrows labeled with the identity in \eqref{eq:GLSBS-definition}.

\begin{example}\label{ex:running-iso}
    Consider the double braid word $\dbr = (-2, 1, 2, 1, -1, 1, 2)$ as in Example \ref{ex:running}. A point $(X_\bullet, Y_\bullet)\in R(\dbr)$, or more precisely the particular representative of a $G$-orbit specified in \cref{lem:iso-glsbs-cgglss}, is described by the following diagram:
    \begin{equation*}
		\begin{tikzcd}
			X_0& \arrow[l,"{\id}"'] X_1& \arrow[l,"{s_1}"']  X_2& \arrow[l,"{s_2}"']  X_3& \arrow[l,"{s_1}"']  X_4& \arrow[l,"{\id}"']  X_5& \arrow[l,"{s_1}"']  X_6& \arrow[l,"{s_2}"']   
			X_7 \\
			U_+=Y_0 \arrow[u, Rightarrow, "\wo"'] \arrow[r,"{s_{2^*}=s_1}"']& Y_1 \arrow[r,"{\id}"'] & Y_2 \arrow[r,"{\id}"'] & Y_3 \arrow[r,"{\id}"'] &Y_4 \arrow[r,"{s_{1^*}=s_2}"']& Y_5 \arrow[r,"{\id}"'] & Y_6 \arrow[r,"{\id}"']  &
			Y_7  \ar[equal]{u}
		\end{tikzcd}
	\end{equation*}
    Reflecting this diagram across the line $y=x$ gives 
\begin{equation*}
       \adjustbox{scale=0.8}{ \begin{tikzcd}[column sep=2ex,row sep=1.5ex]
			 &&&&&&& X_7 =  Y_7 \arrow[rd,"{s_2}"']&&&&&&& \\
          &&&&&&Y_6\arrow[ru,"{\id}"'] && X_6\arrow[rd,"{s_1}"'] &&&&&& \\
          &&&&&Y_5\arrow[ru,"{\id}"'] &&&& X_5\arrow[rd,"{\id}"'] &&&&& \\
          &&&&Y_4\arrow[ru,"{s_2}"'] &&&&& &X_4\arrow[rd,"{s_1}"'] &&&& \\
          &&&Y_3\arrow[ru,"{\id}"'] &&&&&&&& X_3\arrow[rd,"{s_2}"'] &&& \\
          &&Y_2\arrow[ru,"{\id}"'] &&&&&&&&&& X_2\arrow[rd,"{s_1}"'] && \\
          &Y_1\arrow[ru,"{\id}"'] &&&&&&&&&& &&X_1\arrow[rd,"{\id}"'] & \\
          U_+=Y_0\arrow[ru,"{s_1}"'] \arrow[rrrrrrrrrrrrrr, Rightarrow, "w_0"']&&&&&&&&&&&& &&X_0,
                   \end{tikzcd}}
\end{equation*}
then projecting each flag to $G/B$ gives

\begin{equation} \label{fig:triangle_with_flags}
\adjustbox{scale=0.8}{\begin{tikzcd}[column sep=2ex,row sep=1.5ex]
			 &&&&&&& \borel_2 \arrow[rd,"{s_2}"']&&&&&&& \\
          &&&&&&\borel_2\arrow[ru,"{\id}"'] && \borel_3\arrow[rd,"{s_1}"'] &&&&&& \\
          &&&&&\borel_2\arrow[ru,"{\id}"'] &&&& \borel_4\arrow[rd,"{\id}"'] &&&&& \\
          &&&&\borel_1\arrow[ru,"{s_2}"'] &&&&&&\borel_4\arrow[rd,"{s_1}"'] &&&& \\
          &&&\borel_1\arrow[ru,"{\id}"'] &&&&&&&& \borel_5\arrow[rd,"{s_2}"'] &&& \\
          &&\borel_1\arrow[ru,"{\id}"'] &&&&&&&&&& \borel_6\arrow[rd,"{s_1}"'] && \\
          &\borel_1\arrow[ru,"{\id}"'] &&&&&&&&&&&& \borel_7\arrow[rd,"{\id}"'] & \\
          \borel_+=\borel_0\arrow[ru,"{s_1}"'] \arrow[rrrrrrrrrrrrrr, Rightarrow, "w_0"']&&&&&&&&&&&&&& \borel_7=\wo\borel_+,
                   \end{tikzcd}}
\end{equation}
and finally contracting each arrow labeled by the identity gives
\begin{equation*}
        \begin{tikzcd}[column sep=2ex,row sep=1.5ex]
			 &&&&&&& \borel_2 \arrow[rd,"{s_2}"']&&&&&&& \\
          &&&&&& && \borel_3\arrow[rrdd,"{s_1}"'] &&&&&& \\
          &&&&& &&&&  &&&&& \\
          &&&&\borel_1\arrow[rrruuu,"{s_2}"'] &&&&&&\borel_4\arrow[rd,"{s_1}"'] &&&& \\
          &&&&&&&&&&& \borel_5\arrow[rd,"{s_2}"'] &&& \\
          && &&&&&&&&&& \borel_6\arrow[rrdd,"{s_1}"'] && \\
          & &&&&&&&&&&&& & \\
          \borel_+=\borel_0\arrow[rrrruuuu,"{s_1}"'] \arrow[rrrrrrrrrrrrrr, Rightarrow, "w_0"']&&&&&&&&&&&&&& \borel_7=\wo \borel_+.
                   \end{tikzcd}
\end{equation*}

The image of $(X_\bullet, Y_\bullet)$ under $\varphi$ is the point
\[\borel_+=\borel_0 \Rrel{s_1} \borel_1 \Rrel{s_2} \borel_2 \Rrel{s_2}\borel_3 \Rrel{s_1} \borel_4 \Rrel{s_1} \borel_4 \Rrel{s_2} \borel_6 \Rrel{s_1} \borel_7=\wo\borel\]
obtained by reading the diagram above clockwise from $\borel_+$. In symbols,   
    \begin{align*}
    \borel_0 = & \pi(Y_0) = \borel_{+}; & 
    \borel_1 = & \pi(Y_1) = \pi(Y_2) = \pi(Y_3) = \pi(Y_4);\\
    \borel_2 = & \pi(Y_5) = \pi(Y_6) = \pi(Y_7) = \pi(X_7);&
    \borel_3 = & \pi(X_6); \\
    \borel_4 = & \pi(X_5) = \pi(X_4); &
    \borel_5 = & \pi(X_3); \\
    \borel_6 = & \pi(X_2) ;& 
    \borel_7 = & \pi(X_1) = \pi(X_0)= \wo \borel_+.
    \end{align*}
    This illustrates the map $\iso$ for the running example.\qed
\end{example}

\begin{remark}\label{rmk:compatibility-iso-moves} The isomorphism $\iso$ is compatible with certain types of double braid moves:
\begin{enumerate}
\item[(i)] First, assume that $\dbr = \dbr_1 ij \dbr_2$ and $\dbr' = \dbr_1 ji \dbr_2$ where $i$ and $j$ have different signs. This is move (B1) in \cite[Section 4]{GLSB}. Then there is a natural isomorphism $\phi_{\dbr, \dbr'}: R(\dbr) \to R(\dbr')$. If $i$ is the $k$-th letter of $\dbr$, this isomorphism is given by $(\Xbul, \Ybul) \mapsto (X'_{\bullet}, Y'_{\bullet})$ where
\[
(X'_{p}, Y'_{p}) = \begin{cases} (X_p, Y_p) & \text{if} \; p \neq k, k+1 \\ (X_{k+1}, Y_{k+1}) & \text{if} \; p = k \\
(X_k, Y_k) & \text{if} \; p = k+1.\end{cases}
\]
Note that $\brWe = (\dbr')^{(-|+)}$. Then the following diagram commutes
\begin{equation}\label{eq:compatibility-iso-moves}
\begin{tikzcd}
    R(\dbr) \arrow[rr, "\phi_{\dbr, \dbr'}"] \arrow[dr, "\iso_{\dbr}"] && R(\dbr') \arrow[dl, "\iso_{\dbr'}"]\\
     &X(\brWe).& 
\end{tikzcd}
\end{equation}

\item[(ii)] Assume that $\dbr = \dbr_1i$ and $\dbr' = \dbr_1(-i)^{\ast}$.  This is move (B4) in \cite[Section 4]{GLSB}. Then there is a natural isomorphism $\phi_{\dbr, \dbr'}: R(\dbr) \to R(\dbr')$, described in \cite[Section 4.6]{GLSB}. As in (i) above, we have that $\brWe = (\dbr')^{(-|+)}$ and a diagram similar to \eqref{eq:compatibility-iso-moves} commutes.\qed
\end{enumerate}
\end{remark}


\subsection{The commutative algebras}\label{ssec:coordinate_algebras} A central goal in both \cite{CGGLSS} and \cite{GLSBS,GLSB} is the construction of cluster algebras. Specifically, the two commutative algebras that are shown to admit a cluster structure are the rings of regular functions $\Z[X(\br)]$ and $\Z[R(\dbr)]$.

By construction, $X(\beta)$ is an affine variety over $\Z$ and the weave methods are defined and applied over $\Z$. 
Such integrality in the case of $R(\dbr)$ can also be deduced from \cite{GLSBS,GLSB}, even if less explicitly stated there. For simplicity, we focus on the base change to the ground field $k=\C$ and work with the commutative algebras $\C[X(\br)]$ and $\C[R(\dbr)]$.


{\it $(i)$ Braid varieties}. In the case of $X(\br)$, \cref{def:braid-variety} gives a presentation of equations and inequalities cutting out $X(\br)$ as a quasi-affine scheme in the projective variety $(\G/\borel)^{l + 1}$. In fact, \cite[Proposition 3.6]{CGGLSS} implies that the space of sequences of flags $\borel_\bullet = (\borel_0, \dots, \borel_{l})$ satisfying the relative position conditions $\borel_+ = \borel_0 \Rrel{s_{i_1}}  \borel_1  \Rrel{s_{i_2}}   \cdots  \Rrel{s_{i_l}}  \borel_l$ is isomorphic to the affine space $\C^l$, and thus the braid variety $X(\beta)$ can be described as the closed subvariety of $\C^l$ cut out by the condition $\borel_l = \wo\borel_{+}$. With this approach, the $\C$-algebra $\C[X(\br)]$ can be presented rather explicitly as a quotient of $\C[z_1,\ldots,z_l]$ by using braid matrices, as explained in \cite[Corollary 3.7]{CGGLSS} and as we now recall.

Fix a pinning $(\H,\borel_+,\borel_-,x_i,y_i;i\in I)$ of $\G$, cf. \cref{ssec:pinnings}. For $i\in I$ and $z \in \C$ we set
\begin{equation*}%
 \BZ_i(z):=\phi_i \begin{pmatrix}
		z & -1\\
		1 & 0
	\end{pmatrix}=x_i(z) \ds_i.
\end{equation*}

Then \cite[Corollary 3.7]{CGGLSS} implies that a set of equations for $X(\beta)$ in affine space is:
\begin{equation}\label{eq:braid_var_eq}
X(\beta) \cong \{(z_1, \dots, z_l) \in \C^{l} \mid \wo B_{i_{1}}(z_1)\cdots B_{i_{l}}(z_l) \in \borel\},
\end{equation}
where $\wo$ is understood as a lift of the permutation $\wo$.\\

\begin{example}\label{ex:example_zcoordinates}
Consider the braid word $\beta=1221121$, as in \cref{ex:running}, and the following elements $B_1(z),B_2(z)$ of $\G=\SL_3$:
$$B_1(z):=\begin{pmatrix}
z & -1 & 0\\
1 & 0 & 0\\
0 & 0 & 1\\
\end{pmatrix},\quad
B_2(z):=\begin{pmatrix}
1 & 0 & 0\\
0 & z & -1\\
0 & 1 & 0\\
\end{pmatrix}.$$
Denote $B_{\textcolor{blue}{\beta}}(z_1,\ldots,z_7):=B_{\textcolor{blue}{1}}(z_1)B_{\textcolor{blue}{2}}(z_2)B_{\textcolor{blue}{2}}(z_3)B_{\textcolor{blue}{1}}(z_4)B_{\textcolor{blue}{1}}(z_5)B_{\textcolor{blue}{2}}(z_6)B_{\textcolor{blue}{1}}(z_7)$, where we have highlighted in \textcolor{blue}{blue} the dependence on the braid word $\beta$. These are the braid matrices introduced in \cite[Section 3.4]{CGGLSS} to coordinatize $X(\beta)$. Then $X(\beta)$ is isomorphic to the closed subvariety of $\C^7=\Spec\C[z_1,\ldots,z_7]$ cut out 
by the condition
\begin{equation}\label{eq:example_zcoordinates}
\wo B_\beta(z_1,\ldots,z_7)\in \borel_+,
\end{equation}
where $\wo$ is understood as a permutation matrix lifting $\wo\in S_n$. The three diagonal conditions in \cref{eq:example_zcoordinates} are non-equalities, cutting out an open subset, whereas the vanishing of the entries $(2,1),(3,1)$ and $(3,2)$ of $\wo B_\beta(z_1,\ldots,z_7)$ impose three independent conditions. Altogether, these lead to a dimension count of $\dim_\C X(\beta)=7-3=4$, which indeed coincides with $\ell(\beta)-\ell(\delta(\beta))=7-3=4$.\hfill$\Box$
\end{example}

{\it $(ii)$ Double braid varieties}.  In the case of $R(\dbr)$, the quotient by the $\G$-action requires a word. Consider the variety $\bigRo(\dbr)$ given by the subset of points in $(\G/\U_+)^{[0,l]} \times (\G/\U_+)^{[0,l]}$ satisfying \cref{eq:GLSBS-definition}, so that $R(\dbr)$ is the quotient of $\bigRo(\dbr)$ by the diagonal action of $\G$. Since this $\G$-action is free and $\G$ is reductive, Geometric Invariant Theory, cf.~e.g. \cite{GIT}, implies that
\begin{equation}\label{eq:double_braid_Ginvariant_description}
\C[R(\dbr)] \cong \Gamma(\bigRo(\dbr), \mathcal{O}_{\bigRo(\dbr)})^{\G},
\end{equation}

i.e.~the ring of regular functions on $R(\dbr)$ is isomorphic to the ring of $\G$-invariant global sections of the structure sheaf of $\bigRo(\dbr)$. For an alternative description, let $\bigRofix(\dbr) \subseteq \bigRo(\dbr)$ be the closed subvariety defined by the condition $Y_0 = \U_{+}$.

Then the group $\U_{+}$, which is the stabilizer of $Y_0 = \U_{+}$, acts freely on $\bigRofix(\dbr)$ and there is a natural bijection between the $\G$-orbits in $\bigRo(\dbr)$ and the $\U_{+}$-orbits in $\bigRofix(\dbr)$, so that $R(\dbr)\cong \bigRo(\dbr)/\G \cong \bigRofix(\dbr)/\U_{+}$. It follows that the pull-back $i^*: \Gamma(\bigRo(\dbr), \mathcal{O}_{\bigRo(\dbr)}) \to \C[\bigRofix(\dbr)]$, induced by the inclusion $i:\bigRofix(\dbr)\to\bigRo(\dbr)$,  descends to an isomorphism $[i^*]:\Gamma(\bigRo(\dbr), \mathcal{O}_{\bigRo(\dbr)})^{\G} \xrightarrow{\cong} \C[\bigRofix(\dbr)]^{\U_{+}}$. Therefore, the commutative algebra $\C[R(\dbr)]$ is isomorphic to
\begin{equation}\label{eq:coord_ring_doublebraidvariety}
 \C[R(\dbr)] \cong \C[\bigRofix(\dbr)]^{\U_{+}}.
\end{equation}

\subsection{Parametrizations}\label{sec:parametrization}
Let us provide explicit coordinates on the varieties $R(\dbr)$ and $X(\brWe)$, as well as on the auxiliary variety $\bigRofix(\dbr)$ appearing in \cref{ssec:coordinate_algebras}. 

\subsubsection{Parametrizing $R(\dbr)$}\label{ssec:parameters-double-braid-variety} Let $\dbr = i_1\cdots i_{n+m}$ be a double braid word. We parameterize $\Rbr3D(\dbr)$ as follows. Using the $\G$-action, we assume that $Y_0= \U_+$, i.e. we work on the space $\bigRofix(\br)$ defined in \cref{ssec:coordinate_algebras} above. We set $g'_0:=\id\in \G$ and define a sequence of elements $g'_1,\dots,g'_\nm\in \G$ and the associated weighted flags $Y_1=g_1'\U_+,\dots,Y_\nm=g_\nm'\U_+$ inductively by 
\begin{equation*}
  g'_{\c} := 
  \begin{cases}
    g'_{\c-1}, &\text{if $i_\c\in I$,}\\
    g'_{\c-1} \Bz_{|i_\c|^\ast}(\z_\c), &\text{if $i_\c\in -I$,}
  \end{cases}
  \quad\text{for $\c=1,2,\dots,\nm$.}
\end{equation*}
Note that the weighted flags $Y_0, \dots, Y_{n+m}$ uniquely determine the scalars $z_c, c \in J_{-}$ and vice versa. Similarly, we define a sequence of elements $g_{m+n}, \dots, g_0 \in \G$ and the associated weighted flags $X_{m+n} = g_{n+m}\U_{+}, \dots, X_0 = g_0\U_+$ by
\[
g_{m+n} = g'_{m+n} \qquad \text{and} \qquad g_{\c-1} = \begin{cases} g_\c B_{i_\c}(z_\c), & \text{if} \; i_\c \in I, \\ g_\c, & \text{if} \; i_\c \in -I. \end{cases} 
\]
The weighted flags $X_{m+n}, \dots, X_0$ uniquely determine the scalars $z_\c$, $\c \in J_{+}$, and vice versa. Thus, if we let $\bigRfix(\dbr)$ be the space defined as $\bigRofix(\dbr)$ but without the condition that $Y_0 \Rwrel{\wo} X_0$, we obtain that $\bigRfix(\dbr)$ is an affine variety with coordinate algebra
\begin{equation}\label{eq:polynomial-algebra}
\C[\bigRfix(\dbr)] \cong \C[z_1, \dots, z_{n+m}],
\end{equation}
that is, $z_1, \dots, z_{n+m}$ give coordinates on the space $\bigRfix(\dbr)$. The condition $Y_0 \Rwrel{\wo} X_0$ gives an open condition on $\bigRfix(\dbr)$, so we obtain that
\[
\C[\bigRofix(\dbr)] \; \text{is a localization of} \; \C[\bigRfix(\dbr)] \cong \C[z_1, \dots, z_{n+m}]. 
\]
Finally, $\C[R(\dbr)] \cong \C[\bigRofix(\dbr)]^{\U_{+}}$, so $\C[R(\dbr)]$ is a subalgebra of a localization of $\C[z_1, \dots, z_{n+m}]$. 

    \subsubsection{Parametrizing $X(\brWe)$} A parametrization of $X(\brWe)$ has been obtained in \cref{eq:braid_var_eq}, but we recall it for the sake of comparing to the parametrization of $R(\dbr)$ obtained in \cref{ssec:parameters-double-braid-variety}. Let us write $\brWe = j_1\cdots j_{m+n}$. We define elements $f_0 = \id \in \G$, $f_1, \dots, f_{n+m}$ and the associated flags $\borel_0 = \borel$, $\borel_1 = f_1\borel, \dots, \borel_{n+m} = f_{n+m}\borel$ as follows:
    \[
        f_{\c} = f_{\c-1}B_{j_\c}(z'_c). 
    \]
    The variety of collections of flags satisfying the relative position conditions
    \[
    \borel \Rrel{s_{j_1}} \borel_1 \Rrel{s_{j_2}} \cdots \Rrel{s_{j_{n+m}}} \borel_{n+m}
    \]
    is then isomorphic to $\C[z'_1, \dots, z'_{n+m}]$, and the condition $\borel_{n+m} = \wo\borel$ defining $X(\brWe)$ is a closed condition. Thus, we obtain a natural surjection
    \[
    \C[z'_1, \dots, z'_{n+m}] \twoheadrightarrow \C[X(\brWe)].
    \]

\subsubsection{Parametrizing the isomorphism $\iso: \Rbr3D(\dbr) \lr \RbrWe(\brWe)$} By \cref{lem:iso-glsbs-cgglss}, we have an isomorphism $\iso: \Rbr3D(\dbr) \xrightarrow{\cong} \RbrWe(\bmp)$. We now present the induced isomorphism $\iso^{\ast}: \C[X(\brWe)] \lr \C[R(\dbr)]$ in terms of the coordinates obtained above. 

Recall that we set $J_{-} = \{a_1 < \cdots < a_n\}$ for the set of indices $j$ such that $i_j \in -I$, and similarly $J_{+} = \{b_1 < \cdots < b_m\}$. We define a bijection $\phi: [\nm] \lr [\nm]$ as follows:
\[
\phi(d) = a_d\quad\mbox{for}\quad d\in[1,n], \qquad \phi(n+d') = b_{m+d'-1}\quad\mbox{for}\quad d'\in[1,m].
\]

We can use the bijection $\phi$ to present the isomorphism $\iso^{\ast}: \C[(X(\brWe)] \lr \C[R(\dbr)]$ using the coordinates $\zp_1, \dots, \zp_\nm$ and $z_1, \dots, z_\nm$, as follows:

\begin{proposition}\label{prop:iso-polynomial-algebras}
Consider the isomorphism
$$\iso^{\ast}: \C[\zp_1, \dots, \zp_\nm] \lr \C[z_1, \dots, z_\nm],\quad \iso^{\ast}(\zp_\c) = z_{\phi^{-1}(c)}.$$
Then, $\iso^{\ast}$ induces an homonymous isomorphism $\iso^{\ast}: \C[X(\brWe)] \lr \C[R(\dbr)]$ that is dual to the isomorphism $\iso: R(\dbr) \lr X(\brWe)$. In precise terms, the following diagram commutes:
\begin{center}
\begin{tikzcd}
    \C[z'_1, \dots, z'_\nm] \arrow[r, "\iso^{\ast}"] \arrow[d, twoheadrightarrow] & \C[z_1, \dots, z_\nm] \arrow[r, hook]&  \C[\bigRofix(\dbr)] \\
    \C[X(\brWe)] \arrow[r, "\iso^{\ast}"] & \C[R(\dbr)] \arrow[r, "\cong"] & \C[\bigRofix(\dbr)]^{\U_{+}} \arrow[u, hook].
\end{tikzcd}
\end{center}
\qed
\end{proposition}

\cref{prop:iso-polynomial-algebras} follows directly from the construction of the isomorphism $\iso$. Note that a tuple $(\Xbul,\Ybul)\in R(\dbr)$ satisfying~\eqref{eq:GLSBS-definition} is mapped under $\iso$ to the tuple of weighted flags $\Fbul\in X(\brWe)$ given by
\begin{equation}\label{eq:parametrized-flags-CGGLSS}
  F_\cp = \BZ_{\ip_1}(\zp_1) \BZ_{\ip_2}(\zp_2)\cdots \BZ_{\ip_\cp}(\zp_\cp) U_+\quad \mbox{ for each }\cp\in[0,\nm].
\end{equation}



\begin{example}\label{ex:running_variables}
We illustrate the isomorphism in \cref{prop:iso-polynomial-algebras} in our running example, continuing \cref{ex:running-iso}. 
 We have
\begin{align*}
Y_0 = F_0 &= \U_+;\\
Y_1= Y_2 = Y_3 = Y_4 = F_1 &= \BZ_1(\z_1) \U_+; \\
Y_5 = Y_6 = Y_7 = X_7 = F_2 &= \BZ_1(\z_1) \BZ_2(\z_5) \U_+;\\
X_6 = F_3 &= \BZ_1(\z_1) \BZ_2(\z_5) \BZ_2(\z_7) \U_+;\\ 
X_5 = X_4 = F_4 &= \BZ_1(\z_1) \BZ_2(\z_5) \BZ_2(\z_7) \BZ_1(\z_6) \U_+;\\
X_3 = F_5 &= \BZ_1(\z_1) \BZ_2(\z_5) \BZ_2(\z_7) \BZ_1(\z_6) \BZ_1(\z_4) \U_+;\\
X_2 = F_6 &= \BZ_1(\z_1) \BZ_2(\z_5) \BZ_2(\z_7) \BZ_1(\z_6) \BZ_1(\z_4) \BZ_2(\z_3) \U_+;\\
X_1 = X_0 = F_7 &= \BZ_1(\z_1) \BZ_2(\z_5) \BZ_2(\z_7) \BZ_1(\z_6) \BZ_1(\z_4) \BZ_2(\z_3) \BZ_1(\z_2) \U_+.
\end{align*}
Thus, the isomorphism from \cref{prop:iso-polynomial-algebras} is given by $\iso^{\ast}(\zp_1) = z_1, \iso^{\ast}(\zp_2) = z_5$, $\iso^{\ast}(\zp_3) = z_7$, $\iso^{\ast}(\zp_4) = z_6$, $\iso^{\ast}(\zp_5) = z_4$, $\iso^{\ast}(\zp_6) = z_3$ and $\iso^{\ast}(\zp_7) = z_2$.\qed
\end{example}



\def\nm{{n+m}}

\section{Comparison of tori}\label{sec:weaves-tori}

\cref{sec:braid_varieties} introduced $X(\beta)$ and $R(\dbr)$, from \cite{CGGLSS} and \cite{GLSBS,GLSB} respectively, and \cref{lem:iso-glsbs-cgglss} established the isomorphism $ \iso: R(\dbr) \xrightarrow{\sim} X(\br)$ when $\beta=\dbr^{(-|+)}$, cf.~\cref{eq:iso-glsbs-cgglss}. The goal of the present section is to compare the open torus charts in $X(\beta)$ and $R(\dbr)$, respectively described in \cite{CGGLSS} and \cite{GLSBS,GLSB}, that are part of the cluster seeds of the respective cluster structures on $\C[X(\beta)]$ and $\C[R(\dbr)]$.

There is an interesting conceptual difference in the way such cluster tori in $X(\beta)$ and $R(\dbr)$ are built. It is helpful to understand this crucial point when implementing the comparison:

\begin{enumerate}
    \item In \cite{CGGLSS}, given the braid word $\beta$, one chooses an additional piece of combinatorial data to name a cluster torus: a Demazure weave $\ww:\beta\to\delta(\br)$, cf.~\cite[Section 2]{CasalsZaslow} and \cite[Section 4]{CGGLSS}. There are many such Demazure weaves for $\beta$, even up to weave equivalence. Once a specific Demazure weave $\ww:\beta\to\delta(\br)$ is chosen, it names an open torus chart $T_\ww\sse X(\beta)$. Non-equivalent weaves $\ww,\ww'$ give different tori $T_\ww,T_{\ww'}\sse X(\beta)$ in the {\it same} braid variety $X(\beta)$.\\

    \item In \cite{GLSBS,GLSB}, given the double braid word $\dbr$ there is {\it no} additional choice of combinatorial data to name a torus. Instead, a unique torus $T_\dbr\sse R(\dbr)$ is named from the double braid word $\dbr$: such torus is referred to as the Deodhar torus, cf.~\cite[Section 7.1]{GLSBS} and \cite[Section 2.3]{GLSB}. It is possible to describe other tori in $R(\dbr)$ by changing the double braid word from $\dbr$ to $\dbr'$, via a double braid move, and using the induced isomorphism $m:R(\dbr)\to R(\dbr')$ to pull
    $T_{\dbr'}\sse R(\dbr')$ back to $m^*(T_{\dbr'})\sse R(\dbr)$. In many cases, $T_\dbr,m^*(T_{\dbr'})\sse R(\dbr)$ are two different tori in the {\it same} double braid variety $R(\dbr)$.
\end{enumerate}

The main contribution of this section, summarized in \cref{lem:weave-tori-vs-deodhar-tori}, is the description of a Demazure weave $\DW(\dbr):\dbr^{(-|+)}\to \wo$ such that the image $\iso(T_\dbr)$ of the Deodhar torus $T_{\dbr} \sse R(\dbr)$ under the isomorphism $\iso$ coincides with the torus $T_{\DW(\dbr)} \sse X(\dbr^{(-|+)})$.

\subsection{Demazure weaves} Weaves were introduced in \cite{CasalsZaslow} and further studied in \cite[Section 4]{CGGS} and \cite[Section 4]{CGGLSS}. They are the combinatorial object that drives the construction of cluster structures on braid varieties in \cite{CGGLSS}. Indeed, given a Demazure weave $\ww:\beta\to\delta(\beta)$, one names the cluster torus $T_\ww\sse X(\beta)$, the quiver $Q_\ww$ and the cluster variables $\xWe=\{A_i(\ww)\}$ all in terms of the combinatorics of $\ww$. We now briefly summarize their definition and then discuss double inductive weaves, leading to the construction of the weave $\DW(\dbr)$, whose associated torus matches the corresponding Deodhar torus.\\

\begin{definition}[Graph of braid words]\label{def:graph_braid_words}
The directed graph $\Graph=\Graph(\G)$ is the directed graph whose vertex set consists of all the braid words in the positive Artin generators of the braid group $\Br(\G)$, words being of arbitrary finite length. The set of arrows of $\Gamma$ is described as follows, where $\beta,\beta'$ denote two vertices of $\Gamma$:
\begin{itemize}
\item[(a)] There is an arrow $\beta\to\beta'$ if
$$\beta = \dots \underbrace{(iji \dots)}_{m_{ij} \; \text{letters}}\dots \, \to \beta' = \dots\underbrace{(jij\dots)}_{m_{ij} \; \text{letters}} \dots,$$
where $i,j$ are adjacent vertices of the Dynkin diagram of $\G$.\\
\item[(b)] There is an arrow $\beta\to\beta'$ if
$$\beta = \dots ik \dots \, \to \beta' = \dots ki \dots,$$
where $i,k$ are not adjacent in the Dynkin diagram of $\G$.\\
\item[(c)] There is an arrow $\beta\to\beta'$ if
$$\beta = \dots ii \dots \, \to \beta' = \dots i \dots,$$
where $i$ is an arbitrary vertex of the Dynkin diagram of $\G$.\hfill$\Box$
\end{itemize}
\end{definition}

In \cref{def:graph_braid_words} the notation $\beta=\dots(\upsilon)\dots$ denotes that $\beta$ is a braid word which contains $\upsilon$ as a (consecutive) braid subword. Note that in cases (a),(b) of \cref{def:graph_braid_words}, $\Graph$ has arrows $\beta \to \beta'$ and $\beta' \to \beta$, while in (c), $\Graph$ has an arrow only in one direction $\beta \to \beta'$.

\begin{definition}[Demazure weave]\label{def:weave}
Let $\beta \in \Graph$ be a braid word. By definition, a \emph{Demazure weave} $\ww:\beta\to\delta(\beta)$ is a directed path on the graph $\Graph$ from $\beta$ to a reduced word for its Demazure product $\delta(\beta)$.\hfill$\Box$
\end{definition}

The Demazure product gives a bijection between the connected components of $\Graph$ and the elements of the Weyl group $W(\G)$. As explained earlier, we focus on the case $\delta(\beta)=\wo$ in this article, so we can consider vertices $\beta$ in the connected component $\Graph_{\circ}\sse\Graph$ consisting of elements with Demazure product $\wo$, and thus weaves $\ww:\beta\to\wo$.

Demazure weaves can be drawn diagrammatically in the plane, as follows. First, the braid word $\beta \in I^l$ is drawn as a sequence of $l$ vertical strings colored by the elements of $I$: these are drawn so that they spell $\beta$ left-to-right. Second, each arrow in the path is represented by the following 
intertwinings of the vertical strings, going from top to bottom:

\begin{center}
\begin{tikzcd}
\node at (-0.5, 3) {\text{(a)}};
\node at (-0.8, 2.5) {m_{ij} = 3};
\draw[dashed] (0,0) -- (3,0) -- (3,3) -- (0,3)-- cycle;
\draw[red] (0.5, 3) to[out=270, in=135] (1.5, 1.5);
\node at (0.3, 2.5) {{i}};
\draw[blue] (1.5, 3) to (1.5, 1.5);
\node at (1.3, 2.5) {{j}};
\draw[red] (2.5, 3) to [out=270, in=45] (1.5, 1.5);
\node at (2.3, 2.5) {{i}};
\draw[blue] (1.5, 1.5) to [out=225, in=90] (0.5, 0);
\draw[red] (1.5, 1.5) to (1.5, 0);
\draw[blue] (1.5, 1.5) to[out=315, in=90] (2.5, 0);

\node at (4.5, 3) {\text{(a)}};
\node at (4.2, 2.5) {m_{ij} = 4};
\draw[dashed] (5,3)--(8,3)--(8,0)--(5,0)--cycle;
\draw[red] (5.3, 3) to[out=270, in=90] (7.7, 0);
\draw[blue] (6.1, 3) to[out=270, in=90] (6.9, 0);
\draw[red] (6.9, 3) to[out=270, in=90] (6.1, 0);
\draw[blue] (7.7, 3) to[out=270, in=90] (5.3, 0);

\node at (9.5, 3) {\text{(a)}};
\node at (9.2, 2.5) {m_{ij} = 6};
\draw[dashed] (10,3)--(13,3)--(13,0)--(10,0)--cycle;
\draw[red] (10.25, 3) to[out=270,in=90] (12.75, 0);
\draw[blue] (10.75, 3) to[out=270,in=90] (12.25,0);
\draw[red] (11.25, 3) to[out=270,in=90] (11.75, 0);
\draw[blue] (11.75, 3) to[out=270,in=90] (11.25, 0);
\draw[red] (12.25, 3) to[out=270,in=90] (10.75, 0);
\draw[blue] (12.75, 3) to[out=270,in=90] (10.25, 0);

\node at (-0.5, -1) {\text{(b)}};
\draw[dashed] (0,-1)--(3,-1)--(3,-4)--(0,-4)--cycle;
\draw[red] (0.5, -1) to[out=270, in=90] (2.5, -4);
\node at (0.3, -1.5) {{i}};
\draw[teal] (2.5, -1) to[out=270, in=90] (0.5, -4);
\node at (2.2, -1.5) {{k}};

\node at (9.5, -1) {\text{(c)}};
\draw[dashed] (10,-1)--(13,-1)--(13,-4)--(10,-4)--cycle;
\draw[red] (10.5, -1) to[out=270, in=135] (11.5,-2.5);
\node at (10.3, -1.5) {{i}};
\draw[red] (12.5, -1) to[out=270, in=45] (11.5, -2.5);
\node at (12.3, -1.5) {{i}};
\draw[red] (11.5, -2.5) to (11.5, -4);
\end{tikzcd}
\end{center}

Each of these five pictures describes a move $\beta\to\beta'$. The top horizontal slice in any of these five pictures above corresponds to the braid word $\beta$, which is spelled left-to-right by the vertical strings above, and the bottom horizontal slice corresponds to the braid word $\beta'$, which is spelled left-to-right by the vertical strings below. From this diagrammatic viewpoint, a Demazure weave $\ww:\beta\to\delta(\beta)$ can be also described as a planar graph whose edges are colored by the elements of $I$, and with vertices of different types: $(2m_{ij})$-valent vertices as in (a), tetravalent vertices as in (b), trivalent vertices as in (c), and univalent vertices that are located at the top and bottom of the weave, spelling $\beta$ at the top and a reduced word for $\delta(\beta)$ at the bottom.\\

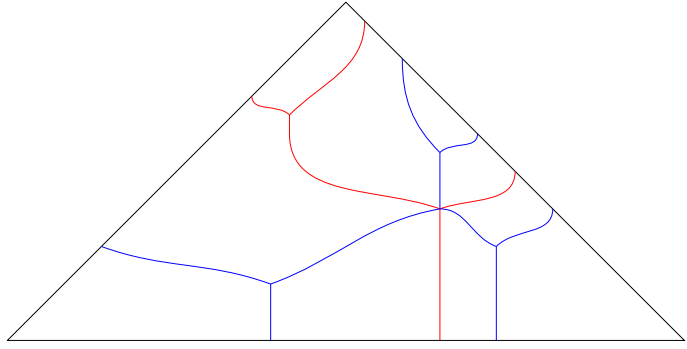
\begin{figure}[h!]
    \centering
    \adjustbox{scale=0.5}{
   \begin{tikzcd}
\draw (-1,-1) to (17,-1) to (8,8) to (-1,-1);

\draw[color=red] (8.5, 7.5) to[out=270, in=45] (6.5, 5);
\draw[color=blue] (9.5, 6.5) to[out=270,in=135] (10.5, 4);
\draw[color=red] (5.5,5.5) to[out=270,in=135] (6.5, 5) to (6.5, 4.5) to [out=270,in=160] (10.5, 2.5);
\draw[color=blue] (11.5, 4.5) to[out=270,in=45] (10.5, 4) to (10.5, 2.5);
\draw[color=red] (12.5, 3.5) to[out=270, in=20] (10.5, 2.5);

\draw[color=red] (10.5, 2.5) to (10.5,-1);

\draw[color=blue] (10.5, 2.5) to[out=0,in=160] (12, 1.5) to (12, -1);

\draw[color=blue] (13.5, 2.5) to[out=270, in=45] (12, 1.5);

\draw[color=blue] (1.5, 1.5) to[out=340, in=160] (6, 0.5);

\draw[color=blue] (10.5, 2.5) to[out=190,in=20] (6, 0.5) to (6,-1);
\end{tikzcd}
}
    \caption{A Demazure weave $\ww:1221121\to121$ drawn in an upward pointing triangle $\triangle$. In this figure, blue is color 1 and red is color 2. This weave $\ww$ is the double inductive weave associated to the double string $(2R, 1R, 1^{*}L, 1R, 2R, 1R, 2^{*}L)$ and double braid word $(-2,1,2,1,-1,1,2)$ from Example \ref{ex:from-double-braid-to-double-string-running-example}. To recover the double string from the weave, scan top to bottom and record $iR$ (or $i^*L$) when you see a weave line of color $i$ on the right (or left) side of $\triangle$. To recover the double braid word, scan bottom to top and record $i$ (or $-i^*$) when you see a weave line of color $i$ on the right (or left) side of $\triangle$.}
    \label{fig:ind-weave-triangle}
\end{figure}

From now onward, a weave $\ww$ will refer to such a planar edge-colored graph. An edge of this colored planar graph is said to be an \emph{edge} or \emph{weave line}, and \emph{top edge} (resp.~\emph{bottom edge}) is edge adjacent to a univalent vertex on the top (resp.~ bottom) of $\weave$. We denote by $E(\weave)$ the set of edges of a weave $\weave$. We often draw $\weave$ inside an upward pointing triangle $\triangle$, as in \cref{fig:ind-weave-triangle}, with one side parallel to the $x$-axis. The \emph{bottom boundary} of this triangle is the horizontal base side, and the \emph{top boundary} is the union of the two other sides, adjacent to the topmost vertex of the triangle. We place the univalent vertices spelling $\beta$ on the top boundary of $\triangle$ and those spelling $\delta(\beta)$ on the bottom boundary. Figure \ref{fig:ind-weave-triangle} depicts a Demazure weave $\ww:\beta\to\wo$ for the braid word $\beta = 1221121$, where the word for its Demazure product is $\delta(\beta)=\wo=121$.

\begin{remark}[Weave equivalences] \label{rmk:equivalence_weave}There are certain local moves between weaves which are said to be equivalences. Though we do not give the precise definition here, we refer to \cite[Section 4.1]{CasalsZaslow}, \cite[Section 4.2]{CGGS} and \cite[Section 4.2 and Section 6.3]{CGGLSS} for the equivalence moves and necessary definitions. The discussions in this article related to weaves are all well-defined for an equivalence class of weaves, and not just a weave itself.\hfill$\Box$
\end{remark}

\subsection{Weave tori}\label{ssec:weave_tori} We now explain how each Demazure weave $\ww:\beta\to\delta(\beta)$ specifies an open algebraic torus $T_\ww\sse X(\beta)$. We will use the following concept.

\begin{definition}\label{def:weave-flag-labeling} Let $\beta$ be a braid word and $\ww:\beta\to\wo$ a Demazure weave, drawn in a triangle $\ww\sse\triangle$. A \emph{flag labeling} of $\ww$ is an assignment of a flag $\borel_C$ to each connected component $C\sse\triangle \setminus \weave$ of the complement of $\ww$ such that:
\begin{enumerate}
\item If $C$ and $D$ are separated by an edge of color $i$, then $\borel_{C} \Rrel{s_i} \borel_{D}$.
\item $\borel_{C_{+}} = \borel_{+}$ and $\borel_{C_{-}} = \wo \borel_{+}$,
\end{enumerate}
where $C_{+}$ (resp.~$C_{-}$) is the only connected component of $\triangle \setminus \weave$ containing the left (resp.~right) vertex of the bottom boundary of the triangle $\triangle$.\hfill$\Box$
\end{definition}

By construction, given a flag labeling of $\weave$, the tuple consisting of flags at the top regions (those that intersect the top boundary of $\triangle$) determines a point of the braid variety $X(\beta)$. This tuple of flags is said to be the \emph{top tuple} of the flag labeling. By \cite[Lemma 3.1]{CGGLSS}, the top tuple of a flag labeling uniquely determines all other flags in that flag labeling. That is, giving the top tuple of a flag labeling is the same as specifying a point in $X(\beta)$, and such flags in the top tuple uniquely propagate down to a flag labeling for the weave $\ww:\beta\to\delta(\beta)$. Given a weave $\ww:\beta\to\delta(\beta)$, not all points of $X(\beta)$ are the top tuple of a flag labeling for $\ww$. In fact, such locus of points in $X(\beta)$ is the required weave torus:

\begin{definition}\label{lem:weave-torus}
Let $\weave:\beta\to\delta(\beta)$ be a Demazure weave. By definition, the \emph{weave torus} $T_\ww\sse X(\beta)$ associated to $\ww$ is
$$T_\ww:=\{(\borel_0, \dots, \borel_l) \in X(\beta)\mbox{ s.t.~$(\borel_0, \dots, \borel_l)$ is the top tuple of a flag labeling of }\weave\},$$
i.e.~the locus of points in $X(\beta)$ that appear as top tuples for some flag labeling of $\ww$.\hfill$\Box$
\end{definition}

It is proven in \cite[Lemma 4.1]{CGGLSS} that $T_\ww$ is indeed isomorphic to an algebraic torus and it is open in $X(\beta)$, thus the nomenclature in \cref{lem:weave-torus}. As an instance of \cref{rmk:equivalence_weave}, the results of \cite[Sections 5\&6]{CGGS} imply that the weave tori $T_{\weave},T_{\weave'}\sse X(\beta)$ associated to two equivalent weaves $\weave, \weave':\beta\to\delta(\beta)$ must coincide $T_{\weave}=T_{\weave'}$ as subvarieties of $X(\beta)$.

\begin{example}\label{ex:running_example_weave_torus}
Consider the braid word $\beta=1221121$, as in \cref{ex:running}, and the Demazure weave $\ww:\beta\to\wo$ depicted in \cref{fig:ind-weave-triangle}. By using the $z$-coordinates in \cref{ex:example_zcoordinates}, we expressed $X(\beta)\sse\C^7$ as a closed subvariety of $\C^7$. In these coordinates, the open torus $T_\ww\sse X(\beta)$ corresponding to $\ww$ is
$$T_\ww=\{(z_1,\ldots,z_7)\in X(\beta): z_3\neq0,z_5\neq0,z_3z_5z_7-z_3z_6-1\neq0,z_3z_6-z_5z_4+1\neq0\}.$$
The general method to obtain such equations is explained in \cite[Section 5]{CGGLSS}. For instance, given the top tuple $(B_0,\ldots,B_7)\in X(\beta)$ of any flag labeling of $\ww$, which is equivalent to giving $(z_1,\ldots,z_7)\in T_\ww \sse X(\beta)$, 
the weave $\ww$ is such that any flag labeling must satisfy:
\begin{enumerate}
    \item $B_1\stackrel{s_2}{\to}B_3$, which translates to the condition $z_3\neq0$,
    \item $B_3\stackrel{s_1}{\to}B_5$, which translates to the condition $z_5\neq0$.
\end{enumerate}
There are two more transversality conditions, expressed by $z_3z_5z_7-z_3z_6-1\neq0$ and $z_3z_6-z_5z_4+1\neq0$, which translate to $B_0$ and $B_7$ being $s_1$-transverse to $F_1$ and $F_2$ respectively, where $F_1,F_2\in \G/\borel$ are certain flags obtained from the top tuple $(B_0,\ldots,B_7)$. Specifically, $F_1,F_2$ are the two flags assigned to the two connected components of $\triangle \setminus \weave$ which are not adjacent to the top boundary of $\triangle$, cf.~\cref{fig:ind-weave-triangle-B-labels} below. \hfill$\Box$
\end{example}

\subsection{Double inductive weaves}\label{ssec:double_inductive_weaves} Given a double braid word $\dbr$, 
the Demazure weave $\ww:\brWe\lr\wo$ whose weave torus $T_\ww$ coincides with $\iso(T_\dbr\sse) R(\dbr)$ is a rather particular type of Demazure weave, called a double inductive weave. These were introduced in \cite[Section 6.4]{CGGLSS}. We recall double inductive weaves here and establish \cref{lem:description-double-inductive-weave}, which is the key result to compare the Deodhar torus $T_\dbr$ in $R(\dbr)$ to a weave torus in $X(\beta)$.

\subsubsection{Double strings} Double inductive weaves $\ww:\br\lr\wo$ are a type of Demazure weaves determined by the following piece of combinatorial data, cf.~\cite[Section 6.4]{CGGLSS}.

\begin{definition}[Double strings]\label{def:dbl-string-and-braid-word-seq}
A \emph{double string} $\dstr=(i_1 A_1, \dots, i_l A_{l})$ is a tuple of entries of the form $iA$, where $i \in I$ and $A \in \{L,R\}$. A double string $\dstr$ determines a sequence of braid words $\br^{\dstr}_0, \dots, \br^{\dstr}_l$ by the rule
	\[\br^{\dstr}_0= e  \quad \text{and} \quad  \br^{\dstr}_{\k}= \begin{cases}
		\br^{\dstr}_{\k-1} i_\k & \text{if } A_\k=R\\
		i_k \br^{\dstr}_{\k-1} & \text{if } A_\k=L.
	\end{cases}\]
The notation $L,R$ in $\dstr$ stands for Left and Right. By definition, $\dstr$ is said to be a double string for $\br^{\dstr}_l$. To ease notation, the Demazure product $\delta(\br_k^{\dstr})$ is denoted by $w_k^{\dstr}$.\hfill$\Box$
\end{definition}

\subsubsection{Double inductive weaves}\label{sssec:double_inductive_weaves} We draw double inductive weaves $\ww$ in a triangle $\triangle$ embedded in the $(x,y)$-plane, as in \cref{fig:ind-weave-triangle}. By definition, a point $(x,y)$ is said to have \emph{depth} $d:=-y$ and the \emph{slice} of a weave $\ww\sse\triangle$ at depth $d$ refers to the list of weave lines encountered at depth $d$, read left-to-right. For reference, the top vertex of the triangle is always drawn at depth $-\varepsilon$ and the (horizontal) base of the triangle is at depth $d=l + \varepsilon$, $l$ being the length of $\beta$ and $\varepsilon\in\R_+$ small enough.

\begin{definition}[Double inductive weaves]\label{def:dbl-inductive-weave}
	Let $\dstr = (i_1 A_1, \dots, i_l A_{l})$ be a double string for a braid word $\br$. By definition, the double inductive weave $\DW$ associated to $\dstr$ as follows. First, each entry $i_\k A_\k$ of $\dstr$ corresponds to a weave line of color $i_\k$ which starts at depth $d=\k-0.4$ on the $A_\k$ side of the triangle and goes down. Second, $\DW$ is then constructed scanning top-to-bottom, according to the following:
 	\begin{enumerate}
		\item The slice at $d=\k$ gives a reduced expression for the Demazure product $w_\k^{\dstr}$.\\
        
		\item If $w_{\k-1}^{\dstr} \neq w_{\k}^{\dstr}$, then there are no vertices between the slices at $d=\k-1$ and $d=\k$ and the weave lines continue vertically through this strip.\\
        
		\item If $w_{\k-1}^{\dstr}= w_{\k}^{\dstr}$, then there are only vertices of degree larger than 3 between the slices at $d=\k-1$ and $d=\k-0.5$. In this case, the slice at $d=\k-0.5$ must be a reduced expression for $w_{\k-1}^{\dstr}$ which ends in $i_{\k}$ if $A_{\k}=R$, and which begins in $i_{\k}$ if $A_{\k}=L$. Between the slices $d=\k-0.4$ and $d=\k$, there is exactly one vertex $\vertex$, which is trivalent and involves the new weave line colored $i_\k$ that starts at $d=\k-0.4$. 
 	\end{enumerate}

Finally, for $\k\in[1,l]$, we define $\DW_\k$ to be the truncation of the weave $\DW$ up to depth $\k$, i.e.~$\DW_\k$ consists of the strips from depth $d=0-\varepsilon$ to depth $\k$.\hfill$\Box$
\end{definition}

Note that $\DW_{l} = \DW$ and that $\DW_\k$ is a double inductive weave for the braid word $\br^{\dstr}_{\k}$ according to \cite[Section 6.4]{CGGLSS}. As introduced in \cref{def:dbl-inductive-weave}, there is an ambiguity in the definition of $\DW$, as different choices of vertices between $d=k-1$ and $d=k-0.5$ result in technically different weaves. They are nevertheless equivalent, cf.~\cref{rmk:equivalence_weave}, and we denote all such equivalent weaves by $\DW$. See \cref{fig:ind-weave-triangle-B-labels} for an instance of a double inductive weave as in \cref{def:dbl-inductive-weave}, with the relevant slices depicted as dashed horizontal lines: it is the double inductive weave associated to the double string $\dstr=(2R, 1R, 2L, 1R, 2R, 1R, 1L)$.

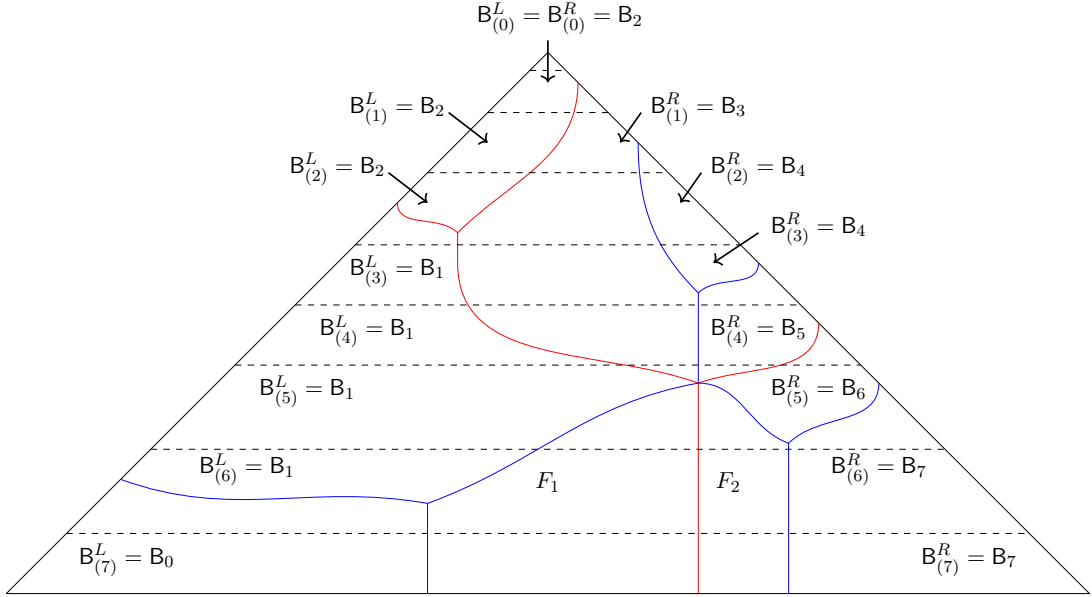
\begin{figure}[h!]
    \centering
    \adjustbox{scale=0.8}{
\begin{tikzcd}
\draw (-1,-1) to (17,-1) to (8,8) to (-1,-1);

\draw[color=red] (8.5, 7.5) to[out=270, in=45] (6.5, 5);
\draw[color=blue] (9.5, 6.5) to[out=270,in=135] (10.5, 4);
\draw[color=red] (5.5,5.5) to[out=270,in=135] (6.5, 5) to (6.5, 4.5) to [out=270,in=160] (10.5, 2.5);
\draw[color=blue] (11.5, 4.5) to[out=270,in=45] (10.5, 4) to (10.5, 2.5);
\draw[color=red] (12.5, 3.5) to[out=270, in=20] (10.5, 2.5);

\draw[color=red] (10.5, 2.5) to (10.5,-1);

\draw[color=blue] (10.5, 2.5) to[out=0,in=160] (12, 1.5) to (12, -1);

\draw[color=blue] (13.5, 2.5) to[out=270, in=45] (12, 1.5);

\draw[color=blue] (0.9, 0.9) to[out=340, in=170] (6, 0.5);

\draw[color=blue] (10.5, 2.5) to[out=190,in=20] (6, 0.5) to (6,-1);

\draw[dashed] (0,0) to (16,0);

\draw[dashed] (1.4,1.4) to (14.6,1.4);

\draw[dashed] (2.8, 2.8) to (13.2, 2.8);

\draw[dashed] (3.8, 3.8) to (12.2, 3.8);

\draw[dashed] (4.8, 4.8) to (11.2, 4.8);

\draw[dashed] (6, 6) to (10, 6);

\draw[dashed] (7,7) to (9,7);

\draw[dashed] (7.7, 7.7) to (8.3, 7.7);

\node at (8, 0.75) {F_1};
\node at (11, 0.75) {F_2};

\node at (1, -0.5) {\mathsf{B}^L_{(7)} = \mathsf{B}_0};

\node at (15, -0.5){\mathsf{B}^R_{(7)} = \mathsf{B}_7};

\node at (3, 1) {\mathsf{B}^L_{(6)} = \mathsf{B}_1};
\node at (13.5, 1) {\mathsf{B}^R_{(6)} = \mathsf{B}_7};

\node at (4, 2.3) {\mathsf{B}^L_{(5)} = \mathsf{B}_1};
\node at (12.5, 2.3) {\mathsf{B}^R_{(5)} = \mathsf{B}_6};

\node at (5, 3.3) {\mathsf{B}^L_{(4)} = \mathsf{B}_1};
\node at (11.5, 3.3) {\mathsf{B}^R_{(4)} = \mathsf{B}_5};

\node at (5.5, 4.3) {\mathsf{B}^L_{(3)} = \mathsf{B}_1};
\node at (12.5, 5) {\mathsf{B}^R_{(3)} = \mathsf{B}_4};
\draw[thick, ->] (11.5, 5) to (10.75, 4.5);

\node at (4.5, 6)  {\mathsf{B}^L_{(2)} = \mathsf{B}_2};
\draw[thick, ->] (5.35, 6) to (6, 5.5);

\node at (11.5, 6)  {\mathsf{B}^R_{(2)} = \mathsf{B}_4};
\draw[thick, ->] (10.55, 6) to (10.2, 5.5);

\node at (5.5, 7)  {\mathsf{B}^L_{(1)} = \mathsf{B}_2};
\draw[thick, ->] (6.35, 7) to (7, 6.5);

\node at (10.5, 7)  {\mathsf{B}^R_{(1)} = \mathsf{B}_3};
\draw[thick, ->] (9.55, 7) to (9.2, 6.5);

\node at (8.2, 8.5) {\mathsf{B}^L_{(0)} = \mathsf{B}^R_{(0)} = \mathsf{B}_2};
\draw[thick, ->] (8, 8.2) to (8, 7.5);

\end{tikzcd} 
    
    }
    \caption{The weave $\DW$ for the double string $(2R, 1R, 1^{*}L, 1R, 2R, 1R, 2^{*}L)$ and double braid word $(-2,1,2,1, -1,1,2)$, together with the labeling of the regions of $\triangle \setminus \DW$ by the flags $\borel_i$. The flags $\borel_{(\c)}, \borel'_{(\c)}$ label the regions neighboring the left and right sides, respectively, of the triangle $\triangle$. Compare with the diagram \eqref{fig:triangle_with_flags} in \cref{ex:running-iso}.
    }
    \label{fig:ind-weave-triangle-B-labels}
\end{figure}

\subsubsection{Weave tori for double inductive weaves}\label{sssec:weavetori_doubleinductive} Let us provide another description of the weave torus $T_{\DW} \sse X(\br^{\dstr}_l)$ associated to a double inductive weave $\DW\sse\triangle$ for a double string $\dstr$ of length $l$.

For each  $\c\in[0,l]$, let $L_{\c}$ (resp.~$R_{\c}$) be the leftmost (resp.~rightmost) connected component of $\triangle \setminus \DW$ intersected by a horizontal line of depth $\d = \c$. Note that $L_0 = R_0$ is the unique component touching the apex of the triangle, $L_{l}$ contains the bottom left vertex and $R_{l}$ contains the bottom right vertex. In addition, for each $\c\in[0,l - 1]$, $L_{\c} = L_{\c +1}$ (resp. $R_{\c} = R_{\c+1}$) if and only if in the double string $\dstr$ we have $A_{\c+1} = R$ (resp. $A_{\c+1} = L$). 

For each point $\borel_{\bullet}=(\borel_0, \borel_1, \dots, \borel_{l}) \in X(\br^{\dstr}_l)$, place the flags in the top regions of $\triangle \setminus \DW$, starting with $\borel_0$ in the region $L_{l}$ and moving parallel to the side edges of the triangle. These uniquely determine a flag labeling, cf.~\cref{ssec:weave_tori}. For each $\k\in[0,l]$, let $\borel_{(\k)}^L$, resp.~$\borel_{(\k)}^R$, be the flag labeling the region $L_\k$, resp.~$R_\k$. See \cref{fig:ind-weave-triangle-B-labels} for an example with all such labels.

Points $\borel_{\bullet}\in X(\br^{\dstr}_l)$ that belong to the weave torus $T_{\DW} \sse X(\br^{\dstr}_l)$ are then characterized by the following property:

\begin{lemma}\label{lem:description-double-inductive-weave}
Let $\dstr$ a double string and $\DW$ the corresponding double inductive weave. 
Then 
\begin{center}
$\borel_{\bullet} \in T_{\DW}\Longleftrightarrow$ $\borel_{(\k)}^L \Rrel{w_{\k}^{\dstr}} \borel^R_{(\k)}$, $\forall k\in[0,l]$.
\end{center}

\end{lemma}

\begin{proof}
$(\Rightarrow)$ Let us show that $(\borel_0, \dots, \borel_{l}) \in T_{\DW}$ implies $\borel_{(\k)}^L \Rrel{w_{\k}^{\dstr}} \borel_{(\k)}^R$ for all $k$. Indeed, fix a flag labeling of $\DW$ with top regions labeled by $(\borel_0, \dots, \borel_{l})$ and restrict this labeling to $\DW_k$. The left region of $\DW_\k$ is labeled by $\borel_{(\k)}^L$ and the right region is labeled by $\borel^R_{(\k)}$. Since the bottom edges of $\DW_\k$ spell a reduced word for $\delta(\br_\k^{\dstr})= w_\k^{\dstr}$, which is a permutation, by \cref{lem:properties-rel-pos} (2) we must have $\borel_{(\k)}^L \Rrel{w_{\k}^{\dstr}} \borel^R_{(\k)}$.\\

$(\Leftarrow)$ Supppose that $\borel_{(\k)}^L \Rrel{w_{\k}^{\dstr}} \borel^R_{(\k)}$ for all $k\in[0,l]$. We show inductively that the sequence $(\borel^L_{(\k)}, \dots, \borel^R_{(\k)})$ appear as the top tuple of a flag labeling for the partial double inductive weave $\DW_{\k}$ obtained after $\k$ steps, relaxing the condition on the leftmost and rightmost flags if necessary. Once this is proven, this implies that $(\borel_0, \dots, \borel_{l}) \in T_{\DW}$, as required. The base case is $k = 1$, and since the weave $\DW_{1}$ is simply a line, the statement is readily true. For the inductive step, suppose that $(\borel^L_{(\k)}, \dots, \borel^R_{(\k)})$ is the top tuple of a flag labeling $\mathcal{L}$ of $\DW_{k}$. For concreteness, assume that $A_{k+1} = L$: the proof in the case $A_{k+1} = R$ is similar. Then $\br^{\dstr}_{k+1} = i_{a_{k}-1}\br^{\dstr}_{k}$, and we need to show that
$$(\borel^L_{(k+1)}, \borel^L_{(k)}, \dots, \borel^R_{(k)} = \borel^R_{(k+1)})$$
is the top tuple of a flag labeling of $\DW_{k+1}$. By construction, $\DW_{k+1}$ is obtained from $\DW_{k}$ by adding first a string on the left and then, possibly, a strip of vertices with degree at least 4, followed by a strip with a trivalent vertex. We label $\DW_{k+1}$ by extending $\mathcal{L}$ to a flag labeling of the strip of degree $\ge 4$ vertices (if it exists) and then placing $\borel^L_{(k+1)}$ in the far left region. By \cref{lem:properties-rel-pos}(4), this fails to be a valid flag labeling of $\DW_{k+1}$ exactly when $\DW_{k+1}$ and $\DW_k$ differ by a trivalent vertex (equivalently, when
~$w_{k+1}^{\dstr}= w_{k}^{\dstr}$) and $\borel^L_{(k+1)}$ is equal to the flag $\borel$ to the right of the trivalent vertex. In that case the distance between $\borel^L_{(k+1)}=\borel$ and $\borel^R_{(k+1)} = \borel^R_{(k)}$ is $s_{i_{a_k}}w_{k+1}^{\dstr}$, which is strictly less than $w_{k+1}^{\dstr}$: this is a contradiction and thus the resulting flag labeling must have been valid.
\end{proof}

\begin{example} 
Note that the weave $\ww$ from \cref{ex:running_example_weave_torus} is double inductive. Let us verify that the description of the torus $T_{\ww}$ given by \cref{lem:description-double-inductive-weave} coincides with the one obtained in \cref{ex:running_example_weave_torus}. Let us recall that in \cref{ex:running_example_weave_torus} we obtained that $T_{\ww}$ is given by the conditions:
\begin{equation}\label{eq:conditions-running-example-weave-torus}
\borel_1 \Rrel{s_2} \borel_3, \qquad \borel_3 \Rrel{s_1} \borel_5, \qquad \borel_0 \Rrel{s_1} F_1, \qquad F_2 \Rrel{s_1} \borel_7,
\end{equation}
where $F_1$ and $F_2$ are flags obtained from the top tuple $(\borel_0, \dots, \borel_7)$, see \cref{fig:ind-weave-triangle-B-labels}. Reading this figure top-to-bottom, we also obtain that the description of $T_{\ww}$ in \cref{lem:description-double-inductive-weave} is:
\begin{equation}\label{eq:conditions-running-example-from-lemma}
\borel_2 \Rrel{s_2} \borel_3, \; \borel_2 \Rrel{s_2s_1} \borel_4, \; \borel_1 \Rrel{s_2s_1}\borel_4, \; \borel_1\Rrel{s_2s_1}\borel_5, \; \borel_1 \Rrel{\wo} \borel_6, \; \borel_1 \Rrel{\wo}\borel_7, \; \borel_0 \Rrel{\wo} \borel_7. 
\end{equation}

By the definition of the braid variety $X(\br)$ and \cref{lem:properties-rel-pos} (2), the first two conditions in \eqref{eq:conditions-running-example-from-lemma} are valid for any element of $X(\br)$, so they can be omitted. By \cref{lem:properties-rel-pos}(2), we have that $\borel_1 \Rrel{s_2s_1}\borel_4$ is equivalent to $\borel_1 \neq \borel_3$, and by \cref{lem:properties-rel-pos}(4) this is in turn equivalent to $\borel_1 \Rrel{s_2}\borel_3$, i.e. the conditions $\borel_1 \Rrel{s_2} \borel_3$ and $\borel_1 \Rrel{s_2s_1}\borel_4$ are equivalent. Similarly, assuming $\borel_1 \Rrel{s_2}\borel_3$, the condition $\borel_3 \Rrel{s_1}\borel_5$ is equivalent to $\borel_1 \Rrel{s_2s_1}\borel_5$.

By \cref{lem:properties-rel-pos}(2) and the definition of $X(\br)$, the condition $\borel_1 \Rrel{s_2s_1}\borel_5$ implies $\borel_1 \Rrel{\wo} \borel_6$, so the latter can be omitted from \eqref{eq:conditions-running-example-from-lemma}. Similarly to the previous paragraph, assuming $\borel_1 \Rrel{\wo} \borel_6$, the condition $\borel_0 \Rrel{s_1} F_1$ is equivalent to $\borel_1 \Rrel{\wo} \borel_7$, and assuming the latter we have that $F_2 \Rrel{s_1}\borel_7$ is equivalent to $\borel_0 \Rrel{\wo} \borel_7$. So \eqref{eq:conditions-running-example-weave-torus} and \eqref{eq:conditions-running-example-from-lemma} indeed define the same subvariety of $X(\br)$.\hfill$\Box$
\end{example}

\subsection{Deodhar tori}\label{sec:deodhar}
Let $\dbr$ be a double braid word, \cite[Definition 7.1]{GLSBS} and \cite[Definition 2.5]{GLSB} introduce an open torus $T_{\dbr} \subseteq R(\dbr)$, referred to as the \emph{Deodhar torus}. It depends on the following data that can be readily read from the double braid word $\dbr$. 

\begin{definition}\label{def:w-and-u}
    Let $\dbr= i_1\cdots i_{n+m}$ be a double braid word. Define a sequence of elements $w_{m+n}, \dots, w_{0} \in W$ by 
    \begin{equation}\label{eq:inductively-get-w}
w_{m+n} = \id \; \text{and} \; w_{c-1} = s_{i_c^{\ast}}^{-} * w_c * s_{i_c}^{+}.
\end{equation}
where $*$ denotes the Demazure product. 
\end{definition}

Note that $w_0 = \delta(\dbr)$. Since we are assuming that $\delta(\dbr) = \wo$, the following definition makes sense.

\begin{definition}[Deodhar torus]\label{def:deodhar-torus} 
Let $\dbr$ be a double braid word. By definition, the Deodhar torus $T_{\dbr} \subset R(\dbr)$ is the subvariety of tuples $(X_\bullet, Y_{\bullet})$ of pairs of flags, modulo the diagonal action of $\G$, satisfying the transversality conditions
	\begin{equation}\label{eq:deodhar-torus}
		\begin{tikzcd}
			X_0& \arrow[l,"{s_{i_1}^+}"'] X_1& \arrow[l,"{s_{i_2}^+}"'] X_2 &\arrow[l,"{s_{i_3}^+}"'] \cdots& \arrow[l,"{s_{i_{m+n-1}}^+}"'] 
			X_{m+n-1} & \arrow[l,"{s_{i_{m+n}}^+}"'] 
			X_{m+n} \\
			Y_0 \arrow[r,"{s_{i_1^\ast}^-}"'] \arrow[u, Rightarrow, "\wo"']& Y_1 \arrow[r,"{s_{i_2^\ast}^-}"'] \arrow[u, Rightarrow, "w_1"']&Y_2 \arrow[r,"{s_{i_3^\ast}^-}"'] \arrow[u, Rightarrow, "w_2"'] & \cdots \arrow[r,"{s_{i_{m+n-1}^\ast}^-}"'] & Y_{m+n-1} \arrow[r,"{s_{i_{m+n}^\ast}^-}"'] \arrow[u, Rightarrow, "w_{m+n-1}"'] & Y_{m+n} \ar[equal]{u}
		\end{tikzcd}
	\end{equation}
	\hfill$\Box$
\end{definition}
It is proven in \cite[Corollary 2.8]{GLSB} that $T_{\dbr}$ is an algebraic torus of dimension $m+n-\ell(\wo)$ and an open subset of $R(\dbr)$. Intuitively speaking, the Deodhar torus $T_{\dbr}$ is given by the condition that for every index $c\in[0,m+n-1]$, the weighted flags $Y_c$ and $X_c$ are as far away from each other as possible, starting with $Y_{m+n} = X_{m+n}$.

\subsection{The Deodhar torus as a weave torus} This section shows that the isomorphism $\iso:R(\dbr)\lr X(\dbr^{(-|+)})$ from \cref{lem:iso-glsbs-cgglss} maps the Deodhar torus $T_{\dbr} \subseteq R(\dbr)$ isomorphically onto a torus $T_{\DW}$ defined by a specific double inductive weave ${\DW}:\brWe\to\wo$. We start by describing this weave.

\subsubsection{Double strings from double braid words} The weave $\DW$ such that $\iso(T_\dbr)=T_{\DW}$ is associated to a double string $\dstr$. The relation between double strings $\dstr$, as introduced in \cref{def:dbl-string-and-braid-word-seq}, and double braid words $\dbr$ is as follows: given a double braid word $\dbr$, we can associate a double string $\dstr(\dbr)$ by reading $\dbr$ right-to-left and replacing each positive letter $i$ with $iR$ and each negative letter $j$ with $|j|^* L$. More formally:

\begin{definition}
\label{def:double_string_from_double_braid_word}
Let $\dbr = i_1\cdots i_{n+m}$ be a double braid word, with negative and positive index sets $\negind, \posind$, and for each $c \in [0, m+n]$ consider $\invc:=m+n-c$. By definition, the double string $\dstr(\dbr) = (j_1A_1, \dots, j_{n+m}A_{n+m})$ associated to $\dbr$ is given by
\begin{equation}\label{eq:deodhar-double-inductive}
j_{c+1}A_{c+1} := \begin{cases} i_{\invc}R & \text{if} \; \invc \in \posind, \\
|i_{\invc}|^{\ast}L & \text{if} \; \invc \in \negind.\end{cases}
\end{equation}
where $c\in[0,n+m-1]$.\hfill$\Box$
\end{definition}

By \cref{def:pos-neg-ind}, the braid word associated to $\dbr$ is $\dbr^{(-|+)}$: the double string $\dstr(\dbr)$ in \cref{def:double_string_from_double_braid_word} is indeed a double string for $\dbr^{(-|+)}$.

\begin{example}\label{ex:from-double-braid-to-double-string-running-example}
Consider the double braid word $\dbr = (-2,1,2,1,-1,1,2)$ as in \cref{ex:running}. The associated double string via \cref{def:double_string_from_double_braid_word} is $\dstr(\dbr) = (2R, 1R, 1^{*}L, 1R, 2R, 1R, 2^{*}L)$, which is indeed a double string for the braid word $\dbr^{(-|+)} = 2^{\ast}1^{\ast}21121=1221121$.\hfill$\Box$
\end{example}

The association of a double string $\dstr(\dbr)$ to each double braid word $\dbr$ in \cref{def:double_string_from_double_braid_word} yields a bijective correspondence:

\begin{lemma} 
\label{lem:double_words_bij_double_strings}
Let $\beta$ be a positive braid word, $\mathbb{W}(\beta)$ the set of double braid words $\dbr$ such that $\dbr^{(-|+)} = \beta$, and $\mathbb{S}(\beta)$ the set of double strings for $\beta$. Then
$$\mathbb{W}(\beta)\lr\mathbb{S}(\beta),\quad \dbr \mapsto \dstr(\dbr)$$
is a bijection.
\end{lemma}

\begin{proof}
For a double string $\dstr=(j_1 A_1, \dots, j_l A_{l})$, consider the double braid word $\dbr(\dstr) = i_1\cdots i_{l}$ given by 
\begin{equation}
i_{c} := \begin{cases} j_{\invc+1} & \text{if} \; A_{\invc+1} = R, \\
-j_{\invc+1}^{\ast} & \text{if} \; A_{\invc+1} = L,\end{cases}
\end{equation}
for $c \in [1, l]$. If $\dstr$ is a double string for $\beta$ then $(\dbr(\dstr))^{(-|+)} = \beta$ and the assignment $\dstr \mapsto \dbr(\dstr)$ defines the inverse $\mathbb{S}(\beta)\lr\mathbb{W}(\beta)$ of the map $\dbr \mapsto \dstr(\dbr)$.
\end{proof}

Recall that each double braid word $\dbr$ defines a sequence of elements $w_{m+n}, \dots, w_0 \in W$, cf. \cref{def:w-and-u}. Likewise, a double string $\dstr$ of length $m+n$ defines a sequence of elements $w_0^{\dstr}, \dots, w_{m+n}^{\dstr}$ as in \cref{def:dbl-string-and-braid-word-seq}.  The next lemma shows that the bijection in \cref{lem:double_words_bij_double_strings} is compatible with these sequences. 

\begin{lemma}\label{lem:br-w-and-weave-w-coincide}
Let $\dbr=i_1\cdots i_{m+n}$ be a double braid word and $\dstr=\dstr(\dbr)$ its associated double string. Then for any $c \in [0, n+m]$, 
\[w_c = w_{\invc}^{\dstr},\] 
where $\invc = m+n-c$. 
\end{lemma}
\begin{proof}
We verify that $w_{\invc}=w_{c}^{\dstr}$ for $c=0, \dots, m+n$, in increasing order. By \cref{eq:inductively-get-w} and \cref{def:dbl-string-and-braid-word-seq}, $w_{m+n} = \id= w_0^{\dstr}$. Assume now that $w_{\invc}= w^{\dstr}_{c}$. By definition, $w_{\overline{c+1}} =w_{\invc -1}=s_{i_{\invc}^{\ast}}^{-} * w_{\invc} * s_{i_{\invc}}^{+}$. To determine $w^{\dstr}_{c+1}$, note that by \eqref{eq:deodhar-double-inductive}, 
\[\beta^{\dstr}_{c+1} = \begin{cases}
    \beta^{\dstr}_{c}i_{\invc} & \text{ if } \invc \in J_+\\
    |i_{\invc}|^* \beta^{\dstr}_{c} & \text{ if } \invc \in J_-.\\
\end{cases}\]
Using the inductive definition of Demazure products, we have that $w^{\dstr}_{c+1}= s_{i_{\invc}^{\ast}}^{-} * w_{\invc} * s_{i_{\invc}}^{+}$ as desired.
\end{proof}

\color{black}

\subsubsection{Tori comparison} Given a double braid word $\dbr$, denote by $\DW(\dbr)$ the double inductive weave associated to the double string $\dstr(\dbr)$ in \cref{def:double_string_from_double_braid_word}. 

\begin{proposition}\label{lem:weave-tori-vs-deodhar-tori}
Let $\dbr$ be a double braid word and $\iso:R(\dbr)\lr X(\brWe)$ the isomorphism from \cref{eq:iso-glsbs-cgglss}. Then
$$\iso(T_{\dbr})=T_{\DW(\dbr)}.$$
\end{proposition}

\begin{proof} Consider a point $(X_\bullet, Y_\bullet) \in T_{\dbr}$ and, using the $\G$-action, we assume that $Y_0 = \U_{+}$ and $\pi(X_0) = \wo \borel_{+}$. Let $\borel_{\bullet} \in X(\dbr^{(-|+)})$ denote $\iso(X_\bullet, Y_\bullet)$. Use the notation of Lemma \ref{lem:description-double-inductive-weave} to define $\B_{(\c)}^L, \B_{(\c)}^R$ for $\c \in [0, n+m]$. We claim that for every $\c \in [0,n+m]$ we have the flag equalities
\[
\pi(Y_\c) = \borel_{(\invc)}^L, \qquad \pi(X_\c) = \borel_{(\invc)}^R,
\]
where we again use the notation $\invc:=m+n-c$. If the claim is established, then the characterization of $T_{\DW(\dbr)}$ in Lemma \ref{lem:description-double-inductive-weave} together with \cref{lem:br-w-and-weave-w-coincide} imply the required result. We prove the claim by induction on $\c$. The base case is $\c = 0$: by definition $\pi(Y_0) = \borel_0 = \borel_{(m+n)}^L$ and $\pi(X_0) = \borel_{m+n} = \borel^R_{(m+n)}$, thus the claim holds. For the inductive step, let us suppose the claim is true for $\c$, and proceed to prove it for $\c + 1$, as follows.

Let us focus on the case $i_{\c + 1} \in I$, as the case $i_{\c + 1} \in -I$ is similar. If $i_{\c + 1} \in I$, then $Y_\c = Y_{\c + 1}$ and in the double string $\dstr$ we have $A_{\invc} = R$. This implies that $L_{\invc} = L_{\invc - 1}$ and therefore $\pi(Y_{\c + 1}) = \pi(Y_{\c}) = \borel_{(\invc)}^L = \borel_{(\invc - 1)}^L$. This shows the first flag equality. For the second one, note that $X_\c \neq X_{\c +1}$: it then follows from \cref{eq:iso-glsbs-cgglss}, and the fact that $R_{\invc} \neq R_{\invc-1}$, that $\pi(X_{\c + 1}) = \borel^R_{(\c - 1)}$. This concludes the claim.
\end{proof}

\begin{remark}\label{rem:iso-on-triangle-weave}
The proof of \cref{lem:weave-tori-vs-deodhar-tori} gives an alternative description of the isomorphism $\iso$, where points in the braid variety $X(\brWe)$ are understood as top tuples, and points in $T_{\DW(\dbr)}$ as top tuples for flag labelings of $\DW$. Indeed, let $(X_\bullet, Y_\bullet)$ be the unique representative of a point in $R(\dbr)$ such that $Y_0=U_+$ and $\pi(X_0)=\wo \borel_+$, and label region $L_{\c}$ with $\pi(Y_{\invc})$ and region $R_{\c}$ with $\pi(X_{\invc})$. That is, set $\borel^L_{(\c)}:= \pi(Y_{\invc})$ and $\borel_{(\c)}^R:= \pi(X_{\invc})$. To obtain the corresponding element $\iso((X_\bullet, Y_\bullet))=\borel_{\bullet} \in X(\brWe)$ read the labels of the regions, from leftmost to rightmost along the two non-horizontal sides of $\triangle$, and that yields the required tuple of flags $\borel_{\bullet}=(\borel_0, \dots, \borel_{m+n})$.\hfill$\Box$
\end{remark}

In \cite[Section 4.3]{CGGLSS} right and left inductive weaves were considered, as special cases of double inductive weaves. 
Specifically, if $\br = i_1 \dots i_r$ is a braid word, then its \emph{right} inductive weave is the double inductive weave corresponding to the double string $(i_1R, i_2R, \dots, i_rR)$, while its \emph{left} inductive weave is the double inductive weave corresponding to the double string $(i_rL, i_{r-1}L, \dots, i_1L)$, cf.~\cite[Section 4.3]{CGGLSS} for details. The corresponding weave tori are referred to as the right and left inductive tori. \cref{lem:weave-tori-vs-deodhar-tori} implies the following fact:

\begin{cor}\label{cor:inductive-weaves} Let $\dbr$ be a double braid word and $\iso:R(\dbr)\lr X(\brWe)$ the isomorphism from \cref{eq:iso-glsbs-cgglss}. The following holds:
\begin{enumerate}
\item Suppose $\dbr$ is a double braid word using only letters from $I$, $\dbr = j_1\cdots j_m$. Then,
$$\iso(T_{\dbr})\sse X(j_m\dots j_1)$$
is the right inductive torus of $\beta=j_m\dots j_1$.\\

\item Suppose $\dbr$ is a double braid word using only letters from $-I$, $\dbr = (-i_1)\cdots(-i_n)$. Then,
$$\iso(T_{\dbr})\sse X(i_1^{\ast}\cdots i_n^{\ast})$$
 is the left inductive torus of $\beta=i_1^{\ast}\cdots i_n^{\ast}$.\qed
\end{enumerate}
\end{cor}

\begin{table}[h!]
\begin{center}
\caption{A schematic reference for the concepts introduced in Sections \ref{sec:braid_varieties} and \ref{sec:weaves-tori}. Objects of a \textcolor{blue}{combinatorial nature are highlighted in blue}, whereas objects of a more \textcolor{purple}{algebraic geometric nature are highlighted in red}.}\label{table:summary_section4}
\begin{tabular}{ |P{3cm}|P{10cm}|  }
	\hline
	\multicolumn{2}{|c|}{ {\bf Summary of constructions and results thus far}} \\
	\hline

\cref{def:pos-neg-ind} & double braid word \textcolor{blue}{$\dbr$} $\leadsto$ braid word \textcolor{blue}{$\dbr^{(-|+)}$} \\
\hline

\cref{def:braid-variety} & braid word \textcolor{blue}{$\br$} $\leadsto$ braid variety \textcolor{purple}{$X(\br)$}\\
\hline

\cref{def:dbl-braid-variety} & double braid word \textcolor{blue}{$\dbr$} $\leadsto$ double braid variety \textcolor{purple}{$R(\dbr)$} \\
\hline

\cref{lem:iso-glsbs-cgglss} & isomorphism \textcolor{purple}{$\iso: R(\dbr) \xrightarrow{\sim} X(\dbr^{(-|+)})$}\\
\hline

\cref{def:dbl-inductive-weave} & double string \textcolor{blue}{$\dstr$} $\leadsto$ double inductive weave \textcolor{blue}{$\DW(\dstr)$}\\
\hline

\cref{def:double_string_from_double_braid_word} & double braid word \textcolor{blue}{$\dbr$} $\leadsto$ double string \textcolor{blue}{$\dstr(\dbr)$}\\
\hline

\cref{sssec:weavetori_doubleinductive} & double inductive weave \textcolor{blue}{$\DW$} $\leadsto$ weave torus \textcolor{purple}{$T_{\DW}\sse X(\beta)$}\\
\hline

\cref{sec:deodhar} & double braid word \textcolor{blue}{$\dbr$} $\leadsto$ Deodhar torus \textcolor{purple}{$T_{\dbr}\sse R(\dbr)$}\\
\hline

\cref{lem:weave-tori-vs-deodhar-tori} & tori equality \textcolor{purple}{$\iso(T_\dbr)=T_{\DW}$} for \textcolor{blue}{$\DW:=\DW(\dstr(\dbr))$}\\
\hline
\end{tabular}
\end{center}
\end{table}

\subsection{A table summary of objects at this stage} \cref{table:summary_section4} summarizes the main ingredients, constructions and results thus far, similar to \cref{table:summary_intro}. The construction and comparison of cluster variables will use these ingredients, so we give a summarizing account at this point to aid the reader in the next sections.

\section{Comparison of cluster variables}\label{sec:cluster_var} The goal of this section is to show that the cluster variables in \cite{CGGLSS} and \cite{GLSBS,GLSB} coincide. Specifically, \cref{lem:weave-tori-vs-deodhar-tori} implies that the isomorphism
$$\iso:R(\dbr)\lr X(\brWe)$$
from \cref{eq:iso-glsbs-cgglss} satisfies $\iso(T_{\dbr})=T_{\DW(\dbr)}$, where $T_\dbr$ is the Deodhar torus of the double braid word $\dbr$ and $T_{\DW(\dbr)}$ is the weave torus for the double inductive weave $\DW$ associated to the double string in \cref{def:double_string_from_double_braid_word}, cf.~\cref{table:summary_section4}. For the cluster variables in such tori:

\begin{enumerate}
    \item In \cite[Sect.~5.2\&6]{CGGLSS}, cluster variables $\{\xWe_{\e}\}$ for the weave torus $T_{\DW(\dbr)}\sse X(\brWe)$ are constructed such that $\xWe_{\e}\in \C[X(\brWe)]$. These cluster variables are indexed by the trivalent vertices of the weave $\DW(\dbr)$, which are in bijection with certain letters of $\brWe$. \\
    
    \item In \cite[Section 2.4]{GLSBS} and \cite[Section 2.7]{GLSB}, cluster variables $\{\x3D_{\inve}\}$ for the Deodhar torus $T_\dbr\sse R(\dbr)$ are constructed such that $\x3D_{\inve}\in\C[R(\dbr)]$. These cluster variables are indexed by Deodhar hypersurfaces, which are in bijection with certain letters of $\dbr$.
\end{enumerate}

Both tori $T_{\dbr}$ and $T_{\DW(\dbr)}$ are given by the non-vanishing of $\{\xWe_{\e}\}$ and $\{\x3D_{\inve}\}$, respectively, i.e.~there are isomorphisms
\begin{equation}\label{eq:tori_nonvanishing_variables}
T_{\dbr}\cong\Spec(\C[(\x3D_{\inve})^{\pm1}]),\quad T_{\DW(\dbr)}\cong\Spec(\C[(\xWe_{\e})^{\pm1}])
\end{equation}
where the indices $\e,\inve$ are understood to range through all cluster variables in the corresponding sets. \cref{eq:tori_nonvanishing_variables} and \cref{lem:weave-tori-vs-deodhar-tori} imply that the pullback isomorphism
$$\iso^*:\C[X(\brWe)]\lr\C[R(\dbr)]$$
must be such that the pullback $\iso^*(\xWe_{\e})$ of a cluster variable $\xWe_{\e}$ is a Laurent monomial on the cluster variables $\{\x3D_{\inve}\}$. The goal of this section is to show that the Laurent monomial $\iso^*(\xWe_{\e})$ consists solely of a cluster variable $\x3D_{\inve}$. That is, for all indices $e$, the goal is to show that there exists a unique $\inve$ such that
\begin{equation}\label{eq:equality_cluster_variables}
\iso^*(\xWe_{\e})=\x3D_{\inve}.
\end{equation}
    and, moreover, the assignment $e \mapsto \inve$ gives a bijection betweeen the corresponding indexing sets. In short, \cref{eq:equality_cluster_variables} states that the cluster variables of \cite{CGGLSS} and \cite{GLSBS,GLSB} coincide under the isomorphism $\iso^*$. This is achieved in two steps:
\begin{enumerate}
    \item First, we show that $\iso^*(\xWe_{\e})$ is proportional to $\x3D_{\inve}$, i.e.~we prove the equality in \cref{eq:equality_cluster_variables} up to non-zero constant $\la_e\in\C^\times$.
    \item Second, we show that the constant of proportionality $\la_e$ is exactly $1$, i.e.~we establish the equality in \cref{eq:equality_cluster_variables}.
\end{enumerate}

Conceptually, the first step requires only part of the technical results of \cite{CGGLSS} and \cite{GLSBS,GLSB}, whereas the second step is proven using a significant amount of the technical details from both these approaches. In what follows, Subsections \ref{sec:weave-cluster-variables1} and \ref{ssec:deodhar_variables1} start discussing $\{\xWe_{\e}\}$ and $\{\x3D_{\inve}\}$, respectively, and \cref{subsec:init_cl_var_comp} achieves Step (1), in \cref{lem:coincidence-cluster-variables-constants}. Subsections \ref{sec:weave-cluster-variables2}, \ref{ssec:deodhar_variables2} and \ref{ssec:comparing-cocharacters} further establish some needed results on $\{\xWe_{\e}\}$ and $\{\x3D_{\inve}\}$, and \cref{ssec:equality_variables} establishes Step (2), in \cref{thm:coincidence-cluster-variables}.


\subsection{The cluster variables $\{\xWe_{\e}\}$}\label{sec:weave-cluster-variables1} In full generality, a description of the cluster variables $\{\xWe_{\e}\}$ for an arbitrary Demazure weave is provided in \cite[Section 5.2]{CGGLSS}. Since the comparison to $\{\x3D_{\inve}\}$ only requires the use of the double inductive weave $\DW$, we now provide a more tailored description of $\{\xWe_{\e}\}$, cf.~\cite[Section 5.3]{CGGLSS}.

The combinatorial input data is a double string $\dstr$ for the braid word $\beta$. We maintain the notation that $\beta=\brWe$ has length $n+m$ and $\DW$ is the double inductive weave associated to $\dstr$ via \cref{def:dbl-inductive-weave}. By \cref{def:dbl-string-and-braid-word-seq}, $\dstr$ yields a sequence of braid words $\br_{\c}^{\dstr}$ whose Demazure products are denoted by $w_\c^{\dstr} = \delta(\br_\c^{\dstr})$. Diagrammatically, the cluster variables $\{\xWe_{\e}\}$ are indexed by trivalent vertices of the weave $\DW$. This can be expressed as follows:

\begin{definition}\label{def:solid-crossings-vs-vertex-crossings}
An index $\e \in \{1, \dots, m+n\}$ is said to be \emph{vertex crossing} of $\dstr$ if $w_{\e}^{\dstr} = w_{\e+1}^{\dstr}$. Equivalently, if there is a (necessarily unique) trivalent vertex between depths $\e$ and $\e+1$ on the weave $\DW$. The set of such vertex crossings associated to $\dstr$ is denoted by $$J^{\text{We}}_{\dstr} \subseteq \{1, \dots, m+n\}.$$
The trivalent vertex corresponding to $e$ is denoted $\vertex_\e$ and its cluster variable $\xWe_\e$. \hfill$\Box$
\end{definition}

The cluster variables $\{\xWe_{\e}\}$, indexed by $J^{\text{We}}_{\dstr}$, can be computed by using an edge labeling of the weave $\DW$, decorating weave lines with elements of $\G$. This decoration uses the data of the chosen pinning for $\G$, specifically the coroot $\chi_i$ associated to $i\in I$ and the variant $\BZ_i(z)=x_i(z) \ds_i$ of the exponential Chevalley generator $x_i(z)$, $z\in\C$. In precise terms:

\begin{definition}\label{def:edge-labeling}
Let $\DW$ be a double inductive weave. By definition, the \emph{edge labeling} of $\DW$ is the labeling of the edges of $\DW$ by elements of the form
$$g_\edge = B_{i}(f_\edge)\chi_i(u_\edge) \in \G,$$where $i$ is the color of the edge $\edge$ and $f_\edge, u_\edge \in \C(X(\brWe)$ are rational functions on $X(\brWe)$ obtained as follows:
\begin{itemize}
\item[(i)] The $j$th top edge $\edge_j$ of $\DW$ is labeled by
$$g_{\edge_j}:=\BZ_{i'_j}(z'_j)\chi_{i'_j}(1),\mbox{ so that }f_{\edge_j}=z'_j\mbox{ and } u_{\edge_j}=1,$$
where we have written $\dbr^{(-|+)} = i'_1\cdots i'_{n+m}$ and the functions $z_j' \in \C[X(\brWe)]$ are as defined in \cref{sec:parametrization}. In particular, $u_\edge=1$ for all top edges $\edge$ of $\DW$.\\

\item[(ii)] The functions $f_\edge,u_\edge$ for an arbitrary weave edge $\edge$ of $\DW$ are obtained by propagating down the functions $f_\edge,u_\edge$ at the top edges by using the rules from \cite[Definition 5.8]{CGGLSS}.\hfill$\Box$
\end{itemize}
\end{definition}

At this stage, \cref{def:edge-labeling} admittedly has little content, as it refers to rules from \cite{CGGLSS} which we have not specified.\footnote{\cite{CGGLSS} has raked weaves: we do not need such refinement for the comparison, so we omit this data.} The rules are such that the edge labeling of $\DW$ is unique. For now, it suffices to know that an edge labeling satisfies the following:

\begin{lemma}\label{lem:element-labeling}
Let $\DW$ be a double inductive weave and $\mathcal{L}$ a flag labeling. Then, $\mathcal{L}$ lifts to a unique labeling $\tilde{\mathcal{L}}$ of $\DW$ by weighted flags so that:
\begin{enumerate}
\item The top labels of $\tilde{\mathcal{L}}$ are the weighted flags $F_i$ given by Equation \eqref{eq:parametrized-flags-CGGLSS}.
\item If two regions $F, F'$ of $\triangle\setminus \DW$ are separated by a weave $i$-edge $\edge$ and $F$ is to the left of $\edge$, then $F' = g_\edge F$.
\end{enumerate}
\end{lemma}

\cref{lem:element-labeling} is a consequence of \cite[Lemma 5.9]{CGGLSS}. Let us denote the set of vertex crossings $J^{\text{W}}_{\dstr(\dbr)}$ by $\JWe$. The cluster variables $\{\xWe_{\e}\}$ are then defined as follows:

\begin{definition}[Weave cluster variables]\label{prop:dbl-weave-cluster-var}
Let $\DW$ be a double inductive weave, $\e \in \JWe$ the index of a trivalent vertex $\vertex_\e$ of $\DW$, and $\edge$ be the southern weave edge of $\vertex_\e$. By definition, the cluster variable $\xWe_\e$ associated to $\e$ is
$$\xWe_\e := u_\edge,$$
where $u_\edge$ is specified by the edge labeling of $\DW$.\hfill$\Box$
\end{definition}

It is implied by \cite[Theorem 5.19.(3)]{CGGLSS} that the rational function $u_\edge \in \C(X(\beta))$ is a regular function on $X(\beta)$, i.e.~we actually have $\xWe_\e\in\C[X(\beta)]$, where $\beta=\brWe$. The set of cluster variables $\{\xWe_\e\}$ in \cref{prop:dbl-weave-cluster-var} is the set of cluster variables for the weave torus $T_{\DW}\sse X(\brWe)$.

The following lemma is used in the proof of \cref{lem:coincidence-cluster-variables-constants}, in the upcoming comparison of cluster variables:

\begin{lemma}\label{lem:weave-var-zero-locus}
Let $\DW$ be a double inductive weave and $\c \in \JWe$.
The locus $$\{\xWe_\e \neq 0 \mid \e \in \JWe, \e < \c\} \cap \{\xWe_\c = 0\}\sse X(\brWe)$$ coincides with the set of points $\borel_{\bullet} \in X(\brWe)$ such that

\begin{equation}\label{eq:rel_position_condition_lemma}
\begin{cases}
\borel_{(\e)} \Rrel{w_{\e}^{\dstr}} \borel'_{(\e)} & \text{ if } \; \e < \c,\\
\borel_{(\c)} \Rrel{w} \borel'_{(\c)} & \text{ for some } \; w < w_{\c}^{\dstr}.
\end{cases}
\end{equation}

\end{lemma}

\begin{proof}
At core, this is implied by Equation (25) in the proof of \cite[Theorem 5.12]{CGGLSS}, as follows. Let $\e \in \JWe$ and let $\e_\ne$ be the north-east edge incident to the trivalent vertex $\vertex_\e$. By this Equation (25), there exists a Laurent monomial $m$ in the cluster variables $\xWe_{\e'}$, with $\e' < \e$, such that we have the equality
\begin{equation}\label{eqn:cluster-variable-ne-edge}
\xWe_\e = f_{\e_\ne}\cdot m.
\end{equation}
    The proof of the lemma is deduced from \cref{eqn:cluster-variable-ne-edge} as follows. Consider a point $\borel_{\bullet} \in X(\brWe)$ with $\xWe_{\e}(\borel_{\bullet}) \neq 0$ for $\e < \c$ and $\xWe_{\c} = 0$: we want to deduce the conditions in \cref{eq:rel_position_condition_lemma}. First we need to show that $\borel_{(e)} \Rrel{w_\e^{\dstr}} \borel'_{(e)}$ for $e < c$. Note that \cref{eqn:cluster-variable-ne-edge} implies that $f_{\e_\ne}(\borel_{\bullet}) \neq 0$ for $\e < \c$, and that $B_i(z)B_i(z') \in \borel_+$ if and only if $z' = 0$. This implies that, labeling the top regions of $\triangle \setminus \DW$ with the flags $\borel_{\bullet}$ and propagating them through the other regions, the flags separated by the southern edge of the trivalent vertex $\vertex_{\e}$ will be distinct. The first required relative position condition in \cref{eq:rel_position_condition_lemma} then follows from \cref{lem:properties-rel-pos}(2). Let us now verify the second condition of \cref{eq:rel_position_condition_lemma}. Since $\xWe_{\c} = 0$ and $m$ is a monomial in cluster variables $\xWe_{\e}$ for $\e < \c$, \cref{eqn:cluster-variable-ne-edge} implies that $f_{\c_\ne} = 0$. This implies that the flags separated by the southern edge of $\vertex_{\c}$ are the same, and the distance between $\borel_{(\c)}$ and $\borel'_{(\c)}$ is strictly less than $w^{\dstr}_{\c}$, as required. This shows that the locus
$$\{\xWe_\e \neq 0 \mid \e \in \JWe, \e < \c\} \cap \{\xWe_\c = 0\}$$
is contained in the locus of points satisfying \cref{eq:rel_position_condition_lemma}. The reverse inclusion follows similarly.
\end{proof}

\begin{remark}
\cref{prop:dbl-weave-cluster-var} and \cref{lem:weave-var-zero-locus} are valid for double inductive weaves but, as stated, do not hold for a general weave. For a description of cluster variables in a general weave see \cite[Section 5.2]{CGGLSS}.\hfill$\Box$
\end{remark} 


\subsection{The cluster variables $\{\x3D_{\inve}\}$}\label{ssec:deodhar_variables1} Let $\dbr$ be a double braid word and $T_\dbr\sse R(\dbr)$ its Deodhar torus. The cluster variables $\{\x3D_\e\}$ are defined in terms of certain hypersurfaces in $R(\dbr)$, which are part of the complement $R(\dbr)\setminus T_\dbr$ of the Deodhar torus. For the indexing set, we need to introduce solid and hollow crossings, cf.~\cite[Section 2.3]{GLSB}. Recall the elements $w_0, \dots, w_{\nm}$ from \cref{{def:w-and-u}}.

\begin{definition}[Solid and hollow crossings]\label{def:almost-positive-sub}
Let $\dbr$ be a double braid word. An index $\e \in \{1, \dots, m+n\}$ is said to be
\begin{itemize}
    \item[(i)] a \emph{solid crossing} if $w_{\e-1}=w_\e$, i.e.~if the Demazure product stays the same through $e$.

    \item[(ii)] a \emph{hollow crossing} if $w_{\e-1}>w_\e$.
\end{itemize}
The set of solid crossings of $\dbr$ is denoted by $\J3D$.\hfill$\Box$

The cluster variables $\{\x3D_\e\}$ are indexed by $\J3D$, i.e.~there is a cluster variable per solid crossing of $\dbr$. Now, given a solid crossing $\e$ we define a sequence of elements $\ap_{m+n}, \dots \ap_{0}$ recursively by $\ap_{m+n}:= \id$ and
\begin{equation}
\ap_{\c-1}:=
\begin{cases}
    s_{i_{\c}^{\ast}}^{-} * \ap_\c * s_{i_{\c}}^{+} & \text{if } \c \neq \e,\\
    s_{i_{\c}^{\ast}}^{-}  \ap_\c  s_{i_{\c}}^{+} & \text{if } \c = \e.
\end{cases}
\end{equation}
\end{definition}

In words, the sequence of permutations $\ap_{m+n}, \dots ,\ap_{\e}$ agrees with the sequence of Demazure products $w_{m+n}, \dots ,w_{\e}$ and then at index $\e-1$ one makes a ``mistake" and takes the usual product (instead of the Demazure product), so that $\ap_{\e-1}< \ap_\e$. Then one proceeds by taking Demazure products to compute $\ap_{\e-2}, \dots, \ap_{0}$. This sequence of permutations names a sequence of relative positions that defines the following hypersurfaces of $R(\dbr)$:

\begin{definition}[Mutable Deodhar hypersurfaces]\label{def:mutable-Deodhar-hypersurface}
Let $\dbr$ be a double braid word and $\e$ be a solid crossing such that $\ap_0 = \wo$. By definition, the \emph{mutable Deodhar hypersurface} $V_\e \subseteq R(\dbr)$ is the Zariski closure of the locus
\[\{(\Xbul, \Ybul) \in R(\dbr) :Y_\c \Rwrel{\ap_\c} X_\c \quad \text{for all }c\in[0,m+n] \}\sse R(\dbr).\]
\hfill$\Box$
\end{definition}

The mutable cluster variables $\{\x3D_\e\}$ on $T_{\dbr}$ will be characters of $T_{\dbr}$ which vanish on exactly one mutable Deodhar hypersurface, and extend themselves to regular functions on $R(\dbr)$. In order to determine these characters exactly, rather than just up to units of $\C[R(\dbr)]$, we need to consider additional hypersurfaces, the frozen Deodhar hypersurfaces, which sit inside of a space different than $R(\dbr)$, as follows:

\begin{enumerate}
    \item Let $\bigRo(\dbr)\sse(G/U_+)^{[0,m+n]} \times (G/U_+)^{[0,m+n]}$ be the subset of points satisfying the relative position conditions from \cref{eq:GLSBS-definition}, so that the quotient of $\bigRo(\dbr)$ by the diagonal (free) $\G$-action is the double braid variety $R(\dbr)$. This is the variety discussed in \cref{ssec:coordinate_algebras}.(ii).\\

    \item Let $\tT_{\dbr} \subset \bigRo(\dbr)$ be the subset satisfying the relative position conditions from \cref{eq:deodhar-torus}, which is the preimage of $T_{\dbr}$ under the $\G$-quotient map $\bigRo(\dbr) \to R(\dbr)$.\\

    \item Let $\bigR(\dbr)$ denote the partial compactification of $\bigRo(\dbr)$ obtained by dropping the condition $X_0 \Lwrel{\wo} Y_0$.
\end{enumerate}

Note that $\bigRo(\dbr)$ has already been considered in \cref{ssec:coordinate_algebras}. The spaces $\bigRo(\dbr)$ and $\bigR(\dbr)$ suffice to name all necessary Deodhar hypersurfaces:

\begin{definition}[Deodhar hypersurfaces]\label{def:all-Deodhar-hypersurfaces}
Let $\dbr$ be a double braid word. For a solid index $e \in \J3D$, the \emph{Deodhar hypersurface} $\tV_{\e} \subset \bigR(\dbr)$ is the Zariski closure of the locus
    \[\{(\Xbul, \Ybul) \in \bigR(\dbr) :Y_\c \Rwrel{\ap_\c} X_\c \quad \text{for all }c\in[0,m+n] \}.\]
    The index $\e$ is said to be \emph{mutable} if $\ap_0 = \wo$, or equivalently if $\tV_{\e} \subset \bigRo(\dbr)$. The index $\e$ is said to be \emph{frozen} if $\ap_0 \neq \wo$, or equivalently if $\tV_{\e} \cap \bigRo(\dbr)= \emptyset$.\hfill$\Box$
\end{definition}

If $\e$ is mutable then the $\G$-action on $\tV_\e$ is free and $\tV_e/\G = V_e$ is the mutable Deodhar hypersurface defined in \cref{def:mutable-Deodhar-hypersurface}. By \cite[Proposition 2.19]{GLSB}, the Deodhar hypersurfaces are irreducible, codimension 1, and their union is $\bigR(\dbr) \setminus \tT_{\dbr}$.

\begin{remark}\label{rmk:where_functions_live}
By \cref{eq:double_braid_Ginvariant_description}, we can identify $\C[R(\dbr)] \cong \Gamma(\bigRo(\dbr), \mathcal{O}_{\bigRo(\dbr)})^{\G}$, i.e.~any regular function in $\C[R(\dbr)]$ defines a $\G$-invariant regular function on $\bigRo(\dbr)$, which in turn defines a $\G$-invariant rational function on $\bigR(\dbr)$.\hfill$\Box$
\end{remark}

In \cite[Section 2.5]{GLSB}, the authors construct a collection of functions $\minor_c \in \C[R(\dbr)]$ for $c \in [m+n]$, referred to as the \emph{chamber minors}, cf.~\cref{def:grid-minors} below, and show in \cite[Prop.~2.12.(2)]{GLSB} that the collection $\{\minor_e\}_{e \in \J3D}$ of chamber minors indexed by the solid crossings provides an isomorphism
\begin{equation}\label{eq:chamber_minor_trivialize_tori}
\{\minor_e\}_{e \in \J3D}:T_{\dbr} \xrightarrow{\sim} (\C^{\times})^{|\J3D|}.
\end{equation}
That is, the restrictions of the chamber minors indexed by $\J3D$ can be used as coordinates on the Deodhar torus $T_\dbr$. Via the isomorphism (\ref{eq:chamber_minor_trivialize_tori}), the character lattice $\Hom(T_{\dbr}, \C^{\times})$ of the Deodhar torus $T_{\dbr}$ has a basis given by the restriction of $\{\minor_e\}_{e\in\J3D}$. See \cref{ssec:deodhar_variables2} for more on chamber minors. By \cref{rmk:where_functions_live}, a character $\chi\in\Hom(T_\dbr,\C^\times)$ can be interpreted as a ($\G$-invariant) rational function on $\bigR(\dbr)$, and the isomorphism (\ref{eq:chamber_minor_trivialize_tori}) provides a specific identification for this. By definition, a character $\chi\in\Hom(T_\dbr,\C^\times)$ is said to vanish along a hypersurface in $\bigR(\dbr)$ if the corresponding Laurent monomial in the solid chamber minors $\{\minor_e\}_{e \in \J3D}$, seen as a rational function on $\bigR(\dbr)$, vanishes along the hypersurface. Such vanishing is crucial in the definition of the cluster variables in \cite{GLSB}, as it allows to name unique characters based on their vanishing along the Deodhar hypersurfaces from \cref{def:all-Deodhar-hypersurfaces}:

\begin{proposition}\label{prop:GLSB-cluster-var-def}
Let $\dbr$ be a double braid word and $\e \in \J3D$ the index of a solid crossing. Then there exists a unique character
$$\x3D_\e\in \Hom(T_{\dbr},\C^\times)$$
vanishing to order 1 on the Deodhar hypersurface $\tV_\e$ and not vanishing on $\tV_\c$ if $\c\neq \e$.
\end{proposition}
The collection of characters $\{\x3D_\e\}_{\e \in \J3D}$ from \cref{prop:GLSB-cluster-var-def} form the cluster of the seed with cluster torus $T_{\dbr}\sse R(\dbr)$ constructed in \cite{GLSB}. For reference, \cref{prop:GLSB-cluster-var-def} is \cite[Proposition-Definition 1.3]{GLSB}.

\subsection{The cluster variables $\xWe_{\e}$ and $\x3D_{\inve}$ are proportional}\label{subsec:init_cl_var_comp} The goal of this subsection is to establish that the cluster variables from \cite{CGGLSS} and \cite{GLSB,GLSBS} are proportional. Specifically, that $\iso^*(\xWe_{\e})$ is equal to a non-zero scalar multiple of $\x3D_{\inve}$, so that \cref{eq:equality_cluster_variables} holds up to constants.

\subsubsection{Bijection between indexing sets} The cluster variables $\xWe_{\e}$ from \cite{CGGLSS} are indexed by the vertex crossings in $\JWe$, cf.~\cref{def:solid-crossings-vs-vertex-crossings}. The cluster variables $\x3D_{\inve}$ from \cite{GLSB,GLSBS} are indexed by solid crossings $\J3D$, cf.~\cref{prop:GLSB-cluster-var-def}. \cref{lem:br-w-and-weave-w-coincide} readily implies following bijection between $\J3D$ and $\JWe$:

\begin{lemma}\label{lem:correspondence-solid-crossings-trivalent-vertices}
Let $\dbr$ be a double braid word with $n+m$ letters and $\inve:=m+n-e$. Then
$$inv:\J3D\lr\JWe,\quad e \longmapsto inv(e):=\inve,$$
is a bijection of sets. 
\end{lemma}

\begin{example}
\label{ex:running_sequences_of_elements_of_W}
Consider $\dbr = (-2, 1, 2, 1, -1, 1, 2)$ as in \cref{ex:running}, so that $n + m = 7$. From the viewpoint of \cite{GLSBS,GLSB}, the sequence of Demazure products reads $w_7 = e, w_6 = s_2$, $w_5 = s_2s_1 = w_4 = w_3$, $w_2 = s_2s_1s_2 = w_1 = w_0$ and the (indices for the) solid crossings are $\J3D = \{5, 4, 2, 1\}$. From the perspective of \cite{CGGLSS}, the associated double string is $\dstr(\dbr) = (2R, 1R, 1^{*}L, 1R, 2R, 1R, 2^{*}L)$ and thus we have $w_0^{\dstr} = e$, $w_1^{\dstr} = s_2$, $w_2^{\dstr} = s_2s_1 = w_3^{\dstr} = w_4^{\dstr}$, $w_5^{\dstr} = s_2s_1s_2 = w_6^{\dstr} = w_7^{\dstr}$. In particular, the indices for the trivalent vertices of the weave $\DW$ associated to $\dstr(\dbr)$ are given by $\JWe = \{2, 3, 5, 6\} = \{\overline{5}, \overline{4}, \overline{2}, \overline{1}\}$, as stated in \cref{lem:correspondence-solid-crossings-trivalent-vertices}.\hfill$\Box$
\end{example}

\subsubsection{Proportionality between cluster variables} Let us use the following notation:

\begin{definition}
    Let $\dbr$ be a double braid word and $\DW(\dbr)$ its double inductive weave.
    \begin{enumerate}
        \item $\bxWe := \{\xWe_\c\}$, $\c \in \JWe$, denotes the set of  cluster variables in $\C[X(\brWe)]$ for the weave torus $T_{\DW(\dbr)}$, as constructed in \cite{CGGLSS}.\\
        \item $\bx3D := \{\x3D_\e\}$, $\e \in \J3D$, denotes the set of cluster variables in $\C[R(\dbr)]$ for the Deodhar torus $T_\dbr$, as constructed in \cite{GLSBS,GLSB}.\hfill$\Box$
    \end{enumerate}
\end{definition}

To establish proportionality, we use that the variables $\bxWe$ and $\bx3D$ can be expressed as polynomials in the $z$-variables introduced in \cref{ssec:coordinate_algebras}.(i) and \cref{sec:parametrization}, for $X(\brWe)$ and $R(\dbr)$ respectively. This property is in fact inductive, in that the polynomial expressions do not change as we add crossings in certain directions:

\begin{lemma}\label{lem:polynomial-expressions-cluster-variables}
Let $\dbr$ be a double braid word.
\begin{enumerate}

    \item For every $\e \in \J3D$, there exists a polynomial expression
    $$\x3D_\e = \x3D_\e(z_1, \dots, z_{m+n})$$
    which is invariant under extending the double braid word $\dbr$ on the left.\\
    
    \item For every $\c \in \JWe$, there exists a polynomial expression
    $$\xWe_\c = \xWe_\c(z_1, \dots, z_{m+n})$$
    which is invariant under extending the double string $\dstr$ on the right (and thus extending the braid word $\brWe$ on either side). 
\end{enumerate}
\end{lemma}
\begin{proof}

For Part (1), \cite[Proposition 2.12]{GLSB} implies that every character of $T_{\dbr}$ may be expressed as a Laurent monomial in the chamber minors $\{\Delta_c\}_{c \in \J3D}$. Each chamber minor $\Delta_{\c+1}$ is a certain $(\U_+ \times \U_+)$-invariant generalized minor of $Z_\c := Y_\c^{-1} X_\c \in \U_+\backslash \G / \U_+$. Assuming $(X_\bullet, Y_\bullet) \in R(\dbr)$ is parametrized as in \cref{sec:parametrization}, $Z_\c$ depends only on the parameters $z_{\c+1}, z_{\c+2}, \dots, z_{m+n}$. Indeed, this follows by using the $\G$-action to multiply every representative flag by $g_{m+n}^{-1}$ on the left, does not affect $Z_\c$. The generalized minors of $Z_{\c}$, and in particular the chamber minor $\Delta_{\c+1}$, are polynomials in $z_{\c+1}, \dots, z_{m+n}$ and do not depend on the letters $i_1, \dots, i_{\c}$ of $\dbr$.

Now, the cluster variable $\x3D_{\e}$ is a character of $T_{\dbr}$. By the upper-triangularity property in \cite[Proposition 2.20]{GLSB}, $\x3D_\e$ can be written as a Laurent monomial in the chamber minors $\{\Delta_{\c}\}_{\c \in \J3D \cap [\e, n+m]}$. In particular, the exponents of each chamber minor are given by the entries in row $\e$ of an upper unitriangular matrix. 
By \cite[Lemma 2.25]{GLSB}, the entries in row $\e$ do not depend on the letters $i_1, \dots, i_{\e}$ of $\dbr$ and thus the expression for $\x3D_\e$ in terms of chamber minors only depends on the suffix $i_{\e+1} \dots i_{n+m}$ of $\dbr$. The expression for these chamber minors in terms of $z_{\e}, \dots, z_{m+n}$ in turn only depends on the suffix $i_{\e+1} \cdots i_{n+m}$. This gives an a priori rational expression for $\x3D_\e$ in terms of $z_{\e}, \dots, z_{m+n}$ which only depends on the suffix $i_{\e+1} \cdots i_{n+m}$. By \cite[Corollary 2.24]{GLSB}, $\x3D_\e$ lifts to a regular function on $\bigR(\dbr)$ and thus this expression for $\x3D_\e$ must be a polynomial in the $z$-variables, see \cref{eq:polynomial-algebra}.

Part (2) follows directly from the construction of cluster variables $\xWe_\c$, cf.~\cite[Section 7.2]{CGGLSS}, since extending the double string does not affect the top part of the double inductive weave $\DW$. 
\end{proof}

Let us establish the proportionality of the cluster variables $\x3D_{\e}$ and $\xWe_{\inve}$:

\begin{proposition}[Proportionality of cluster variables]
\label{lem:coincidence-cluster-variables-constants}
Let $\dbr$ be a double braid word and $\e \in \J3D$ the index of a solid crossing. Then $\exists\prop_\e \in \C^{\times}$ such that
\begin{equation}\label{eq:proportionality}
\x3D_{\e} = \prop_\e\cdot\xWe_{\inve}. 
\end{equation}
\end{proposition}

\begin{proof}
We use the isomorphism $\iso: R(\dbr) \xrightarrow{\sim} X(\brWe)$ from \cref{eq:iso-glsbs-cgglss} to identify the two algebraic varieties $R(\dbr)$ and $X(\brWe)$, and denote by $\iso^{\ast}: \C[X(\brWe)] \to \C[R(\dbr)]$ the corresponding pullback isomorphism between their rings of functions. To ease notation, we denote the pullbacks $\iso^{\ast}(\xWe_\c)$ of the weave cluster variables to $R(\dbr)$ by $\xWe_\c$. For clarity, we separate the proof in the following steps:\\

{\it Step 1.} By construction there is an inclusion $\C[R(\dbr)] \subseteq \C[\bigRofix(\dbr)]$, and the latter algebra $\C[\bigRofix(\dbr)]$ is isomorphic to a localization of the polynomial algebra $\C[z_1, \dots, z_\nm]$. By \cite[Theorem 1.3]{GLS}, the units in $\C[R(\dbr)]$ are precisely the monomials in frozen variables of $\{\x3D_c\}$. Therefore, we can directly work with the inclusion
\begin{equation}\label{eq:inside-polynomial-algebra}
\C[R(\dbr)]\subseteq \C[z_1, \dots, z_{m+n}][(\x3D_c)^{-1}\mid \x3D_c \; \text{is frozen}].
\end{equation}

{\it Step 2.} Both $\bx3D$ and $\bxWe$ are cluster variables for a seed in a cluster structure on the double braid variety $R(\dbr)$. By \cite[Definition 2.23]{GLSB} and \cite[Theorem 5.12]{CGGLSS}, together with Lemma \ref{lem:weave-tori-vs-deodhar-tori}, the associated cluster tori determined by these two clusters (in a priori different cluster structures) are both the Deodhar torus $T_{\dbr}$. Since cluster variables form a free generating set of units in the algebra $\C[T_{\dbr}]$, for every $\c \in \JWe$ there exists a Laurent monomial $m_{\c}(\bx3D)$ in the $\bx3D$-cluster variables such that
\begin{equation}\label{eq:weave-cluster-as-monomial}
\xWe_{\c} = m_{\c}(\bx3D).
\end{equation}
By Step 1, this equality can be understood as holding in the polynomial algebra
$$\C[z_1, \dots, z_{m+n}][(\x3D_{\c})^{-1} \mid \x3D_{\c} \; \text{is frozen}].$$

{\it Step 3.} Let $\c \in \JWe$ be such that $\xWe_\c$ is a mutable variable. By \cite[Theorem 3.1]{GLS}, the cluster variable $\xWe_\c$ is irreducible in $\C[R(\dbr)]$. Thus, in this case the monomial $m_\c$ must be of the form
\begin{equation}\label{eq:weave-cluster-as-monomial-mutable}
\xWe_\c = \x3D_{\phi(\c)}m'_\c(\bx3D),
\end{equation}
where $m'_\c(\bx3D)$ is a monomial involving only frozen variables in $\bx3D$ and $\phi$ is a bijection of the indexing sets. Note that if $\c \in \JWe$ is such that $\xWe_\c$ is a frozen variable, then the Laurent monomial $m_\c(\bx3D)$ must be a monomial in the frozen variables by \cite[Theorem 2.2]{GLS}.

{\it Step 4.}  We now argue that $m'_\c = \prop_\c \in \C^{\times}$ for $m'_c$ as in \cref{eq:weave-cluster-as-monomial-mutable}. For this, we use \cref{lem:polynomial-expressions-cluster-variables} as follows. Assume there is a frozen variable appearing in $m'_\c$. Clear denominators in \eqref{eq:weave-cluster-as-monomial-mutable} so that we have an equality
\begin{equation}\label{eq:denominators-cleared}
m_{c,1}(\bx3D)\xWe_{c} = m_{c,2}(\bx3D)\x3D_{\phi(\c)}
\end{equation}
that is now valid in the polynomial algebra $\C[z_1, \dots, z_{m+n}]$. Extend the double braid word on the left so that all frozen variables appearing in \eqref{eq:denominators-cleared} become mutable; this can be done, for example, by appending a reduced expression for $\wo$ on the left of $\dbr$. By \cref{lem:polynomial-expressions-cluster-variables}, the expression \eqref{eq:denominators-cleared} is still valid for the longer double braid word $\widetilde{\dbr}$. The factoriality of $\C[R(\widetilde{\dbr})]$, cf.~\cite[Lemma 5.29]{CGGLSS}, and the irreducibility of cluster variables imply that $m_{c,1}(\bx3D)$ is proportional to $m_{c,2}(\bx3D)$, and so are $\xWe_{\c}$ and $\x3D_{\phi(\c)}$.\footnote{Note that the cluster variable $\xWe_{\c}$ cannot be a factor of $m_{c,2}(\x3D)$ since $\xWe_{\c}$ is not frozen in $\C[R(\dbr)]$. Similarly, $\x3D_{\varphi(\c)}$ cannot be a factor of $m_{c,1}(\bx3D)$.}  Thus, $\xWe_\c = \prop_\c\x3D_{\phi(\c)}$ for some constant $\prop_\c \in \C^{\times}$. 

{\it Step 5.} Let us now verify that $\phi(\c) = \invc$, i.e.~that the indexing bijection $\phi$ is indeed given by the inversion $c\mapsto\invc$. Consider the locus $\{\xWe_{\e} \neq 0 | \e < \c\}\cap\{\xWe_{\c} = 0\}$. By \cref{lem:weave-var-zero-locus}, this locus coincides with the set of elements $(\Xbul, \Ybul) \in R(\dbr)$ satisfying 
\[
Y_{\inve} \Rwrel{w_{\inve}} X_{\inve} \; \text{for} \; \e < \c \; \text{but} \; Y_{\invc} \Rwrel{w} X_{\invc} \; \text{with} \; w < w_{\invc}.  
\]
We note that the only possibility for $w$ above is $s_{i_{\invc}^*}^-w_{\invc}s_{i_{\invc}}^+$, by the properties of relative position. It thus follows that the locus $\{\xWe_{\e} \neq 0 | \e < \c\}\cap\{\xWe_{\c} = 0\}$ contains the locally closed subset from \cref{def:mutable-Deodhar-hypersurface}, whose closure is the Deodhar hypersurface $V_{\invc}$, and thus the vanishing locus of $\xWe_\c$ contains the Deodhar  hypersurface $V_{\invc}$. By \eqref{eq:weave-cluster-as-monomial-mutable}, the vanishing locus of $\xWe_\c$ is exactly the same as the vanishing locus of $\x3D_{\phi(c)}$. By \cref{prop:GLSB-cluster-var-def}, of all the cluster variables in $\bx3D$ only $\x3D_{\invc}$ vanishes on $V_{\invc}$. Therefore $\phi(\c) = \invc$, as required. 

{\it Step 6.} We have shown that $\xWe_{\c} = \prop_\c\cdot \x3D_{\invc}$ with $\prop_\c\in\C^\times$ for $\c \in \JWe$ such that $\xWe_{\c}$ is mutable. If $\xWe_{\c}$ is frozen, extend the double braid word (and thus the double inductive weave $\DW$) until $\xWe_{\c}$ becomes mutable, and apply the arguments above. This finishes the proof. 
\end{proof}

\subsection{Towards equality of cluster variables}
The current goal is to show that $\prop_\e = 1$ for every $\e \in \JWe$, where $\prop_\e\in\C^\times$ are the constants of proportionality in \cref{lem:coincidence-cluster-variables-constants}. That is, we want to show that the Deodhar and weave cluster variables $\bx3D$ and $\bxWe$ coincide identically, not just up to constants. For this, we need to further study the combinatorics of weave and Deodhar cluster variables and their relation. 

For the remainder of this section, we let $\dbr=(i_1,i_2,\dots,i_\nm) = i_1i_2\cdots i_\nm \in (\pm I)^\nm$ be a double braid word,
$J_- = \{a_1<a_2<\cdots<a_n\}\quad\mbox{and}\quad J_+ = \{b_1<b_2<\cdots<b_m\}$
be the subsets of $[\nm]$ recording positions of negative and positive letters in $\dbr$, and
$$\brWe = \ip_1\ip_2\cdots \ip_\nm\in I^\nm$$
be the (ordinary) braid word corresponding to $\dbr$, cf.~ \cref{lem:iso-glsbs-cgglss}. Let
$$\doublestring(\dbr) = (j_1A_1,j_2A_2,\dots, j_\nm A_\nm)$$ be the associated double string and let $\DW(\dbr)$ be the double inductive weave, cf.~\cref{sssec:double_inductive_weaves}.

\subsection{Combinatorics of weave cluster variables}\label{sec:weave-cluster-variables2} 
From \cref{def:edge-labeling} we have the edge labeling of the double inductive weave $\DW$ and the rational functions $u_\edge$ involved in the edge label $g_\edge$. By \cref{prop:dbl-weave-cluster-var}, for the southern edge $\edge$ of a trivalent vertex, the function $u_\edge$ is equal to $\xWe_\c$, for a certain index $c$. By \cite[Theorem 5.12 (i)]{CGGLSS}, for all edges $\edge$, $u_\edge$ is a product of cluster variables. In this section, we review the combinatorics of this product formula for $u_\edge$ and how the cluster variables spread along the weave edges top-to-bottom. This scanning procedure is essential in our comparison of the cluster variables $\xWe$ to the cluster variables $\x3D$.

Recall that if $\weave$ is a weave, then $E(\weave)$ denotes the set of edges of the graph $\weave$. The following notion is a key aspect of \cite{CGGLSS}:

\begin{definition}[Lusztig cycles]\label{def:lusztig-cycle}
Let $\weave$ be a weave. A function $C: E(\weave) \to \Z_{\geq0}$ is said to be \emph{Lusztig cycle} if it satisfies the \emph{tropical Lusztig rules}. Specifically, at $6$, $4$ and $3$-valent vertices the tropical Lusztig rules are given by the following local models: 
	
\begin{equation}\label{fig:tropical-lusztig-rules}
\begin{tikzcd}
\draw[dashed] (0,0) -- (3,0) -- (3,3) -- (0,3)-- cycle;
\draw[red] (0.5, 3) to[out=270, in=135] (1.5, 1.5);
\node at (0.3, 2.5) {a_1};
\draw[blue] (1.6, 3) to (1.5, 1.5);
\node at (1.3, 2.5) {a_2};
\draw[red] (2.5, 3) to [out=270, in=45] (1.5, 1.5);
\node at (2.7, 2.5) {a_3};
\draw[blue] (1.5, 1.5) to [out=225, in=90] (0.5, 0);
\node at (0.3, 0.5) {b_1};
\draw[red] (1.5, 1.5) to (1.5, 0);
\node at (1.3, 0.5) {b_2};
\draw[blue] (1.5, 1.5) to[out=315, in=90] (2.5, 0);
\node at (2.7, 0.5) {b_3};

\draw[dashed] (5,3)--(8,3)--(8,0)--(5,0)--cycle;
\draw[red] (5.5, 3) to[out=270, in=90] (7.5, 0);
\node at (5.3, 2.5) {{a_1}};
\draw[teal] (7.5, 3) to[out=270, in=90] (5.5, 0);
\node at (7.7, 2.5) {{a_2}};
\node at (5.3, 0.5) {a_2};
\node at (7.7, 0.5) {a_1};

\draw[dashed] (10,3)--(13,3)--(13,0)--(10,0)--cycle;
\draw[red] (10.5, 3) to[out=270, in=135] (11.5, 1.5);
\node at (10.3, 2.5) {{a_1}};
\draw[red] (12.5, 3) to[out=270, in=45] (11.5, 1.5);
\node at (12.7, 2.5) {{a_2}};
\draw[red] (11.5, 1.5) to (11.5, 0);
\node at (11.2, 0.5) {\min(a_1, a_2)};
\end{tikzcd}
\end{equation}

where the weights $a_i$ and $b_j$ in the 6-valent case satisfy the equalities:
\begin{enumerate}

\item $b_1 = a_2 + a_3 - \min(a_1, a_3)$.
\item $b_2 = \min(a_1, a_3)$.
\item $b_3 = a_2 + a_1 - \min(a_1, a_3)$. 
\end{enumerate}
\end{definition}

More generally, at $8$- and $12$-valent vertices the rules are obtained by tropicalizing the formulas in \cite[Theorem 3.1]{BZ-positivity}, see e.g. \cite[(36), (37)]{CGGLSS} for the tropical Lusztig rules at an $8$-valent vertex.\hfill$\Box$

Let us fix a double inductive weave $\DW$ for $\brWe$. For each cluster variable $\xWe_\c$, we now introduce a function $\lcyc_\c: E(\DW) \to \Z_{\geq0}$ which tracks the exponent of $\xWe_\c$ in $u_\edge$. Recall from \cref{def:solid-crossings-vs-vertex-crossings} that if $\c \in \JWe$, the unique trivalent vertex between depths $\c$ and $\c+1$ 
is denoted $\vertex_\c$. The function $\lcyc_\c$ is defined as follows:

\begin{definition}[Vertex cycles]\label{def:cycles-for-trivalents}
Let $\DW$ be a double inductive weave for $\brWe$, $\c \in \JWe$ and $\edge$ the southern edge of the vertex $\vertex_\c$ in $\DW$. By definition, the \emph{vertex cycle} $\lcyc_\c$ is the function $\lcyc_\c:E(\DW) \to \Z_{\geq0}$ uniquely characterized by
\begin{enumerate}
    \item $\lcyc_\c(\edge)=1$,
    \item $\lcyc_\c$ vanishes on all other edges which intersect the region above depth $\c+1$,
    \item $\lcyc_\c$ propagates down from $\vertex_\c$ by the tropical Lusztig rules (cf. Definition \ref{def:lusztig-cycle}) for edges that lie entirely below depth $\c+1$.\hfill$\Box$
\end{enumerate}

\end{definition}

The following result asserts that the collection of vertex cycles, as introduced in \cref{def:cycles-for-trivalents}, exactly record the factorization of $u_\edge$ into cluster variables. Note that this fact is true for general weaves, though we state it here only for double inductive weaves:

\begin{thm}[{\cite[Theorem 5.12 (i)]{CGGLSS}}]\label{thm:u-as-cluster-var-product}
    Let $\DW$ be a double inductive weave for $\brWe$. In the edge labeling of \cref{def:edge-labeling}, for any $\edge \in E(\DW)$, we have
    \[u_\edge = \prod_{\c \in \JWe} (\xWe_{\c})^{\lcyc_{\c}(\edge)}.\]
\end{thm}

\subsection{Deodhar cluster variables and grid minors}\label{ssec:deodhar_variables2} In this section we discuss the torus element $\hp_\c$ of \cite{GLSB}, and further comment on the chamber minors, which briefly appeared in \cref{ssec:deodhar_variables1} when defining the Deodhar cluster variables $\bx3D$. Each element $\hp_\c$ belongs to the Cartan torus, and the chamber minors are certain characters on the torus that can be evaluated on $\hp_\c$. Since the chamber minors are Laurent monomials in the cluster variables $\bx3D = (\x3D_\e)_{\e\in\J3D}$, 
$\hp_\c$ can be understood as a certain cocharacter evaluated on $\bx3D$. Intuitively, in analogy with weaves, the Deodhar cluster variables $\bx3D$ spread through the chamber minors analogous to the way the weave cluster variables $\bxWe$ spread through the weave; this analogy is made more precise in \cref{ssec:comparing-cocharacters}, see also \cref{ssec:compare-cocharacters-type-A}. The chamber minors and $\hp_c$ are rigorously introduced as follows.\\

Fix a point $(X_{\bullet}, Y_\bullet) \in R(\dbr)$ in the double braid variety. For $\c \in [0, n+m]$, we define the coset $Z_\c:=Y_\c^{-1}X_\c\in U_+\bs \G/U_+$ and, to ease notation, we use $Z_\c$ to also denote a representative of such double coset in $G$. Following~\cite[Equation~(2.6)]{GLSB}, for $(X_{\bullet}, Y_\bullet) \in T_\dbr$, we define $\hp_\c\in\H$ to be the unique torus element such that 
\begin{equation}\label{eq:hp_c_dfn}
  Z_\c\in U_+ \dw_\c \hp_\c U_+.
\end{equation}

We now define certain functions of $\hp_\c$, which extend to regular functions on $R(\dbr)$, in \cref{def:grid-minors}. Set $u_c:= \wo w_c$, where $w_\c$ is as in \cref{def:w-and-u}, use $\omega_i$ to denote the fundamental weights. If $w  = s_{i_1}\cdots s_{i_{\ell}}\in W$, we define the group element
\[
\overline{\overline{w}}:=\dot{s}_{i_1}^{-1} \cdots \dot{s}_{i_\ell}^{-1}\in \G.
\]

\begin{definition}[Grid and chamber minors]\label{def:grid-minors}
    Let $(X_{\bullet}, Y_\bullet) \in T_\dbr$, $\c \in [0, m+n]$, and $i \in I$. By definition, the \emph{grid minors} of the point $(X_{\bullet}, Y_\bullet)$ are
    \[\grid_{\c, i}:= \omega_i(\hp_c) \quad \text{and}\quad \grid_{\c, -i}:=\omega_i(\overline{\overline{u_c}} \hp_\c \overline{\overline{u_c}}^{-1}).\]
    For $\c \in [m+n]$, the \emph{chamber minor} is defined to be the grid minor $\grid_\c:= \grid_{c-1, i_c}$, and it is said to be \emph{solid} whenever $c$ is solid.\hfill$\Box$
\end{definition}
The grid minors in \cref{def:grid-minors} are defined on the Deodhar torus $T_\dbr\sse R(\dbr)$ and \cite[Lemma 2.17]{GLSB} implies that they extend to regular functions on $R(\dbr)$ and $\G$-invariant regular functions on $\bigR(\dbr)$. In particular, they can be written as polynomials in the parameters $z_i$ from \cref{sec:parametrization}.\footnote{The regularity of the grid minors follows from the fact that they are  \emph{generalized minors} of $\G$ evaluated at the element $Z_c$, cf. \cite[Section 2.5]{GLSB}.}

The important features of the grid and chamber minors, which follow from \cite[Prop. 2.12 \& Cor. 2.22]{GLSB}, can be summarized as follows:

\begin{lemma}[\cite{GLSB}]\label{lem:impt-facts-minors} The following holds:
\begin{enumerate}
    \item The grid minors are characters on the Deodhar torus $T_{\dbr}$.

    \item The collection $\{\grid_\c\}_{\c \in \J3D}$ of solid chamber minors is a basis of characters of $T_{\dbr}$.

    \item  The collection of cluster variables $\bx3D$ is related to the chamber minors $\{\grid_\c\}_{\c \in \J3D}$ by an upper unitriangular transformation.
\end{enumerate}
\end{lemma}

\noindent The most important fact is \cref{lem:impt-facts-minors}.(3), which allows for the explicit computation of the cluster variables $\bx3D$ via Laurent monomials on the chamber minors. \cref{lem:impt-facts-minors}, together with the definition of $\bx3D$ and the regularity of grid minors and cluster variables on $R(\dbr)$, imply the following lemma:

\begin{lemma}\label{lem:grid-as-mono-in-cluster}
    Let $\c \in [0, m+n]$ and $i \in \pm I$. Then the grid minor $\grid_{\c, i}$ relates to $\bx3D$ via the formula
    \[\grid_{\c, i}= \prod_{\e \in \J3D} (\x3D_\e)^{\ord_{V_e}(\c,i)}\]
    where $\ord_{V_e}(\c,i) \in \mathbb{Z}_{\ge 0}$ is the order of vanishing of $\grid_{\c,i}$ on the Deodhar hypersurface $\tV_\e$.
\end{lemma}

\noindent In a certain sense, \cref{lem:grid-as-mono-in-cluster} can be understood as the analogue of \cref{thm:u-as-cluster-var-product} in the Deodhar approach.

\subsection{Deodhar torus element in terms of weave $u$-variables} \label{ssec:comparing-cocharacters} Let us now compare the data introduced in \cref{sec:weave-cluster-variables2}, for weaves, and in \cref{ssec:deodhar_variables2}, for Deodhar tori. The main contribution of this section is \cref{lem:hc-in-terms-of-uc}, which expresses the torus element $\hp_c$ for a Deodhar torus, as defined by \cref{eq:hp_c_dfn}, in terms of the $u$-variables associated to the corresponding weave.

\begin{figure}
    \centering
    \adjustbox{scale=0.7}{
    
\begin{tikzcd}
\draw (-1,-1) to (17,-1) to (8,8) to (-1,-1);

\draw[color=blue] (8.5, 7.5) to[out=270, in=45] (6.5, 5);
\draw[color=red] (9.5, 6.5) to[out=270,in=135] (10.5, 4);
\draw[color=blue] (5.5,5.5) to[out=270,in=135] (6.5, 5) to (6.5, 4.5) to [out=270,in=160] (10.5, 2.5);
\draw[color=red] (11.5, 4.5) to[out=270,in=45] (10.5, 4) to (10.5, 2.5);
\draw[color=blue] (12.5, 3.5) to[out=270, in=20] (10.5, 2.5);

\draw[color=blue] (10.5, 2.5) to (10.5,-1);

\draw[color=red] (10.5, 2.5) to[out=0,in=160] (12, 1.5) to (12, -1);

\draw[color=red] (13.5, 2.5) to[out=270, in=45] (12, 1.5);

\draw[color=red] (0.9, 0.9) to[out=340, in=170] (6, 0.5);

\draw[color=red] (10.5, 2.5) to[out=190,in=20] (6, 0.5) to (6,-1);

\draw[dashed] (0,0) to (16,0);

\draw[dashed] (1.4,1.4) to (14.6,1.4);

\draw[dashed] (2.8, 2.8) to (13.2, 2.8);

\draw[dashed] (3.8, 3.8) to (12.2, 3.8);

\draw[dashed] (4.8, 4.8) to (11.2, 4.8);

\draw[dashed] (6, 6) to (10, 6);

\draw[dashed] (7,7) to (9,7);

\draw[dashed] (7.7, 7.7) to (8.3, 7.7);

\node at (1, -0.5) {Y_0};

\node at (15, -0.5){X_0};

\node at (3, 1) {Y_1};
\node at (13.5, 1) {X_1};

\node at (4, 2.3) {Y_2};
\node at (12.5, 2.3) {X_2};

\node at (5, 3.3) {Y_3};
\node at (11.5, 3.3) {X_3};

\node at (5.5, 4.3) {Y_4};
\node at (10.75, 4.35) {X_4};

\node at (6, 5.5)  {Y_5};

\node at (10.1, 5.5)  {X_5};

\node at  (7,6.5)  {Y_6};

\node at (9, 6.5)  {X_6};

\node at (8, 7.2) {Y_7 = X_7};

\end{tikzcd} 
    }
    \caption{The labels on the regions $L_c, R_c$, $\c\in[0,7]$ for our running example, cf.~Examples \ref{ex:running_variables}, \ref{ex:from-double-braid-to-double-string-running-example} and \ref{ex:running_sequences_of_elements_of_W}. The dashed lines are at 
    depth $\c$, counting from top to bottom, as $Y_\invc$ labels $L_\c$ and $X_\invc$ labels $R_\c$.}
    \label{fig:ind-weave-triangle-labels}
\end{figure}
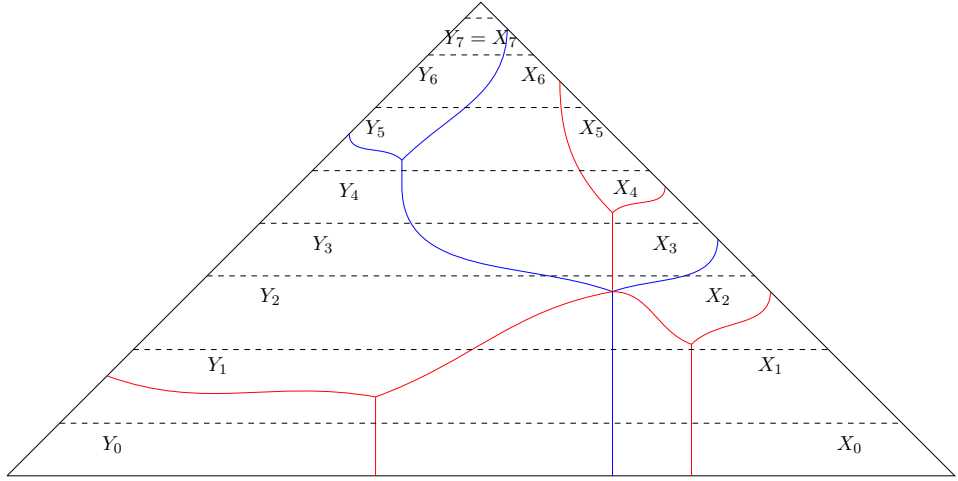

For that, we fix the double inductive weave $\DW$ associated to the Deodhar torus $T_\dbr$. We first determine how to compute $Z_{\invc}$ from the slice of $\DW$ at depth $\c$, as follows. Applying the isomorphism $\iso$ from \cref{lem:iso-glsbs-cgglss}, the top regions of $\triangle \setminus \DW$ are labeled by the weighted flags $(\Xbul, \Ybul)$, in some order. Recall from \cref{rem:iso-on-triangle-weave} that $L_{\c}$ is the region labeled by $Y_{\invc}$ and $R_{\c}$ is the region labeled by $X_{\invc}$, see e.g.~\cref{fig:ind-weave-triangle-labels}.

\begin{figure}
\adjustbox{scale=0.7}{
\begin{tikzcd}
\draw (-1,-1) to (17,-1) to (8,8) to (-1,-1);

\draw[color=blue] (8.5, 7.5) to[out=270, in=45] (6.5, 5);
\draw[color=red] (9.5, 6.5) to[out=270,in=135] (10.5, 4);
\draw[color=blue] (5.5,5.5) to[out=270,in=135] (6.5, 5) to (6.5, 4.5) to [out=270,in=160] (10.5, 2.5);
\draw[color=red] (11.5, 4.5) to[out=270,in=45] (10.5, 4) to (10.5, 2.5);
\draw[color=blue] (12.5, 3.5) to[out=270, in=20] (10.5, 2.5);

\draw[color=blue] (10.5, 2.5) to (10.5,-1);

\draw[color=red] (10.5, 2.5) to[out=0,in=160] (12, 1.5) to (12, -1);

\draw[color=red] (13.5, 2.5) to[out=270, in=45] (12, 1.5);

\draw[color=red] (0.9, 0.9) to[out=340, in=170] (6, 0.5);

\draw[color=red] (10.5, 2.5) to[out=190,in=20] (6, 0.5) to (6,-1);

\node (Y4) at (5, 4.35) {Y_4};
\node (X4) at (11, 4.35) {X_4};

\draw[->, ultra thick, color=cyan] (Y4) to (X4);
\draw[ultra thick, color=purple, ->,rounded corners=10] (Y4) -- (8, 7.5) -- (X4);

\node at (7.2, 6) {\color{purple} P_3^{\uparrow}};

\node at (8, 3.8) {\color{cyan} P_3^{\downarrow}};

\end{tikzcd} 

}
  \caption{\label{fig:ind-weave-triangle-paths} A double inductive weave corresponding to the running \cref{ex:running_variables}, shown together with the paths $\Ptop_3,\Pbot_3$ from \cref{def:highest-lowest-paths}. In this case $\c=3$ and $\invc=4$, so $\Ptop_3,\Pbot_3$ start in the region $L_\c=L_3$, labeled by $Y_\invc=Y_4$, and end in the region $R_\c=R_3$, labeled by $X_\invc=X_4$. Additional flag labels $Y_\invc,X_\invc$ are depicted in \cref{fig:ind-weave-triangle-labels}.}
\end{figure}
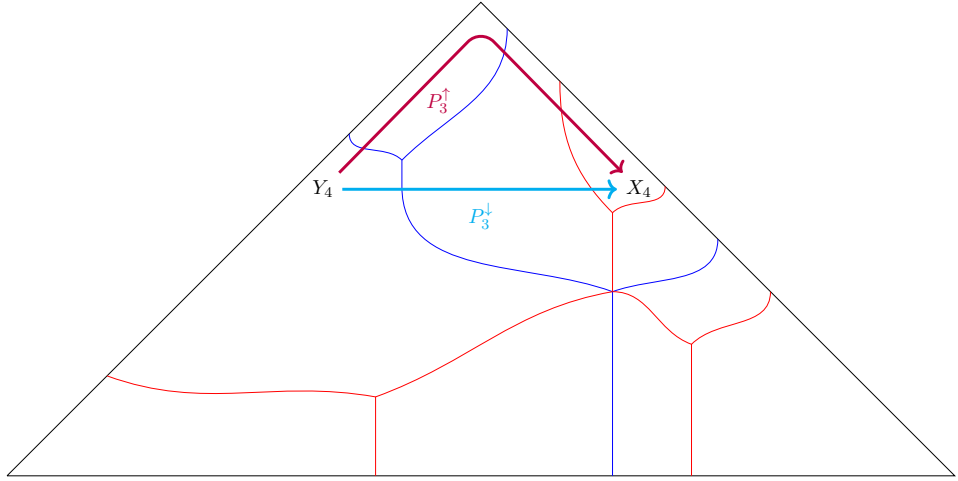

Let us now construct, using the weave $\DW$, an element  $\gPtop_\c\in\G$ which will match the coset (representative) $Z_{\invc}$ used in \cref{ssec:deodhar_variables2}:

\begin{definition}\label{def:highest-lowest-paths}
In the notation above, $\Ptop_\c$ is defined to be the highest possible path in the weave $\DW$ from the flag region $L_{\c}$ to the flag region $R_{\c}$: it is obtained by moving left-to-right and passing through only the top regions of $\triangle \setminus \DW$, as in \cref{fig:ind-weave-triangle-paths} (green). By definition, $\Pbot_\c$ is the path in the weave $\DW$ from $L_{\c}$ to $R_{\c}$ that is a horizontal line at depth $\c$, as in Figure \ref{fig:ind-weave-triangle-paths} (teal). Suppose that the path $\Ptop_\c$ intersects weave edges $\edge_1, \dots, \edge_r$, in this order. By definition, the element $\gPtop_\c$ is
\[\gPtop_\c:=g_{\edge_1} \cdots g_{\edge_r},\]
\noindent where $g_\edge$ denotes the label of $\edge$ in the edge labeling of $\DW$, cf.~\cref{def:edge-labeling}. The element $\gPbot_\c$ is defined analogously.\hfill$\Box$
\end{definition}

\begin{lemma}\label{lemma:Z_c_agrees_with_gPbot_c}
For all $c\in[0,m]$, we have the equality
$$Z_{\invc}=\gPtop_\c.$$ 
In addition, $Z_{\invc}$ agrees with $\gPbot_\c$ up to right multiplication by $U_+$.
\end{lemma}

\begin{proof}
    To show that $Z_{\invc}=Y_{\invc}^{-1}X_{\invc}$ equals $\gPtop_\c$, use the parametrization of $X_{\invc}, Y_{\invc}$ from \cref{sec:parametrization} to express $Z_{\invc}$ as a product of elements $B_i(z_j)$. By \cref{def:edge-labeling}, the edge labels of the top edges of the weave are exactly the elements $B_i(z_j)$ appearing in the parametrization, and thus the equality follows. The fact that $Z_{\invc}$ agrees with $\gPbot_\c$ up to right multiplication by $U_+$ follows from the equality $Z_{\invc}=\gPtop_\c$ and proof of Lemma 5.9 in \cite{CGGLSS}, specifically the second paragraph of it.
\end{proof}

Let $\bj=s_{j_1} \cdots s_{j_\ell}$ be a reduced expression for $w \in W$ and consider the following sequence of coroots $\chi^{\bj}_i:\C^\times\lr \torus$, which are the \emph{inversions} of $w$:
\begin{equation}\label{eq:coroots-for-red-expr}
    \chi^{\bj}_\ell := \chi_{j_\ell}, \quad \chi^{\bj}_{\ell-1} := s_{j_{\ell}} \cdot\chi_{j_{\ell-1}}, \quad \dots \quad, \chi^{\bj}_{1} := s_{j_{\ell}} s_{j_{\ell-1}} \cdots s_{j_2}\cdot\chi_{j_{1}}.
\end{equation}
Now we can use \cref{lemma:Z_c_agrees_with_gPbot_c} to express $\hp_c$ in terms of the $u$-variables in the edge labeling of the weave $\DW$:

\begin{proposition}\label{lem:hc-in-terms-of-uc}
In the notation above, suppose that the weave path $\Pbot_\c$ intersects weave edges $\edge_1, \dots, \edge_l$, the corresponding reduced expression for $w_\c^{\dstr}=w_{\invc}$ is $\bj=s_{j_1} \dots s_{j_\ell}$ and let $\chi^{\bj}_\ell, \dots, \chi^{\bj}_{1}$ be the inversions of $w_{\invc}$, as in \eqref{eq:coroots-for-red-expr}. Then the following equality holds:
   \begin{equation}\label{eq:hp_in_terms_of_u-s}
   \hp_{\invc}= \chi^{\bj}_1(u_{\edge_1}) \cdot \chi^{\bj}_2(u_{\edge_2}) \cdots \chi^{\bj}_\ell(u_{\edge_\ell}).\end{equation}
\end{proposition}
\begin{proof} Consider the product 
    \[\gPbot_c=g_{\edge_1} \cdots g_{\edge_\ell}= \BZ_{j_1}(f_{\edge_1}) \chi_{j_1}(u_{\edge_1}) \cdots \BZ_{j_\ell}(f_{\edge_\ell}) \chi_{j_\ell}(u_{\edge_\ell}) \]
    and recall that $\BZ_j(f) = x_j(f) \ds_j$. Since the word $\bj$ is reduced, we can move all factors $x_j(f)$ to the left, conjugating them by elements of the form $\chi_j(u)$ and $\ds_j$ along the way, and they will give an element of $U_+$. We can simultaneously move all factors $\chi_j(u)$ to the right and they will get conjugated by factors $\ds_j$. Thus, we have
    \[\gPbot_\c= g' \cdot \dot{w_\invc} \cdot \chi^{\bj}_1(u_{\edge_1}) \cdot \chi^{\bj}_2(u_{\edge_2}) \cdots \chi^{\bj}_\ell(u_{\edge_\ell})\]
    where $g' \in U_+$. 
    By \cref{lemma:Z_c_agrees_with_gPbot_c}, $\gPbot_\c$ differs from $Z_{\invc}$ only by right multiplication by $U_+$. So, the uniqueness of $\hp_{\invc}$ in the factorization~\eqref{eq:hp_c_dfn} implies the required equality \eqref{eq:hp_in_terms_of_u-s}.
\end{proof}

\begin{example}
Let us compute these elements for our running \cref{ex:running}, cf.~Examples \ref{ex:running-iso},  \ref{ex:running_variables}, and \ref{ex:running_sequences_of_elements_of_W}. Applying the computation from \cref{ex:running_variables} to $\c=3$, we obtain
\begin{equation*}
X_{4}=\smat{z_{1} z_{6} - z_{5} z_{7} + 1 & -z_{1} & z_{5} \\
z_{6} & -1 & 0 \\
z_{7} & 0 & -1} U_+,\quad 
Y_{4}=\smat{z_{1} & -1 & 0 \\
1 & 0 & 0 \\
0 & 0 & 1}U_+,\quad
Z_{4} = Z(\Ptop_3)=\smat{z_{6} & -1 & 0 \\
z_{5} z_{7} - 1 & 0 & -z_{5} \\
z_{7} & 0 & -1}. 
\end{equation*}
By~\eqref{eq:hp_c_dfn}, we thus have the following expression for the Cartan element at $\invc=4$:
\begin{equation}\label{eq:ex:hp_4}
  \hp_4 =\smat{z_{7} & 0 & 0 \\
0 & 1 & 0 \\
0 & 0 & \frac{1}{z_{7}}}.
\end{equation}
Let us independently compute the right hand side of \cref{eq:hp_in_terms_of_u-s} to verify the equality. The weave path $\Pbot_3$ in \cref{fig:ind-weave-triangle-paths} intersects a blue edge and a red edge, which gives $w_\invc = s_2s_1$ 
and $(\jc_1,\jc_2) = (2,1)$. By~\cite[Def.~5.8 \& Lemma 5.9]{CGGLSS}, we have
\begin{equation*}
  \gPbot_3 = \BZop_2\left(\frac{z_5z_7 - 1}{z_7}\right) \chi_2(z_7) \cdot \BZ_1(z_6) \chi_1(1) = 
\smat{z_{6} & -1 & 0 \\
z_{5} z_{7} - 1 & 0 & -\frac{1}{z_{7}} \\
z_{7} & 0 & 0}.
\end{equation*}
First, this equality illustrates directly that $\gPbot_3$ differs from $Z_4$ by right multiplication by $U_+$, in agreement with \cref{lemma:Z_c_agrees_with_gPbot_c}. Second, we observe from~\eqref{eq:ex:hp_4} that  $\hp_4 = (\alpha_1+\alpha_2)(z_7)$, whereas the right hand side of \eqref{eq:hp_in_terms_of_u-s} reads
\begin{equation*}
  \hp_\c = (s_1(\alpha_2))(\uc_1)\cdot  \alpha_1(\uc_2) = (\alpha_1+\alpha_2)(\uc_1)\cdot  \alpha_1(\uc_2) = (\alpha_1+\alpha_2)(z_7),
\end{equation*}
since $\uc_1 = z_7$ and $\uc_2 = 1$. This illustrates the equality \eqref{eq:hp_in_terms_of_u-s} in \cref{lem:hc-in-terms-of-uc}.\hfill$\Box$
\end{example}

\subsection{Coincidence of cluster variables}\label{ssec:equality_variables} Let us now upgrade \cref{lem:coincidence-cluster-variables-constants} to an equality where the proportionality constants are identically one. The following is one of our main results:

\begin{thm}[Equality of cluster variables]\label{thm:coincidence-cluster-variables}
Let $\dbr$ be a double braid word, $\C[R(\dbr)]$ the ring of regular functions of its associated double braid variety and $\bx3D$ the Deodhar cluster variables of the Deodhar torus $T_\dbr$. Consider the isomorphism $\iso: R(\dbr) \xrightarrow{\sim} X(\dbr^{(-|+)})$ in \cref{eq:iso-glsbs-cgglss} and the weave cluster variables $\bxWe$ associated to the double inductive weave of $\dbr$. Then, for every $\e \in \JWe$, we have the equality
\[
\iso^{\ast}(\xWe_{\e}) = \x3D_{\inve}. 
\]
\end{thm}
\begin{proof}
As before, we work in the algebra $\C[R(\dbr)]$ and use $\xWe_\e$ to denote $\iso^{\ast}(\xWe_\e)$. By Lemma \ref{lem:coincidence-cluster-variables-constants}, there exists a constant $\prop_\e\in\C^\times$ such that $\xWe_{\e} = \prop_\e\cdot\x3D_{\inve}$, and we need to show that $\prop_e = 1$. It is enough to show $\prop_e = 1$ when we restrict to the open Deodhar torus $T_{\dbr}\sse R(\dbr)$ and thus, to ease notation, we denote by $\xWe_{\e}, \x3D_{\inve}$ the restriction of these functions to $T_{\dbr}$.

The weave cluster variables $\{\xWe_{\e} \mid \e \in \JWe\}$ give an explicit isomorphism
$$T_{\dbr} \xrightarrow{\sim} (\C^{\times})^{\left|\JWe\right|},$$
thus endowing $T_{\dbr}$ with a multiplicative structure in such a way that $\{\xWe_{\e} : \e \in \JWe\}$ form a basis of the character lattice. Let us denote the torus $T_{\dbr}$ with this multiplicative structure by $T_{\dbr}^{\We}$. We will show that for every $\c \in \J3D$, $\x3D_{\c}$ is a character of $T_{\dbr}^{\We}$. Note that this will imply the result, as we will have that the constant function $\prop_{\e} = \xWe_{\e}(\x3D_{\inve})^{-1}$ is a character of $T_{\dbr}^{\We}$ and thus $\prop_\e = 1$ for every $e \in \JWe$. 

By \cref{lem:impt-facts-minors}, if the grid minors $\grid_{\c, i}$ are characters of $T_{\dbr}^{\We}$ for every $\c \in [1, m+n]$ and $i \in I$, then $\x3D_{\c}$ is a character of $T_{\dbr}^{\We}$ for every $c \in \J3D$. It thus suffices to show that the grid minors $\grid_{\c, i}$ are characters of $T_{\dbr}^{\We}$ for every $\c \in [1, m+n]$ and $i \in I$. By \cref{thm:u-as-cluster-var-product}, the functions $u_{\edge}$ are characters of $T_{\dbr}^{\We}$ for every edge $\edge$ of the double inductive weave $\DW=\DW(\dbr)$ associated to the double braid word $\dbr$. By \cref{lem:hc-in-terms-of-uc}, the torus element $\hp_\c$ is a particular cocharacter applied to some $u_\edge$-variables. By \cref{def:grid-minors}, the grid minors are obtained by applying $\omega_i$ to $\hp_\c$ or a particular conjugate of $\hp_\c$. Since cocharacters are multiplicative, as is conjugation and applying $\omega_i$, the grid minors are characters of $T_{\dbr}^{\We}$, as desired.\end{proof}

\cref{thm:coincidence-cluster-variables} concludes the proof of equality of cluster \emph{variables}. In order to show coincidence of cluster \emph{structures} we still need to show that the exchange matrix associated to the double braid word $\dbr$ coincides with the exchange matrix associated to the corresponding double inductive weave. This will be achieved in \cref{sec:quiver-compare}.


\subsection{Lusztig data and cocharacter comparison}\label{ssec:Lusztig_data_cocharacter_comparison}

We conclude this section by introducing the concept of Lusztig datum, which helps crystallize the role of cocharacters in the context of weaves, and by studying how cocharacters used in \cite[Section 7]{GLSB} and (implicitly) in \cite{CGGLSS} compare to each other. \cref{sec:vertes_cycles_to_Lusztig_data} focuses on Lusztig data, and \cref{ssec:consequence_equality_variables} establishes the equality between the associated weave cocharacters and the Deodhar cocharacters. Note that neither of these two subsections are logically required to prove \maintheoref, but might be of independent interest.

\subsubsection{From vertex cycles to Lusztig data}\label{sec:vertes_cycles_to_Lusztig_data} Given a Demazure weave $\weave$, it can be conceptually useful to translate the information of a vertex cycle $\lcyc_\e: E(\weave) \to \Z_{\geq0}$ into a collection of \emph{Lusztig data} $[\bj, f]$ and a list of coweights, as introduced in \cref{def:lusztig-datum} below. Conceptually, coroots $\chi:\C^\times\lr T_{\weave}$ can be evaluated at $u_\edge$, which \cref{thm:u-as-cluster-var-product} expresses as cluster variables with exponents given by vertex cycles $\lcyc_{\c}$. This relation between weaves cycles and coroots can be formalized as follows.

\begin{definition}[Lusztig datum]\label{def:lusztig-datum}
    A \emph{weighted expression} for $w \in W$ is a pair $(\bj, f)$, where $\bj$ is a reduced expression for $w$ and $f: [\ell(w)] \mapsto \Z_{\geq0}$ is a nonnegative integer weighting of the letters. By definition, two weighted expressions $(\bj, f)$ and $(\bj', f')$ are \emph{equivalent} if for any Demazure weave $\weave: \bj \to \bj'$, the unique Lusztig cycle with values given by $f$ at the top of $\weave$ has values given by $f'$ at the bottom of $\weave$. By definition, a \emph{Lusztig datum} for $w$ is an equivalence class of weighted expressions for $w$, denoted by $[\bj, f]$.\hfill$\Box$
\end{definition}

\noindent By definition, the \emph{coweight} of a Lusztig datum $[\bj, f]$ for $w\in W$ is the following linear combination of coroots:
    \begin{equation}\label{eq:coweight-of-Lusztig-datum}
       \charof[\bj, f]:= \sum_{r=1}^{\ell(w)} f(r) \chi^{\bj}_r.
    \end{equation}

The coroots used in \cref{eq:coweight-of-Lusztig-datum} are the inversions of $w$, as defined in \cref{eq:coroots-for-red-expr}. Note that the coweight in \eqref{eq:coweight-of-Lusztig-datum} depends only on the equivalence class $[\bj, f]$, and not a given representative $(\bj, f)$, i.e.~ the coweight of $[\bj, f]$ is well-defined.

Let $\dstr$ be a double string of length $\ell$ and $\DW=\DW(\dstr)$ the corresponding weave. The connection between Lusztig data in \cref{def:lusztig-datum} and weaves $\DW$ starts with the fact that the horizontal slice of the weave $\DW$ at depth $\c$ gives a reduced expression $\bjc$ for the permutation $w_\c^{\dstr}$. Then, given a fixed trivalent vertex $\vertex_\e$ of $\DW$, we can evaluate the vertex cycle $\lcyc_\e$ on the edges in this slice, which gives a weighting $\nudc$ of the reduced expression $\bjc$. This pair gives a well-defined equivalence class $[\bjc, \nudc]$ which is a Lusztig datum for $w_\c^{\dstr}$:

\begin{definition}[Lusztig datum from weave slices]\label{def:lusztig-data-from-dbl-weave} In the notation above, we define the Lusztig datum for $w_\c^{\dstr}$ to be given by the equivalence class $[\bjc, \nudc]$. Its coweight is denoted by $\gamWe_{\dstr,c,e}:= \charof[\bjc, \nudc]$ and, if $\dstr=\dstr(\dbr)$ is the double string associated to the double braid word $\dbr$, we write $\gamWe_{\brWe,c,e}$ for the coweight $\gamWe_{\dstr,c,e}$.\hfill$\Box$
\end{definition}

Note that slicing $\DW$ anywhere from depth $\c$ to depth $\c+0.5$ gives rise to the same Lusztig datum, though the reduced word may change. In particular, all of the equivalent weaves we denote by $\DW$ give rise to the same collection of Lusztig data. For an example of a weave and all its Lusztig data, see Figure \ref{fig:ind-weave-triangle-cycles} and Table \ref{tab:Lusztig-databanks}. 

\begin{figure}[h!]
    \centering
    \adjustbox{scale=0.7}{
    \begin{tikzcd}
\draw (-1,-1) to (17,-1) to (8,8) to (-1,-1);

\draw[color=blue] (8.5, 7.5) to[out=270, in=45] (6.5, 5);
\draw[color=red] (9.5, 6.5) to[out=270,in=135] (10.5, 4);
\draw[color=blue] (5.5,5.5) to[out=270,in=135] (6.5, 5) to (6.5, 4.5) to [out=270,in=160] (10.5, 2.5);
\draw[color=red] (11.5, 4.5) to[out=270,in=45] (10.5, 4) to (10.5, 2.5);
\draw[color=blue] (12.5, 3.5) to[out=270, in=20] (10.5, 2.5);

\draw[color=blue] (10.5, 2.5) to (10.5,-1);

\draw[color=red] (10.5, 2.5) to[out=0,in=160] (12, 1.5) to (12, -1);

\draw[color=red] (13.5, 2.5) to[out=270, in=45] (12, 1.5);

\draw[color=red] (0.9, 0.9) to[out=340, in=170] (6, 0.5);

\draw[color=red] (10.5, 2.5) to[out=190,in=20] (6, 0.5) to (6,-1);

\draw[dashed] (0,0) to (16,0);

\draw[dashed] (1.4,1.4) to (14.6,1.4);

\draw[dashed] (2.8, 2.8) to (13.2, 2.8);

\draw[dashed] (3.8, 3.8) to (12.2, 3.8);

\draw[dashed] (4.8, 4.8) to (11.2, 4.8);

\draw[dashed] (6, 6) to (10, 6);

\draw[dashed] (7,7) to (9,7);

\draw[dashed] (7.7, 7.7) to (8.3, 7.7);

\node at (6.5, 4.9) {\color{Magenta} \bullet};
\node at (6.3, 4) {\color{Magenta} 1};
\node at (11.3, 1.7) {\color{Magenta} 1};

\node at (10.5, 3.9) {\color{LimeGreen} \bullet};
\node at (10.3, 3.3) {\color{LimeGreen} 1};
\node at (11.1, 1.9) {\color{LimeGreen} 1};
\node at (8.6, 1.9) {\color{LimeGreen} 1};

\node at (12, 1.4) {\color{YellowOrange} \bullet};
\node at (11.8, 0.8) {\color{YellowOrange} 1};

\node at (6, 0.4) {\color{Blue} \bullet};
\node at (5.8, -0.4) {\color{Blue} 1};

\end{tikzcd} 
}
    \caption{The double inductive weave from Figure \ref{fig:ind-weave-triangle}, with the nonzero values of the vertex cycles indicated and the integer depths depicted by dashed lines. The top dashed line has depth $c=0$, then scanning downwards until $c=7$ for the bottom dashed line. The Lusztig data associated to the trivalent vertices are displayed in Table \ref{tab:Lusztig-databanks}.}
    \label{fig:ind-weave-triangle-cycles}
\end{figure}
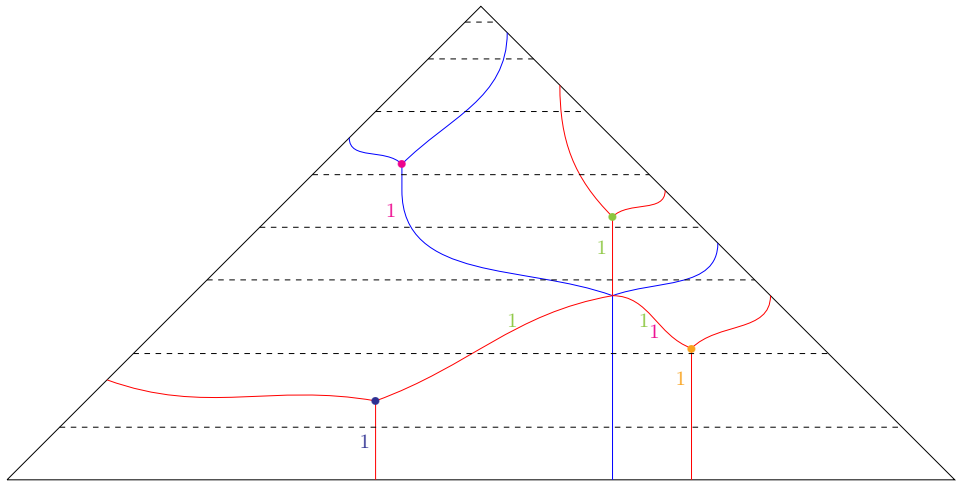

\begin{table}[h!]
    \centering
\begin{tabular}{|c|c|c|c|c|c|c|c|c|c|}
    \hline
    Depth $c$ & $\bjc$ & {\color{Magenta}{$\nu_{2}^{(c)}$}} & {\color{Magenta}{$\charof[\bjc, \nu_{2}^{(c)}]$}} & {\color{LimeGreen}{$\nu_{3}^{(c)}$}} & {\color{LimeGreen}{$\charof[\bjc, \nu_{3}^{(c)}]$}} & {\color{YellowOrange}{$\nu_{5}^{(c)}$}} & {\color{YellowOrange}{$\charof[\bjc, \nu_{5}^{(c)}]$}} & {\color{Blue}{$\nu_{6}^{(c)}$}} & {\color{Blue}{$\charof[\bjc, \nu_{6}^{(c)}]$}} \\
    \hline
     1 & {\bf 2} & {\color{Magenta}{$0$}} & {\color{Magenta}{$0$}} &{\color{LimeGreen}{$0$}} & {\color{LimeGreen}{$0$}} &{\color{YellowOrange}{$0$}} & {\color{YellowOrange}{$0$}} & {\color{Blue}{$0$}} & {\color{Blue}{$0$}} \\
     \hline 
      2 & {\bf 21} & {\color{Magenta}{$00$}} & {\color{Magenta}{$0$}} & {\color{LimeGreen}{$00$}} & {\color{LimeGreen}{$0$}} & {\color{YellowOrange}{$00$}} & {\color{YellowOrange}{$0$}} & {\color{Blue}{$00$}} & {\color{Blue}{$0$}} \\
      \hline
        3 & {\bf 21} & {\color{Magenta}{$10$}} & {\color{Magenta}{$s_1\cdot \chi_2$}} & {\color{LimeGreen}{$00$}} & {\color{LimeGreen}{$0$}} & {\color{YellowOrange}{$00$}} &  {\color{YellowOrange}{$0$}} & {\color{Blue}{$00$}} & {\color{Blue}{$0$}} \\
        \hline 
         4 & {\bf 21} & {\color{Magenta}{$10$}} & {\color{Magenta}{$s_1\cdot \chi_2$}} & {\color{LimeGreen}{$01$}} & {\color{LimeGreen}{$\chi_1$}} & {\color{YellowOrange}{$00$}} & {\color{YellowOrange}{$0$}} & {\color{Blue}{$00$}} & {\color{Blue}{$0$}} \\
         \hline
        5 & {\bf 212} & {\color{Magenta}{$100$}} & {\color{Magenta}{$s_2s_1\cdot \chi_2$}} & {\color{LimeGreen}{$010$}} & {\color{LimeGreen}{$s_2\cdot \chi_1$}} & {\color{YellowOrange}{$000$}} & {\color{YellowOrange}{$0$}} & {\color{Blue}{$000$}} & {\color{Blue}{$0$}} \\
         \hline
      6 & {\bf 121} & {\color{Magenta}{$000$}} & {\color{Magenta}{$0$}} & {\color{LimeGreen}{$100$}} & {\color{LimeGreen}{$s_1s_2\cdot \chi_1$}} & {\color{YellowOrange}{$001$}} & {\color{YellowOrange}{$\chi_1$}} & {\color{Blue}{$000$}} & {\color{Blue}{$0$}}\\
         \hline
     7 & {\bf 121} & {\color{Magenta}{$000$}} & {\color{Magenta}{$0$}} & {\color{LimeGreen}{$000$}} & {\color{LimeGreen}{$0$}} & {\color{YellowOrange}{$001$}} & {\color{YellowOrange}{$\chi_1$}} &  {\color{Blue}{$100$}} & {\color{Blue}{$s_1s_2\cdot\chi_1$}}\\
         \hline
    \end{tabular}
    \caption{Lusztig data for the weave in Figure \ref{fig:ind-weave-triangle-cycles}. The color-code is that of the trivalent vertices in Figure \ref{fig:ind-weave-triangle-cycles}. The four vertex cycles are $\nu_2,\nu_3,\nu_5,\nu_6$, with $\nu_c$ corresponding to the vertex in between depths $c$ and $c+1$. Note that there are no $\nu_1$ or $\nu_4$ as there are no trivalent vertices between depths $c=1$ and $c=2$, and depths $c=4$ and $c=5$.}
    \label{tab:Lusztig-databanks}
\end{table}


\subsubsection{Comparing weave and Deodhar cocharacters}\label{ssec:consequence_equality_variables} For each $\c \in [0, m+n]$ we have many grid minors, which are all monomials in the Deodhar cluster variables. In order to record how the cluster variable $\x3D_\e$ spread across all grid minors at $\c$, it can be useful to introduce the following cocharacter:

\begin{definition}\label{def:Deodhar-gamma-cocharacter}
    Let $\dbr$ be a double braid word, $\c \in [0, m+n]$ and $\e \in \J3D$. By definition, the cocharacter $\gamchp_{\dbr, \c, \e}: \C^* \to \torus$  
    is given by
    \[\gamchp_{\dbr, \c, \e}:= \sum_{i \in I} \ord_{V_e}(\c,i) \chi_i\]
    where $\ord_{V_e}(\c,i)$ is the order of vanishing of $\grid_{\c, i}$ on $\tV_\e$, as in \cref{lem:grid-as-mono-in-cluster}.\qed
\end{definition}
Note that the Cartan element $\hp_c$ from \cref{eq:hp_c_dfn} satisfies the equality
\begin{equation}\label{eq:hp_in_terms_of_x3D-s}
  \hp_\c = \prod_{\e\in\J3D} \gamchp_{\dbr,\c,\e}(\x3D_\e).
\end{equation}

\begin{remark}
Similar to \cref{def:Deodhar-gamma-cocharacter}, one may define a cocharacter $\gamchm_{\dbr, \c, \e}:= \sum_{i \in I} \ord_{V_e}(\c,-i) \chi_i$ encoding the vanishing of $\grid_{\c,-i}$. This cocharacter records the ``spread" of the cluster variables in the negative chamber minors, and satisfies an analogous formula to \eqref{eq:hp_in_terms_of_x3D-s}, using $\bar{\bar{u_c}} \hp_c \bar{\bar{u_c}}^{-1}$ rather than $\hp_c$. It suffices to analyze $\gamchp_{\dbr, \c, \e}$.\qed
\end{remark}

\cref{def:Deodhar-gamma-cocharacter} gives a geometric description of the cocharacters $\gamchp_{\dbr, \c, \e}$, whereas \cite[Section 7]{GLSB} gives an algorithm for computing $\gamchp_{\dbr, \c, \e}$ in general type using root-system combinatorics and a somewhat intricate induction. As an application of our main cluster comparison result, we will see in \cref{cor:Deodhar_cocharacters_via_weaves} that the weave $\DW$ corresponding to $\dstr(\dbr)$ gives a simpler way to compute these cocharacters.

\begin{remark}
In Lie Type $A$, \cite[Section 3.4]{GLSBS} gives an algorithm for computing $\gamchp_{\dbr, \c, \e}$  using the combinatorics of \emph{monotone multicurves}. \cref{ssec:compare-cocharacters-type-A} establishes the comparison between these specific combinatorics and the weave framework.\hfill$\Box$
\end{remark}


The next corollary shows that the Deodhar cocharacter, recording the orders of vanishing of grid minors along Deodhar hypersurfaces, can be computed entirely combinatorially from weaves, Lusztig cycles, and the corresponding Lusztig datum.

\begin{cor}
\label{cor:Deodhar_cocharacters_via_weaves}
    Let $\dbr$ be a double braid word, $\brWe$ the corresponding braid word and $\DW$ the corresponding double inductive weave. Choose $\e \in \JWe$ and $\c \in [0, m+n]$. Then 
    \[\gamchp_{\dbr, \invc, \inve} = \gamWe_{\brWe, \c, \e}.\]
\end{cor}
\begin{proof}
We use the notation from the proof of \cref{lem:hc-in-terms-of-uc}. By \eqref{eq:hp_in_terms_of_u-s} and \eqref{eq:hp_in_terms_of_x3D-s}, we have
    \[\prod_{\inve\in\J3D} \gamchp_{\dbr,\invc,\inve}(\x3D_\inve)= \hp_{\invc} = \chi^{\bj}_1(u_{\edge_1}) \cdot \chi^{\bj}_2(u_{\edge_2}) \cdots \chi^{\bj}_\ell(u_{\edge_\ell}).\]
    Recall from \cref{thm:u-as-cluster-var-product} that each $u_{\edge_i}$ is a product of cluster variables, and the exponents are given precisely by the values of the vertex cycles on $\edge_i$. The contribution of $\xWe_\e$ to the right-hand side of the above equation is
    \[\chi^{\bj}_1({(\xWe_\e)}^{\lcyc_\e(\edge_1)}) \cdot \chi^{\bj}_2({(\xWe_\e)}^{\lcyc_e(\edge_2)}) \cdots \chi^{\bj}_\ell({(\xWe_\e)}^{\lcyc_e(\edge_\ell)})= \gamWe_{\brWe, \c, \e} (\xWe_\e).\]
    So in fact 
    \[\prod_{\inve\in\J3D} \gamchp_{\dbr,\invc,\inve}(\x3D_\inve)= \hp_{\invc} = \prod_{\e\in\JWe} \gamWe_{\brWe,\c,\e}(\xWe_\e).\]
    By \cref{thm:coincidence-cluster-variables}, $\x3D_{\inve} = \xWe_\e$ for every $\e \in \JWe$. Let us fix $e \in \JWe$. Since cluster variables are algebraically independent, we can set $\x3D_{\invf} = \xWe_\f = 1$ for $\f \neq \e$ and the above equation then shows that $\gamchp_{\dbr, \invc, \inve} = \gamWe_{\brWe, \c, \e}$.
\end{proof}

\begin{remark}
\cref{cor:Deodhar_cocharacters_via_weaves} could also be proved purely combinatorially by showing that the cocharacter $\gamWe_{\brWe, \c, \e}$ satisfies all the conditions in~\cite[Proposition~7.2]{GLSB}.\qed
\end{remark}
\color{black}

\section{Comparison of exchange matrices}\label{sec:quiver-compare}


Let $\dbr$ be a double braid word. \cref{sec:braid_varieties} established an isomorphism of algebraic varieties
$$\iso: R(\dbr) \xrightarrow{\sim} X(\dbr^{(-|+)}).$$
\cref{sec:weaves-tori} proved the equality of tori between the weave and Deodhar cluster algebra structures: $\iso(T_\dbr)=T_{\DW(\dstr(\dbr))}$, cf.~\cref{lem:weave-tori-vs-deodhar-tori}. \cref{sec:cluster_var} showed the corresponding equality of cluster variables: $\iso^{\ast}\bxWe=\bx3D$, cf.~\cref{thm:coincidence-cluster-variables}. In order to complete our proof of \maintheoref, it remains to establish the equality of 2-forms: $\iso^{\ast}\OmWe = \Om3D$, which is the content of the following theorem.

\begin{thm}\label{thm:coincidence-forms}
Let $\dbr$ be a double braid word, $\C[R(\dbr)]$ the ring of regular functions of its associated double braid variety and $\Om3D$ the Deodhar 2-form on the Deodhar torus $T_\dbr$. Consider the isomorphism $\iso: R(\dbr) \xrightarrow{\sim} X(\dbr^{(-|+)})$ in \cref{eq:iso-glsbs-cgglss} and the weave 2-form $\OmWe$ on the weave torus $T_{\DW(\dstr(\dbr))}$ associated to the double inductive weave of $\dbr$. Then
\[
\iso^{\ast}\OmWe = \Om3D. 
\]
\end{thm}

The remainder of this section is devoted to proving \cref{thm:coincidence-forms}.


\subsection{Setup and notation}\label{ssec:preliminaries_2forms} In this section, we take the viewpoint that a cluster seed $\Sigma$ consists of a collection $\mathbf{x} = \{x_1, \dots, x_n\}$ of cluster variables, a skew-symmetrizable matrix  $(\varepsilon_{ij})$ for $1\leq i, j\leq n$, together with symmetrizers $d_i$ for $1\leq i \leq n$. The matrix $(\varepsilon_{ij})$ is often referred to as the exchange matrix, and in the skew-symmetric case it is equivalently encoded by a quiver. In our case, where the cluster algebra $A=\C[X]$ is the coordinate ring of an affine variety $X$, the cluster variables $x_i\in\C[X]$ are regular functions on $X$ and the torus $T_\Sigma\sse X$ associated to $\Sigma$ is given by the non-vanishing of the cluster variables $\mathbf{x}$ in $\Sigma$, so that we have the defining equality $\C[T_\Sigma]=\C[x_1^{\pm1},\ldots,x_n^{\pm1}]$.

The data of the cluster seed $\Sigma$ defines a 2-form on the corresponding cluster torus $T_\Sigma$ by
\begin{equation}\label{def:2-form_of_a_seed}
\Omega(\Sigma) := \sum_{1 \leq i, j \leq n} d_i \varepsilon_{ij} \dlog x_i \wedge \dlog x_j, 
\end{equation}
where the symbol $\dlog(x)$ is notation for $\dlog(x):=x^{-1}dx$. In other words, the cluster variables $\mathbf{x}$ are the exponential Darboux coordinates for $\Omega(\Sigma)$ on $T_\Sigma$, where the constant coefficients are given by the data of $(\varepsilon_{ij})$ and $d_i$. In general, for a locally acyclic cluster algebra, this form $\Omega(\Sigma)$ extends from the torus $T_\Sigma$ to a (homonymous) regular $2$-form $\Omega(\Sigma)$ on the entirety of the (spectrum of the) associated cluster algebra $A(\Sigma)$, see \cite[Theorem 4.4]{Muller13} and \cite[Lemma 6.10]{FockGoncharov_ensemble}. The extended 2-form $\Omega(\Sigma)$ does not depend on the chosen seed, and thus we denote it simply by $\Omega$. Note that the 2-form $\Omega$ together with the cluster $\mathbf{x}$ determines the matrix $(\varepsilon_{ij})$ and thus, in the presence of $\mathbf{x}$, the data of the exchange matrix $(\varepsilon_{ij})$ and the 2-form are interchangeable. Since both the weave and Deodhar cluster algebras are locally acyclic, by \cite[Theorem 7.13]{CGGLSS} and \cite[Theorem 4.10]{GLSB} respectively, these considerations apply to the coordinate rings of $X=X(\dbr^{(-|+)})$ and $X=R(\dbr)$.

Let $\dbr$ be a double braid word and $\brWe$ its associated braid word. We have shown in \cref{thm:coincidence-cluster-variables} the equality of cluster variables
\[
\xWe_\c = \x3D_\invc,
\]
where we denote by $\xWe_\c$ the pullback of $\xWe_\c \in \C[X(\brWe)]$ to $\C[R(\dbr)]$ under the isomorphism $\iso$, to ease notation. Let $\OmWe\in\Omega^2(X(\brWe))$ be the regular $2$-form on $X(\brWe)$ obtained from the weave cluster structure, as constructed and studied in \cite[Sections 8\&9]{CGGLSS}, and let $\Om3D\in\Omega^2(R(\dbr))$ be the regular $2$-form on $R(\dbr)$ obtained from the Deodhar cluster structure, as constructed in \cite[Section 2.9]{GLSB}. To ease notation, we still denote by $\OmWe$ the pullback $\iso^*\OmWe$ of the form $\OmWe$ under the isomorphism $\iso$, so that both $\OmWe$ and $\Om3D$ are defined on $R(\dbr)$. In this notation, our goal is to establish the equality
\begin{equation}\label{eq:equality-2-forms}
\OmWe = \Om3D,
\end{equation}
as regular 2-forms on $R(\dbr)$.

\begin{remark}\label{rmk:only-positive-letters}
In order to show \eqref{eq:equality-2-forms}, it suffices to assume that $\dbr$ is a double braid word only on positive letters. Indeed, if $\dbr$ and $\dbr'$ are related by either of the moves (B1) or (B4) from \cref{rmk:compatibility-iso-moves}, \cite[Section 4]{GLSB} implies that $\phi_{\dbr, \dbr'}^{\ast}(\Om3D_{\dbr'}) = \Om3D_{\dbr}$, where $\phi_{\dbr, \dbr'}$ is the isomorphism from \cref{rmk:compatibility-iso-moves}.\footnote{This equality is called Property (F) in the proof of \cite[Theorem 4.2]{GLSB}.} Using (B1) and (B4) moves, we can bring any double braid word to a double braid word with only positive letters. The commutativity of the diagram \eqref{eq:compatibility-iso-moves} implies that we can assume $\dbr$ only has positive letters.\qed
\end{remark}

We show \eqref{eq:equality-2-forms} by first analyzing how the $2$-forms are constructed in finer detail: \cref{ssec:Deodhar-2-form} does so for the Deodhar 2-form $\Om3D$, and \cref{ssec:weave-2-form} for the weave 2-form $\OmWe$. Both forms are defined inductively, using a \lq\lq sum of local contributions\rq\rq \; procedure. We then show that partial sums of these local contributions coincide in \cref{ssec:equality-2-forms}, which implies \eqref{eq:equality-2-forms}.


\subsection{The Deodhar 2-form}\label{ssec:Deodhar-2-form}

Let $\dbr$ be a double braid word. The Deodhar 2-form $\Om3D$ on $R(\dbr)$ is constructed by using the grid minors $\Delta_{c,i}$ of $\dbr$, introduced in \cref{def:grid-minors}. Note that \cref{lem:grid-as-mono-in-cluster} expresses such grid minors $\Delta_{c,i}$ as a monomials in the Deodhar cluster variables $\x3D$.

For $i,j \in I$, let $a_{ij}=\langle\alpha_i,\chi_j\rangle$
denote the entries of the Cartan matrix of $\G$. The Cartan matrix has symmetrizers $d_i$, where $d_i a_{ij} = d_j a_{ji}$. For $i, j \in \pm I$, we set $a_{ij}:=0$ if $i,j$ have different signs, and $a_{ij}:=a_{|i||j|}$
otherwise. We also set $d_{i}:=d_{|i|}$. For $c \in [0,\nm]$ and $i \in \pm I$, we define the following 1-form on the Deodhar torus $T_\dbr$:
\begin{equation*}
    L_{c,i}:= \frac{1}{2} \sum_{k \in \pm I} a_{ik} \dlog \Delta_{c,k}\in\Omega^1(T_\dbr). 
\end{equation*}
Note that if $i \in I$, then 
\[L_{c,i}= \frac{1}{2} \dlog \left(\prod_{k \in I} \omega_k(\hp_c)^{a_{ik}} \right)= \frac{1}{2} \dlog \alpha_i(\hp_c)\]
where $\alpha_i$ is the simple root indexed by $i$. Here we are using the fact that the Cartan matrix gives a change of basis between simple roots and fundamental weights.

For each crossing $c \in [\nm]$, we define the following 2-form on $T_\dbr$
\begin{equation*}
    \Om3D_{\dbr,c}:= \sign(i)~2 d_i~L_{c-1, i} \wedge L_{c,i} \quad \text{where }i:=i_c.
\end{equation*}
Finally, by \cite[Section 2.9]{GLSB}, the Deodhar 2-form $\Om3D=\Om3D_\dbr$ associated to the double braid word $\dbr$ can be given as the sum
\begin{equation}\label{eq:om3D-defn}
    \Om3D_\dbr:= \sum_{c \in [m]} \Om3D_{\dbr,c} = \sum_{c \in \J3D} \Om3D_{\dbr,c}.
\end{equation}
The fact that the second equality holds in \cref{eq:om3D-defn}, i.e.~the 2-forms $\Om3D_{\dbr,c}$ vanish if $c\not\in\J3D$, is \cite[Equation (2.28)]{GLSB}.

\begin{remark}
    The convention of \cite{GLSB} is that the 2-form of a seed $\Sigma$ is $\frac{1}{2} \Omega(\Sigma)$ (compare \cite[Equation (1.3)]{GLSB} with \eqref{def:2-form_of_a_seed}). Thus $\Om3D_{\dbr,c}$ and $\Om3D_{\dbr}$ above differ from the 2-form in \cite{GLSB} by a factor of 2. However, the exchange matrices obtained in \cite{GLSB} and those defined here agree.
    \qed
\end{remark}


\subsection{The weave 2-form}\label{ssec:weave-2-form}

Let $\ww$ be a Demazure weave associated with a braid word $\br$. As in Theorem \ref{thm:u-as-cluster-var-product}, every edge $\edge$ of $\ww$ is associated with a $u$-variable
\begin{equation}
\label{u-x.change}
u_{\edge}=\prod_{c} (x_{c}^{\rm W})^{\lcyc_c(\edge)},
\end{equation}
where the product is over all Lusztig cycles $c$. We now present the weave 2-form $\OmWe$, as in \cite[Sections 8\&9]{CGGLSS}, and describe it in terms of these $u$-variables. This description in terms of $u$-variables is a key step in \cref{thm:coincidence-forms}, to compare to the Deodhar form. Indeed, once the weave 2-form is expressed in terms of $u$-variables, we can use \cref{lem:hc-in-terms-of-uc}, which expresses the torus element $h^+$ in terms of coroots evaluated at the $u$-variables, to obtain an expression for $\OmWe$ which is much closer to the Deodhar 2-form as presented in \cref{ssec:Deodhar-2-form}.

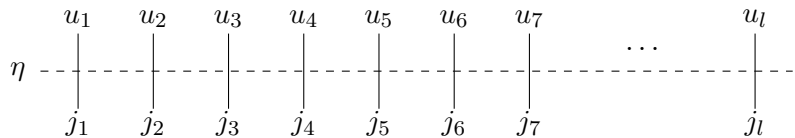
\begin{figure}[h]
\begin{center}
\begin{tikzpicture} 
\draw[dashed] (0,0)--(10,0);
\draw (0.5, 0.5) -- (0.5, -.5);
\draw (1.5, 0.5) -- (1.5, -.5);
\draw (2.5, 0.5) -- (2.5, -.5);
\draw (3.5, 0.5) -- (3.5, -.5);
\draw (4.5, 0.5) -- (4.5, -.5);
\draw (5.5, 0.5) -- (5.5, -.5);
\draw (6.5, 0.5) -- (6.5, -.5);
\node at (8, 0.3) {$\ldots$};
\node at (-.3,0) {$\eta$};
\draw (9.5, 0.5) -- (9.5, -.5);

\node at (0.5, 0.7) {$u_1$};
\node at (1.5, 0.7) {$u_2$};
\node at (2.5, 0.7) {$u_3$};
\node at (3.5, 0.7) {$u_4$};
\node at (4.5, 0.7) {$u_5$};
\node at (5.5, 0.7) {$u_6$};
\node at (6.5, 0.7) {$u_7$};
\node at (9.5, 0.7) {$u_l$};

\node at (0.5, -0.7) {$j_1$};
\node at (1.5, -0.7) {$j_2$};
\node at (2.5, -0.7) {$j_3$};
\node at (3.5, -0.7) {$j_4$};
\node at (4.5, -0.7) {$j_5$};
\node at (5.5, -0.7) {$j_6$};
\node at (6.5, -0.7) {$j_7$};
\node at (9.5, -0.7) {$j_l$};
\end{tikzpicture}
\end{center}
\caption{A horizontal slice $\eta$ of the weave $\ww$}
\label{scaijns3}
\end{figure}

Following \cite[Section 4.1]{CGGLSS}, a generic horizontal slice $\eta$ of the weave $\ww$ is a horizontal cut that intersects the weave $\ww$ without passing through any vertices, as indicated by the dashed line in Figure \ref{scaijns3}. The edges intersecting $\eta$ are colored $j_1, \ldots, j_l$ from left to right, and we denote by $u_1, \ldots, u_l$ their associated $u$-variables. Let ${\bf j}=s_{j_1}\ldots s_{j_l}$ be the corresponding word in simple reflections (note that ${\bf j}$ is not necessarily reduced), and recall the sequence of coroots $\chi_1^{\bf j}, \ldots, \chi_l^{\bf j}$ defined in \eqref{eq:coroots-for-red-expr}, whose norms are given by the symmetrizer $d_{j_i}= ||\chi_{j_i}||^2$. Similarly, define the sequence of roots
\begin{equation}
 \alpha_k^{\bf j}:= s_{j_l}s_{j_{l-1}}\cdots s_{j_{k+1}}\cdot \alpha_{j_k}, \qquad \mbox{where } 1\leq k\leq l.
\end{equation}

For every trivalent vertex $c$ of $\ww$ with associated Lusztig cycle $\lcyc_c$, we consider the Langlands dual Lusztig cycle, cf.~\cite[Section 6.6]{CGGLSS}:
\begin{equation}
\label{co-lusztig-rel}
\lcyc^{\vee}_c(i)= \lcyc_c(i)d_{j_i}d_c^{-1},
\end{equation}
where $d_c := d_j$ if $c$ is colored by $j \in I$.

Following \cite[Section 6.1.(38)]{CGGLSS}, the intersections between Lusztig cycles and their duals at a given slice $\eta$ can be defined as
\[
\sharp_\eta \left(\lcyc^{\vee}_c\cdot \lcyc_d\right):= \frac{1}{2}\sum_{i,k=1}^l{\rm sign}(k-i)\lcyc_c^\vee (i)\lcyc_d(k) \cdot (\alpha_i^{\bf j}, \chi_k^{\bf j}).
\]
These intersections determine the weave 2-form $\Omega^W_\eta$ at the given slice $\eta$, by declaring
\begin{equation}\label{def:slice_weave2form}
\Omega_\eta^{\rm W}:= \sum_{c,d} d_c \cdot \sharp_\eta \left(\lcyc^{\vee}_c\cdot \lcyc_d\right)~ \dlog \xWe_c \wedge \dlog \xWe_d.
\end{equation}
Such weave 2-form $\Omega^W_\eta$ can be expressed in terms of the $u$-variables as follows:
\begin{lemma} \label{lemma-u-slice}
In the notation above, the weave 2-form $\Omega^W_\eta$ equals
\begin{equation}\label{eq:boundary-form}
\Omega_\eta^{\rm W}= \sum_{1\leq i < k \leq l} d_{j_i} (\alpha_i^{\bf j}, \chi_k^{\bf j}) \dlog u_i \wedge \dlog u_k.
\end{equation}
\end{lemma}
\begin{proof}

By using \eqref{u-x.change}, we can describe the $u$-variables in the right hand side of \eqref{eq:boundary-form} in terms of the $x$-variables. This leads to the following description for the 2-form:
\begin{align*}
&\sum_{1\leq i < k \leq l} d_{j_i} (\alpha_i^{\bf j}, \chi_k^{\bf j}) \dlog u_i \wedge \dlog u_k\\
=&\sum_{1\leq i < k \leq l} d_{j_i} (\alpha_i^{\bf j}, \chi_k^{\bf j})\left( \sum_{c,d }\lcyc_c(i)\lcyc_d(k)\dlog \xWe_c \wedge \dlog \xWe_d\right)\\
=&\sum_{c,d}d_c\left(\sum_{1\leq i<k\leq l}\lcyc^\vee_c(i)\lcyc_d(k)  (\alpha_i^{\bf j}, \chi_k^{\bf j}) \right) \dlog \xWe_c \wedge \dlog \xWe_d\\
    =&\frac{1}{2} \sum_{c,d}d_c\left(\sum_{i,k=1}^l{\rm sign}(k-i)\lcyc_c^\vee (i)\lcyc_d(k)  (\alpha_i^{\bf j}, \chi_k^{\bf j}) \right) \dlog \xWe_c \wedge \dlog \xWe_d\\
=& \Omega_\eta^{\rm W},
\end{align*}
as required. In the above the second equality uses \eqref{co-lusztig-rel}.
\end{proof}

\begin{example}[Type $G_2$]Let ${\bf j}=s_2s_1s_2s_1s_2s_1$, where $\alpha_1$ is the longer simple root, with symmetrizers $d_1=1$ and $d_2=3$. The sequence of  coroots $\chi_i^{\bf j}$ is
\[
\chi_2, ~~\chi_1+\chi_2, ~~ 3\chi_1+2\chi_2,~~2\chi_1+\chi_2,~~3\chi_1+\chi_2,~~\chi_1.
\]
The sequence of roots $\alpha_i^{\bf j}$ is
\[
\alpha_2,~~\alpha_1+3\alpha_2,~~\alpha_1+2\alpha_2,~~2\alpha_1+3\alpha_2,~~\alpha_1+\alpha_2,~~\alpha_1,
\]
and note that the pairing between roots and coroots gives
\[(\alpha_1,\chi_1)=2,~(\alpha_1,\chi_2)=-3,~(\alpha_2, \chi_1)=-1, ~(\alpha_2,\chi_2)=2.
\]
Thus the $6\times 6$ symmetric matrix $M= (m_{ik})$, $m_{ik}:=d_{j_i} (\alpha_i^{\bf j},\chi_k^{\bf j})$, reads
\begin{equation}
\label{2025.2.24.11}
\begin{pmatrix}
    6 &3 &3 &0 &-3 &-3\\
    3 &2 &3 &1 &0 &-1\\
    3 &3 &6 &3 &3 &0\\
    0 &1 &3 &2 &3 &1\\
    -3&0 &3 &3 &6 &3 \\
    -3 &-1 &0 &1 &3 &2 \\
\end{pmatrix}.
\end{equation}

Therefore
\[
\Omega_\eta^{\rm W}= \sum_{1\leq i<k\leq 6} m_{ik} \dlog u_i \wedge \dlog u_k,
\]
which in part illustrates \cref{lemma-u-slice} in this example.\qed
\end{example}

\cref{def:slice_weave2form} defines the weave 2-form at a given slice $\eta$ of a weave $\ww$. The global weave 2-form $\OmWe$ is obtained by scanning the weave downwards, accounting for the local contributions of these sliced 2-forms $\OmWe_\eta$ as the slice $\eta$ moves past the weave vertices. Let us now consider the contributions of the internal vertices of $\ww$ to the global weave form $\OmWe$: starting with the 3-valent case, then 4-valent and 6-valent, and finally the 8-valent and 12-valent cases.

\begin{figure}[h]
\begin{center}
\begin{tikzpicture} 
\node at (11.9, 1.5) {$v$};
\draw[red] (10.5, 3) to[out=270, in=135] (11.5, 1.5);
\node at (10.5, 3.2) {$u_1$};
\draw[red] (12.5, 3) to[out=270, in=45] (11.5, 1.5);
\node at (12.5, 3.2) {$u_3$};
\draw[red] (11.5, 1.5) to (11.5, 0);
\node at (11.5, -0.2) {$u_2$};
\end{tikzpicture}
\end{center}
\caption{The $u$-variables of the edges adjacent to a trivalent vertex $v$.}
\label{scaijns2}
\end{figure}
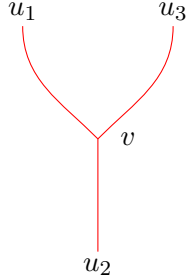


\subsubsection{Local contribution at 3-valents}\label{sssec:3valent} First, let $v$ be a trivalent vertex, as in Figure \ref{scaijns2}. Consider the following local contribution at such trivalent vertex $v$:
\[
\OmWe_v:=\sum_{c,d}d_c \left|\begin{matrix}
1 & 1 & 1\\
\nu_c^\vee(1) & \nu_c^\vee(2) & \nu_c^\vee(3)\\
\nu_d(1) & \nu_d(2) & \nu_d(3)\\
\end{matrix}\right| \dlog \xWe_c\wedge \dlog \xWe_d
\]

In line with \cref{lemma-u-slice}, we can express this local contribution in terms of $u$-variables:

\begin{lemma} \label{lem.x-u.tri} In the notation above, if the three edges incident to $v$ are colored by $i$, then
\[
\OmWe_v= 2d_i \left({\dlog u_1}\wedge {\dlog u_2}+{\dlog u_2}\wedge {\dlog u_3}+ {\dlog u_3}\wedge {\dlog u_1}\right).
\]
\end{lemma}
\begin{proof} Let $e$ and $f$ be edges of $\ww$ with color $i$. Then
\begin{align*}
2d_i\dlog u_e \wedge \dlog u_f &= 2d_i\sum_{c,d} \nu_c(e)\nu_d(f) \dlog \xWe_c\wedge \dlog \xWe_d\\
&=d_i\sum_{c,d}\left(\nu_c(e)\nu_d(f)-\nu_c(f)\nu_d(e)\right)\dlog \xWe_c\wedge \dlog \xWe_d\\
&= \sum_{c,d}d_c \left|\begin{matrix}
\nu_c^\vee(e) & \nu_c^\vee(f) \\
\nu_d(e) & \nu_d(f) \\
\end{matrix}\right|\dlog \xWe_c\wedge \dlog \xWe_d.
\end{align*}
The rest of the equality follows from direct computation.
\end{proof}


\subsubsection{Local contribution at 4-valents}\label{sssec:4valent} For $v$ a $4$-valent vertex, we simply set 
\[
\Omega_v^{\rm W}=0.
\]
That is, 4-valent vertices do {\it not} contribute to the weave 2-form.

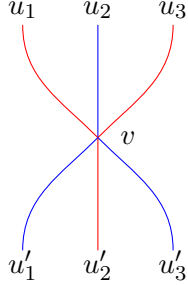
\begin{figure}[h]
\begin{center}
\begin{tikzpicture} 
\draw[red] (0.5, 3) to[out=270, in=135] (1.5, 1.5);
\node at (0.5, 3.2) {${u_1}$};
\node at (1.9,1.5) {$v$};
\draw[blue] (1.5, 3) to (1.5, 1.5);
\node at (1.5, 3.2) {$u_2$};
\draw[red] (2.5, 3) to [out=270, in=45] (1.5, 1.5);
\node at (2.5, 3.2) {$u_3$};
\draw[blue] (1.5, 1.5) to [out=225, in=90] (0.5, 0);
\draw[red] (1.5, 1.5) to (1.5, 0);
\draw[blue] (1.5, 1.5) to[out=315, in=90] (2.5, 0);
\node at (0.5, -0.2) {${u_1'}$};
\node at (1.5, -0.2) {${u_2'}$};
\node at (2.5, -0.2) {${u_3'}$};
\end{tikzpicture}
\end{center}
\caption{The $u$-variables of edges adjacent to a 6-valent vertex $v$.}
\label{scaijns2.1}
\end{figure}


\subsubsection{Local contribution at 6-valents}\label{sssec:6valent} Let $v$ be a $6$-valent vertex, as in Figure \ref{scaijns2.1}.  The local contribution at such hexavalent vertex $v$ is
\[
\OmWe_v := \sum_{c,d}\frac{d_c}{2}\left(\left|\begin{matrix}
1 & 1 & 1\\
\nu_c^\vee(1) & \nu_c^\vee(2) & \nu_c^\vee(3)\\
\nu_d(1) & \nu_d(2) & \nu_d(3)\\
\end{matrix}\right|-\left|\begin{matrix}
1 & 1 & 1\\
\nu_c^\vee(1') & \nu_c^\vee(2') & \nu_c^\vee(3')\\
\nu_d(1') & \nu_d(2') & \nu_d(3')\\
\end{matrix}\right|\right) \dlog \xWe_c \wedge \dlog \xWe_d.
\]
Here the edges incident to $v$ are of two distinct colors $i$ and $j$, where $i, j$ are joined by a simple edge in the Dynkin diagram. Hence, their corresponding symmetrizers coincide, $d_i = d_j$. In line with Lemmas \ref{lemma-u-slice} and \ref{lem.x-u.tri}, the expression in terms of $u$-variables is:

\begin{lemma} \label{lem.x-u.6-val}In the notation above, we have
\begin{align*}
\OmWe_v = &d_i \left({\dlog u_1}\wedge {\dlog u_2}+{\dlog u_2}\wedge {\dlog u_3}+ {\dlog u_3}\wedge {\dlog u_1}\right) \\
&-d_i\left({\dlog u_1'}\wedge {\dlog u_2'}+{\dlog u_2'}\wedge {\dlog u_3'}+ {\dlog u_3'}\wedge {\dlog u_1'}\right).
\end{align*}
\end{lemma}
\begin{proof} It follows by the same calculation as in the proof of Lemma \ref{lem.x-u.tri}. 
\end{proof}
\begin{remark} By \cite[Lemma 4.28]{CGGLSS}, if $v$ is not a trivalent vertex, then $\OmWe_v$ is given by the difference between the contributions of its top and bottom slices. Therefore Lemma \ref{lem.x-u.6-val} can also be understood as a consequence of Lemma~\ref{lemma-u-slice}.\qed
\end{remark}


\subsubsection{Local contribution at 8-valent and 12-valents}\label{sssec:8and12valent} For vertices of valency 8 and 12, the 2-form $\OmWe_v$ is constructed via unfolding and reduction to the simply-laced case. By \cite[Lemma 4.28]{CGGLSS}, $\OmWe_v$ for such a vertex $v$ can be expressed as the difference between the contributions of its top and bottom slices. By Lemma~\ref{lemma-u-slice}, the 2-forms for these top and bottom slices can themselves can be expressed in terms of the $u$-variables. This leads to an expression for the local contributions $\OmWe_v$ in terms of the $u$-variables. For completeness, we provide the explicit formulas below.

\begin{figure}[h]
\begin{center}
\begin{tikzpicture} 
\begin{scope}[shift={(3,0)}]
\draw[red] (0, 3) to[out=270, in=165] (1.5, 1.5);
\draw[blue] (1, 3) to[out=270, in=135] (1.5, 1.5);
\draw[red] (2, 3) to[out=270, in=45] (1.5, 1.5);
\draw[blue] (3, 3) to[out=270, in=15] (1.5, 1.5);
\node at (0, 3.2) {${u_1}$};
\node at (2,1.5) {$v$};
\node at (1, 3.2) {$u_2$};
\node at (2, 3.2) {$u_3$};
\node at (3, 3.2) {$u_4$};
\draw[blue] (1.5, 1.5) to [out=195, in=90] (0, 0);
\draw[red] (1.5, 1.5) to [out=225, in=90] (1, 0);
\draw[blue] (1.5, 1.5) to[out=315, in=90] (2, 0);
\draw[red] (1.5, 1.5) to[out=345, in=90] (3, 0);
\node at (0, -0.2) {${u_1'}$};
\node at (1, -0.2) {${u_2'}$};
\node at (2, -0.2) {${u_3'}$};
\node at (3, -0.2) {${u_4'}$};
\end{scope}
\end{tikzpicture}
\end{center}
\caption{For every 8-valent vertex $v$ of the weave, we consider the corresponding $u$-variables associated with its edges. }
\label{scaijns2.2}
\end{figure}
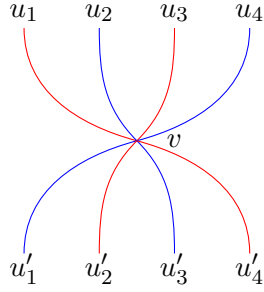

The following lemma can also be proven in the same line as Lemmas \ref{lemma-u-slice}, \ref{lem.x-u.tri}, \ref{lem.x-u.6-val}:
\begin{lemma} In the notation above:\\

(1) Let $v$ be an 8-valent vertex, as in Figure \ref{scaijns2.2}. Then the resulting contribution $\OmWe_v$ in terms of the $u$-variables is
\begin{align*}
\OmWe_v = 2\left(\dlog u_1\wedge \dlog u_2 + \dlog u_2 \wedge \dlog u_3 + \dlog u_3\wedge \dlog u_4 + \dlog u_4 \wedge \dlog u_1\right)  \\
- 2\left (\dlog u_1'\wedge \dlog u_2' + \dlog u_2' \wedge \dlog u_3' + \dlog u_3'\wedge \dlog u_4' + \dlog u_4' \wedge \dlog u_1'\right).
\end{align*}

(2) Let $v$ be a 12-valent vertex with the incoming edges labeled $2,1,2,1,2,1$ and the outgoing edges labeled $1,2,1,2,1,2$. Then
\[
\OmWe_v = \sum_{1\leq i<k\leq 6} m_{ik}\left(\dlog u_i\wedge \dlog u_k - \dlog u_{7-k}' \wedge \dlog u_{7-i}'\right),
\]
where the entries $m_{ik}$ are given as in \eqref{2025.2.24.11}.\qed
\end{lemma}


\subsubsection{The global weave 2-form $\OmWe$} Following \cite[Definition 4.26]{CGGLSS}, the weave 2-form $\OmWe$ is defined as the sum
\begin{equation}\label{eq:OmWe_definition}
\OmWe:= \OmWe_{\bf bottom} + \OmWe_{\mathrm{i}}.
\end{equation}

The first summand of $\OmWe$ is the contribution of the bottom slice $\eta$ of the weave $\ww$, where we denoted $\OmWe_{\bf bottom}: =\OmWe_\eta$, as introduced in \cref{def:slice_weave2form}.\footnote{In this case, $\eta$ encodes a reduced word for $\delta(\beta)$.} The second summand of $\OmWe$ comes from the contributions of all the internal vertices $v$ of $\ww$, where we denoted
\begin{equation}
\OmWe_{\mathrm{i}}:=\sum_{v} \OmWe_v,
\end{equation}
with the local contributions $\OmWe_v$ at internal vertices $v$ as in Subsections \ref{sssec:3valent},\ref{sssec:4valent},\ref{sssec:6valent} and \ref{sssec:8and12valent} above.



\subsection{Proof of \cref{thm:coincidence-forms} (Equality of 2-forms)}\label{ssec:equality-2-forms} 
As discussed in \cref{rmk:only-positive-letters}, it suffices to consider the case where $\dbr = i_{m}\cdots i_1$ consists only of positive letters. Then $\brWe = i_1\cdots i_m$ and \cref{cor:inductive-weaves} implies that the corresponding weave $\ww = \DW(\dstr{(\dbr))} $ is right inductive, that is, all weave lines start on the right side of $\triangle$.

Both 2-forms $\OmWe$ and $\Om3D$ are defined as a sum of local contributions:
\begin{enumerate}
    \item In $\Om3D$, there is a contribution $\Om3D_c$ for each index $c \in [m]$.\\
    \item In $\OmWe$, there is a contribution $\OmWe_v$ for each vertex of the weave $\ww$ and a separate contribution for the bottom of the weave $\ww$. The bottom contribution can be considered as a contribution of the bottom side of $\triangle$ in the description of \cref{ssec:double_inductive_weaves}.
\end{enumerate}
For each depth $d\in[0,m]$, we will now show the identity
\begin{equation}\label{eq:truncated-forms}
\sum_{c \le d}\Om3D_{\invc} = \left(\sum_{v \mbox{\tiny{ s.t.~ }}d(v)\leq d} \OmWe_v\right) + \OmWe_{\eta_d}.
\end{equation}
The sum on the right of \cref{eq:truncated-forms} is taken over vertices $v$ that have depth $d(v)$ at most $d$, and $\OmWe_{\eta_d}$ is the 2-form defined in \cref{eq:boundary-form} for the weave $\DW_d$ obtained by truncating $\ww$ at depth $d$. \cref{eq:truncated-forms} implies \cref{thm:coincidence-forms} for maximal depth, as the left hand side coincides with $\Om3D$ and the right hand side coincides with $\OmWe$, by the defining \cref{eq:OmWe_definition}. Thus the remainder of the proof of \cref{thm:coincidence-forms} is now focused on showing \cref{eq:truncated-forms}. We write $\OmWe_{\DW_{d}}$ for the right hand side of \eqref{eq:truncated-forms} to ease notation. 

We prove \cref{eq:truncated-forms} by induction on the depth $d$, $d\in[0,m]$. The base case is $d=0$: since both sides of \eqref{eq:truncated-forms} are $0$, we have the required equality in the base case. We now assume that \eqref{eq:truncated-forms} is valid for depth $d$, which is our inductive hypothesis, and proceed to show that this implies \eqref{eq:truncated-forms} for depth $d+1$.

Since we are assuming that $\dbr$ has only positive letters, we have that $\br^{\dstr}_{d+1} = \br_d^{\dstr}i$ for some letter $i \in I$ where, as in \cref{def:dbl-string-and-braid-word-seq}, $\br^{\dstr}_{d}$ is the braid word obtained by truncating the double string $\dstr$ after $d$ steps. Let us use again the notation $w_{d}^{\dstr} := \delta(\br^{\dstr}_{d})$ for the Demazure product, and split the analysis into the two possible cases: $w^{\dstr}_{d+1} = w^{\dstr}_d$, where the Demazure product remains equal from depth $d$ to $d+1$, and $w^{\dstr}_{d+1} = w^{\dstr}_ds_i$, where it changes. The first case, which accounts for the trivalent vertices, is what requires most attention.

{\bf \framebox{Case 1: $w^{\dstr}_{d+1} = w^{\dstr}_d$}}. By \cite[Corollary 4.29]{CGGLSS}, braid moves at the bottom of a weave do not change the corresponding 2-form. Thus, after applying a sequence of braid moves, we can and do assume that the rightmost bottom leg of the weave $\DW_{d}$ is colored by $i$. Following \cref{acnsdjc}, let ${\rm A}_0, \ldots, {\rm A}_{d+1}\in \G/\U_+$ be a sequence of decorated flags such that ${\rm A}_0=\U_+$ and ${\rm A}_{k-1}\Rrel{s_{i_k}}{\rm A}_k$ for $k\in[1,d+1]$, and let $\Dot{w}\in \G$ be a lift of $w^{\dstr}_{d}$. By \cref{lem:properties-rel-pos}, the underlying flags of the pair $({\rm A}_0, {\rm A}_d)$ are in relative position $w^{\dstr}_{d}$, so there is a unique $h\in H$ such that 
\begin{equation}\label{eq:relposition_depthd}
({\rm A}_0, {\rm A}_d) = (g \U_+, g \Dot{w}h \U_+),\quad \mbox{for some }g\in \G.
\end{equation}

Let $\eta_d$ and $\eta_{d+1}$ be the bottom slices of $\DW_d$ and $\DW_{d+1}$ respectively. 
Suppose that the edges intersecting $\eta_d$ are colored $j_1, \ldots, j_{l-1}, j_l=i$, reading left to right, and let $u_1, \ldots, u_l$ be the $u$-variables associated with the slice $\eta_{d}$. By \cref{lem:hc-in-terms-of-uc}, the torus element $h\in H$ can be expressed as the product
\[h=\prod_{k=1}^l \chi_{k}^{\bf j}(u_k) \in H.\]
Let $x$ be the cluster variable associated to the south-most trivalent vertex $v$ of the weave $\DW_{d+1}$, drawn in red in \cref{acnsdjc}. Since we are in the case $w_{d}^{\dstr} = w_{d+1}^{\dstr}$, we have that
\begin{equation}\label{eq:relposition_depthd1}
({\rm A}_0, {\rm A}_{d+1})= (g\U_+, g\Dot{w} h' \U_+),\quad \mbox{for some }g\in \G,
\end{equation}
as in \cref{eq:relposition_depthd} but for depth $d+1$. By using \cref{lem:hc-in-terms-of-uc} again, this new torus element $h'\in H$ at depth $d+1$ is
\[
h'= \left(\prod_{k=1}^{l-1} \chi_k^{\bf j}(u_k)\right) \cdot \chi_l^{\bf j}(x) \in H.
\]

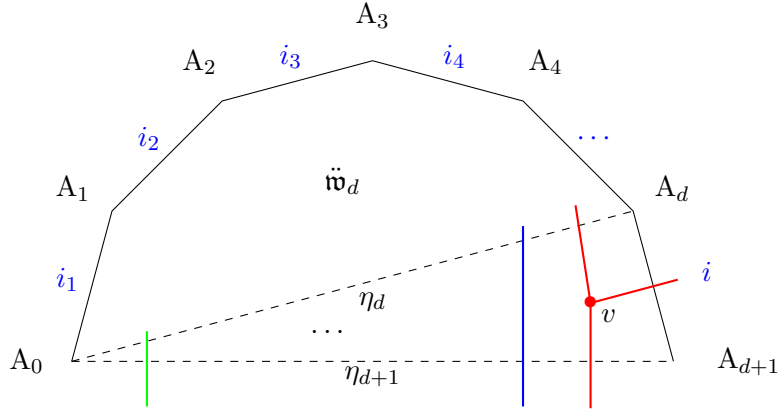
\begin{figure}[ht]
\begin{tikzpicture}[scale=2]
\draw (180:2)--(150:2)--(120:2)--(90:2)--(60:2)--(30:2)--(0:2);
\draw[dashed] (180:2) -- (0:2);
\draw[dashed] (180:2) -- (30:2);
\node at (180:2.3) {${\rm A}_0$};
\node at (150:2.3) {${\rm A}_1$};
\node at (120:2.3) {${\rm A}_2$};
\node at (90:2.3) {${\rm A}_3$};
\node at (60:2.3) {${\rm A}_4$};
\node at (30:2.3) {${\rm A}_{d}$};
\node at (0:2.5) {${\rm A}_{d+1}$};
\node[blue] at (165:2.1) {$i_1$};
\node[blue] at (135:2.1) {$i_2$};
\node[blue] at (105:2.1) {$i_3$};
\node[blue] at (75:2.1) {$i_4$};
\node[blue] at (45:2.1) {$\cdots$};
\node[blue] at (15:2.3) {$i$};
\node[red] at (15:1.5) {$\bullet$};
\node at (11: 1.6) {$v$};
\draw[thick, red] (15:2.1)--(15:1.5);
\draw[thick, red] (15:1.5) -- + (0,-0.7);
\draw[thick, red] (15:1.5)--+ (-0.1,0.65);
\draw[thick, blue] (1,0.9) -- (1, -.3);
\node[thick] at (-.3,0.2) {$\ldots$};
\draw[thick, green] (-1.5,0.2) -- (-1.5,-.3); 
\node at (-.2,1.2) {$\DW_{d}$};
\node at (0,-.1) {$\eta_{d+1}$};
\node at (0, .4) {$\eta_d$};
\end{tikzpicture}
\caption{The weave $\DW_{d+1}$, with depth $d+1$, constructed from the weave $\DW_{d}$, with depth $d$. It depicts the (red) trivalent vertex at the bottom right, labeled with color $i$, which is the key addition to obtain $\DW_{d+1}$ from $\DW_{d}$.}
\label{acnsdjc}
\end{figure}

As in \cref{ssec:Deodhar-2-form} above, we note that we have the equalities
\[
L_{d,i}=\frac{1}{2}\dlog \alpha_i(h), \qquad L_{d+1,i}=\frac{1}{2}\dlog \alpha_i(h'). 
\]
The inductive step needed to prove \cref{eq:truncated-forms} in this case $w^{\dstr}_{d+1} = w^{\dstr}_d$ is a now consequence of the following lemma:

\begin{lemma} 
\label{alma.cluster.form}$\OmWe_{\DW_{d+1}}- \OmWe_{\DW_{d}}= 2d_i L_{d,i}\wedge L_{d+1,i}.$
\end{lemma}
\begin{proof} 

By definition, we have 
\[
\OmWe_{\DW_{d+1}}- \OmWe_{\DW_{d}}
= \OmWe_{\eta_{d+1}} - \OmWe_{\eta_d} + \OmWe_v.
\]
In the notation above, the $u$-variables for the edges adjacent to $v$ are $u_l$, $x$, and 1. By Lemma \ref{lem.x-u.tri}, the local contribution of the trivalent vertex $v$ in the $u$-variables is
\begin{equation}\label{eq:partial-sum-2}
\OmWe_{v}=  2d_{i}\dlog u_l\wedge \dlog x.
\end{equation}
Note that $d_{j_k}(\alpha_k^{\bf j},\chi_l^{\bf j})=d_{j_l}(\alpha_l^{\bf j},\chi_k^{\bf j})=d_{i}(\alpha_l^{\bf j},\chi_k^{\bf j})$ and use Lemma \ref{lemma-u-slice} to obtain the equality
\begin{equation}\label{eq:partial-sum-1}
\OmWe_{\eta_{d+1}} - \OmWe_{\eta_d}
=\sum_{k=1}^{l-1}d_{i}(\alpha_l^{\bf j},\chi_k^{\bf j})\,  \dlog u_k \wedge \left(\dlog x - \dlog u_l\right).
\end{equation}
Since $\alpha_l^{\bf j}=\alpha_i$ and $(\alpha_l^{\bf j}, \chi_l^{\bf j})=2$, adding \eqref{eq:partial-sum-2} and \eqref{eq:partial-sum-1} yields
\begin{equation}\label{eq:lhs_2forms_weaves}
\OmWe_{\DW_{d+1}}- \OmWe_{\DW_{d}}
= \sum_{k=1}^{l}d_{i}(\alpha_i,\chi_k^{\bf j})\, \dlog u_k \wedge \left(\dlog x -  \dlog u_l\right)
\end{equation}
for the left hand side of the equality in the statement, that we are trying to establish. Its right hand side can be re-written using the following equalities:
\[
L_{d,i}=\frac{1}{2}\dlog \alpha_i(h)=\frac{1}{2} \dlog \alpha_i \left(\prod_{k=1}^l \chi_{k}^{\bf j}(u_k)\right)= \frac{1}{2} \dlog \prod_{k=1}^l u_k^{( \alpha_i ,\chi_k^{\bf j})}  =\frac{1}{2}\sum_{k=1}^l(\alpha_i, \chi_k^{\bf j}) \dlog u_k,
\]
and there the difference of the 1-forms $L_{d+1,i},L_{d,i}$ in depths $(d+1)$ and $d$ reads
\begin{equation}\label{eq:rhs_2forms_weaves}
 L_{d+1,i}-L_{d,i}=\frac{1}{2} \dlog \alpha_i(h'h^{-1})=\frac{1}{2} \dlog \alpha_i(\chi_i(xu_l^{-1}))=\dlog x - \dlog u_l.
\end{equation}
Therefore, Equations \eqref{eq:lhs_2forms_weaves} and \eqref{eq:rhs_2forms_weaves} together imply the required equality:
\[
\OmWe_{\DW_{d+1}} - \OmWe_{\DW_{d}} = 2d_{i}L_{d,i}\wedge (L_{d+1,i}-L_{d,i}) =  2d_{i}L_{d,i}\wedge L_{d+1,i}. \qedhere
\]
\end{proof}

By the definition of the Deodhar $2$-form, \cref{alma.cluster.form} implies that \cref{eq:truncated-forms} holds in this first case. To conclude \cref{thm:coincidence-forms}, it remains to establish the second case, which is done as follows.\\

{\bf \framebox{Case 2: $w^{\dstr}_{d+1} = w^{\dstr}_{d}s_i$}}.
In this case, the weaves $\DW_{d}$ and $\DW_{d+1}$ define the same exchange matrix, so we directly obtain $\OmWe_{\DW_{d+1}} = \OmWe_{\DW_{d}}$. That is, the right-hand side of \eqref{eq:truncated-forms} remains unchanged when passing from depth $d$ to depth $d+1$. By \cite[Corollary 2.26]{GLSB}, the same is true for the left-hand side. This concludes the argument in this case, and the proof of \cref{thm:coincidence-forms}.\qed

\color{black}
\subsection{Proof of \maintheoref}\label{ssec:proof_main_theorem} Let $\iso:R(\dbr)\lr X(\dbr^{(-|+)})$ be the isomorphism in \cref{lem:iso-glsbs-cgglss}. Then Part (1) is \cref{lem:weave-tori-vs-deodhar-tori}, Part (2) is \cref{thm:coincidence-cluster-variables} and Part (3) is \cref{thm:coincidence-forms}.\qed

\section{Type A: weaves and 3D plabic graphs}\label{sec:weave_from_3DPG}

Cluster algebras related to a Lie group $\G$ with Lie algebra $\mathfrak{g}=\mathfrak{s}\mathfrak{l}_n$, a.k.a.~Type A, are related to particularly beautiful combinatorics. Specifically, the connection between plabic graphs and cluster algebras has been fruitful, cf.~\cite[Chapter 7]{FWZ}, and see also \cite{KLS} or \cite{CGGS2} and references therein for equivalent combinatorial data in the case of positroid varieties. Even if the combinatorics of weaves, as developed in \cite{CasalsZaslow,CasalsWeng22,CGGLSS}, work for any $\G$ and provide a unifying combinatorial setting for all Lie types, it is valuable to relate weaves in a specific Lie type to previously studied combinatorics. For instance, \cite[Section 3]{CLSBW} shows how to transition from a reduced plabic graph to a weave in Type A, by the procedure of iterated $T$-shifts. In the Deodhar side, \cite{GLSBS} developed the combinatorics of 3D plabic graphs to study the Deodhar cluster structures in the specific case of double braid varieties in Type A. In this section, we further develop two aspects of our comparison between weave and Deodhar combinatorics in this Type A case:

\begin{enumerate}
    \item The combinatorial relation between 3D plabic graphs and weaves, in \cref{ssec:weaves_and_3Dplabicgraphs}.
    \item The relation between monotone multicurves, which are associated to 3D plabic graphs, and the tropical rules for Lusztig cycle propagation in weaves, \cref{ssec:compare-cocharacters-type-A}.
\end{enumerate}

The geometric relation between 3D plabic graphs and weaves, building on and generalizing \cite[Theorem 3.1]{CasalsLi22} and \cite[Section 5]{CLSBW}, is developed in \cite{CasalsKimWeng}.

\begin{remark}
\label{rmk:sign_convention_3D}
A technical point: the 3D plabic graph cluster seeds of \cite{GLSBS} and the Deodhar cluster seeds of \cite{GLSB} differ by an overall minus sign in their exchange 2-form, i.e.~their quivers are opposite from each other in their respective conventions.\qed
\end{remark}


\subsection{Weaves for 3D plabic graphs}\label{ssec:weaves_and_3Dplabicgraphs} 
Let us focus on Type A, setting $\G$ with Lie algebra $\mathfrak{g}=\mathfrak{s}\mathfrak{l}_n$, e.g.~$\G=\SL_n$. In \cite{CGGLSS}, cluster algebras are constructed using {\it weaves}. In \cite{GLSBS}, cluster algebras are constructed using {\it 3D plabic graphs}. The object of this subsection is to give a combinatorial procedure that inputs a 3D plabic graph and outputs a weave, up to equivalence, in a manner compatible with the comparison between the two works \cite{CGGLSS,GLSBS} as presented above. This procedure is consistent with the main comparison in the manuscript, e.g.~equivalence class of weaves obtained from a 3D plabic graph is such that the resulting cluster algebras coincide, and so does the initial seed, cf.~\maintheoref.


\subsubsection{Inductive weaves and red projections of 3D plabic graphs}\label{ssec:3DPlabicToWeave_RedCase} Let us present a combinatorial construction which inputs a 3D plabic graph associated to a braid word in the alphabet $I$ and outputs an inductive weave, up to equivalence. Consider a 3D plabic graph $\bG_{u,\dbr}$, as defined in \cite[Section 3]{GLSBS}. Following \cite[Section 6.4]{GLSBS}, we can and do assume that $u=\wo$. For now, we suppose that $\bG_{u,\dbr}$ only uses the alphabet $I$, i.e.~in this section $\dbr=\dbr^+$ denotes a braid word in the alphabet $I$ of the form $s_{i_\ell} \ldots s_{i_1}$. Note that then $\dbr^{(-|+)} = s_{i_1} \ldots s_{i_\ell}$.

As a consequence, the red projection of such 3D plabic graph $\bG_{u,\dbr}$, cf.~\cite[Section 3.2]{GLSBS}, is such that all short bridges have a black dot on top and it has no long bridges. In particular, we can draw $\bG_{u,\dbr}$ as a plabic fence, which contains the solid crossings in the language of \cite[Section 3]{GLSBS}, with additional hollow crossings interspersed as dictated by $u$. By definition, a weave class is an equivalence class of weaves, defined as a set of weaves up to weave equivalence, cf.~\cite[Section 4.2]{CGGLSS} or \cite[Theorem 4.2]{CasalsZaslow}.

We associate a weave $\ww(\bG_{u,\dbr})$ to $\bG_{u,\dbr}$ by scanning the 3D plabic graph right-to-left, drawing certain local models to construct $\ww(\bG_{u,\dbr})$ as we scan each crossing (solid or hollow) of $\bG_{u,\dbr}$. This weave $\ww(\bG_{u,\dbr})$ is well-defined up to weave equivalence, i.e.~$\ww(\bG_{u,\dbr})$ is a well-defined weave class. We draw the weave horizontally as in \cite{CasalsWeng22}, in line with the 3D plabic graphs of \cite{GLSBS}, which are also drawn horizontally. (Weaves are drawn vertically in \cite{CGGLSS} and in this manuscript, cf.~Subsection \ref{sssec:weave_for_Ialphabet} below for further discussion.)

The local model we insert for $\ww(\bG_{u,\dbr})$ depends on whether the crossing we scan in $\bG_{u,\dbr}$ is solid or hollow:

\begin{enumerate}
    \item If the crossing in in $\bG_{u,\dbr}$ is solid, the weave constructed thus far acquires a new trivalent vertex. The procedure is closely related to \cite[Section 3.3.2]{CasalsWeng22}. Intuitively, a new strand appears from the top vertically and attaches to the top (and left) of the existing weave with a trivalent vertex.\\
    
    \noindent This typically requires introducing hexavalent vertices to the left of the existing weave, before this new trivalent is created. These hexavalent vertices are necessary so as to modify the existing weave to a weave whose top (left) weave line matches the $I$-letter of the new strand appearing vertically from the top. In that situation, a trivalent vertex of the corresponding $I$-letter can then be attached to that top weave line. \\

    \item If the crossing in $\bG_{u,\dbr}$ is hollow, the weave does not acquire any new trivalent vertices. Intuitively, a new strand appears from the top vertically and takes a sharp left turn to become horizontal. At the leftmost point, the new strand union the existing weave (below the new strand) constitute the new piece of horizontal weave to the left.
\end{enumerate}

Let us now provide the necessary details, presented in Subsections \ref{sssec:weave_for_solidcrossing}, \ref{sssec:weave_for_hollowcrossing} and \ref{sssec:weave_for_Ialphabet}. The weave $\ww(\bG_{u,\dbr})$, up to weave equivalence, is then defined in Definition \ref{def:weave_3D_plabic} below. 
We use the convention from \cite{GLSBS} that the positive alphabet corresponds to bridges with white on top, which is opposite to the convention in \cite{CasalsWeng22}.


\subsubsection{Weaves when scanning solid crossings}\label{sssec:weave_for_solidcrossing} Let $w\in S_n$ be any permutation and $w=s_{j_r}\cdots s_{j_1}=(j_r,\ldots,j_1)$ a fixed reduced word of length $r=\ell(w)$. Consider the following local weaves $\mathfrak{n}(w),\cc^{\uparrow}(w)$, as in \cite[Section 3.3]{CasalsWeng22}.
 
\begin{definition}\label{def:nc_weaves}
The weave $\mathfrak{n}(w)$ is given by $n$ horizontal parallel weave lines such that the $k$-th strand, counting from the bottom, is labeled by the transposition $s_{j_k}$, $k\in[1,r]$. By definition, the weave $\cc^{\uparrow}(w)$ is the weave $\mathfrak{n}(w)$ with a trivalent vertex added at the top strand -- labeled by $s_{j_r}$ -- such that the third leg of this trivalent vertex is a vertical ray starting at the top strand and continuing upwards.\hfill$\Box$
\end{definition}

\noindent Figure \ref{fig:Weave_SolidCrossing1} depicts two instances of the weaves $\cc^{\uparrow}(w)$ in Definition \ref{def:nc_weaves}.

\begin{center}
	\begin{figure}[h!]
		\centering
		\includegraphics[scale=1.2]{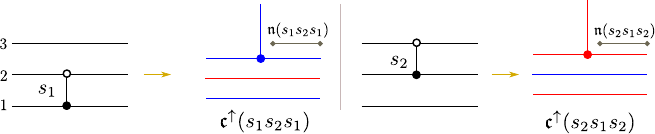}
		\caption{The weave $\cc^{\uparrow}(s_1s_2s_1)$ on the left and $\cc^{\uparrow}(s_2s_1s_2)$ on the right. Here weave lines labeled with $s_1$ (resp.~with $s_2$) are depicted in blue (resp.~red). We have also highlighted a portion of the weave that can be thought to correspond to $\mathfrak{n}(w)$, before adding any trivalent in \cref{def:nc_weaves}.\\}
		\label{fig:Weave_SolidCrossing1}
	\end{figure}
\end{center}

Suppose that we are scanning $\bG_{u,\dbr}$, equivalently $\dbr$, right-to-left and it is time to add a solid crossing. Let $\beta_{j-1}$ be the braid subword of $\dbr$ to the right of the $j$th crossing $i_j$ of $\dbr$, $j$th counting from the right; so $\beta_0$ is empty and $\beta_\ell=\dbr$ . Denote by $\ww(\beta_{j-1})$ the weave constructed thus far: at its left end, $\ww(\beta_{j-1})$ consists of a series of horizontal weave edges which spell a reduced word $\delta_{j-1}$ for the Demazure product $\delta(\beta_{j-1})$ of $\beta_{j-1}$.

\begin{remark}
In our convention, this reduced word $\delta_{j-1}$ for $\delta(\beta_{j-1})$ is spelled right-to-left when the weave lines at the left end of $\ww(\beta_{j-1})$ are read bottom-to-top; this matches how the letters of $w$ correspond to weaves lines of $\mathfrak{n}(w)$ in \cref{def:nc_weaves}.\hfill$\Box$
\end{remark}

The fact that this new crossing is solid (instead of hollow) is equivalent to the condition that the Demazure product of the braid $\beta_{j-1}$ does {\it not} change when we add $i_j$ to its left. That is, $\delta(s_{i_j}\beta_{j-1})=\delta(\beta_{j-1})$ and so there is a reduced word for $\delta(\beta_{j-1})$ which begins (from the left) with $s_{i_j}$. 

\begin{definition}\label{def:weave_sequence}
Let $\beta_{j-1}$ and $i_j\in I$ be such that $\delta(s_{i_j}\beta_{j-1})=\delta(\beta_{j-1})$. Let $\delta_{j-1}$ be a given reduced word for $\delta(\beta_{j-1})$ and $\sigma_{i_j}(\delta_{j-1})$ a (possibly different) reduced word for $\delta(\beta_{j-1})$ starting with $s_{i_j}$ at its left. By definition, the weave class $\n_{i_j}^{\uparrow}(\beta_{j-1})$ is the unique weave class represented by an homonymous weave $\n_{i_j}^{\uparrow}(\beta_{j-1})$ such that:

\begin{enumerate}
    \item The weave $\n_{i_j}^{\uparrow}(\beta_{j-1})$ contains only 4-valent and 6-valent vertices.

    \item The right end of $\n_{i_j}^{\uparrow}(\beta_{j-1})$ is $\delta_{j-1}$, i.e.~it spells $\delta_{j-1}$ when read bottom-to-top.

    \item The left end of $\n_{i_j}^{\uparrow}(\beta_{j-1})$ is $\sigma_{i_j}(\delta_{j-1})$.\hfill$\Box$
\end{enumerate}
\end{definition}

The fact that the three properties in Definition \ref{def:weave_sequence} determine a unique weave class follows from \cite[Theorem 4.12]{CGGS}. The weave vertices of $\n_{i_j}^{\uparrow}(\beta_{j-1})$ correspond to a sequence of braid moves (if 6-valent) and commutations (if 4-valent) that are applied to the reduced word $\delta_{j-1}$ to obtain $\sigma_{i_j}(\delta_{j-1})$. See Figure \ref{fig:Weave_SolidCrossing_StrandUp} 
for four examples of \cref{def:weave_sequence} . The theory of higher syzygies for permutations and weave equivalence moves are such that different choices of such sequences of braid moves yield equivalent weaves, cf.~\cite[Theorem 4.12.(a)]{CGGS} or \cite[Section 4.2]{CGGLSS}. Note also that there are no trivalent vertices in $\n_{i_j}^{\uparrow}(\beta_{j-1})$.

\begin{center}
	\begin{figure}[h!]
		\centering
		\includegraphics[scale=1.2]{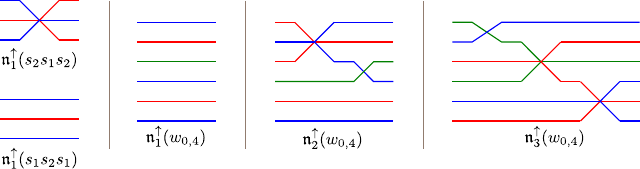}
		\caption{The weaves $\n_{i_j}^{\uparrow}(\beta_{j-1})$ in \cref{def:weave_sequence}. Here we abbreviated $w_{0,4}=s_1s_2s_3s_1s_2s_1$ for this reduced word for $w_0\in S_3$. In this example, we have chosen the Demazure products $\delta_{j-1}$ to be $w_0$, as it is the most interesting case: we emphasize that $\delta_{j-1}$ can be any permutation.}
		\label{fig:Weave_SolidCrossing_StrandUp}
	\end{figure}
\end{center}

\begin{remark}\label{rmk:weave_bring_strand_up}
(1) If $\delta_{j-1}$ starts with $s_{i_j}$ on its left, then we can choose $\sigma_{i_j}(\delta_{j-1})=\delta_{j-1}$ and $\n_{i_j}^{\uparrow}(\beta_{j-1})$ is represented by the identity weave: all horizontal lines spelling $\delta_{j-1}$ bottom-to-top and with no vertices at all.\\

(2) More generally, a representative for $\n_{i_j}^{\uparrow}(\beta_{j-1})$ can be chosen with a minimal number of weave vertices, by choosing the sequence of braid moves from $\delta_{j-1}$ to a word starting with $s_{i_j}$ in a greedy manner. In practice, this is often the choice of representative.\\

(3) The condition $\delta(s_{i_j}\beta_{j-1})=\delta(\beta_{j-1})$ in \cref{def:weave_sequence} is necessary to construct such weave $\n_{i_j}^{\uparrow}(\beta_{j-1})$. Indeed, if the Demazure product increased, i.e.~$\delta(s_{i_j}\beta_{j-1})=s_{i_j}\delta(\beta_{j-1})$, there would be no reduced word for $\delta(\beta_{j-1})$ starting with $s_{i_j}$, so $\n_{i_j}^{\uparrow}(\beta_{j-1})$ could not exist.\hfill$\Box$
\end{remark}

\begin{definition}[\color{black}Weave class for a solid crossing\color{black}]\label{def:weave_solid_crossing}

The weave class $\mathfrak{s}(\bG_{u_{j-1},\beta_{j-1}},j;\ww(\beta_{j-1}))$ is given by the horizontal concatenation of $\cc^{\uparrow}(\gamma_{j-1})$, $\n_{i_j}^{\uparrow}(\beta_{j-1})$ and $\ww(\beta_{j-1})$, where $\cc^{\uparrow}(\gamma_{j-1})$ is concatenated to the left of $\n_{i_j}^{\uparrow}(\beta_{j-1})$ and $\n_{i_j}^{\uparrow}(\beta_{j-1})$ is concatenated to the left of $\ww(\beta_{j-1})$. Here we have assumed that:
\begin{enumerate}
   
    \item Representatives of the weave classes $\n_{i_j}^{\uparrow}(\beta_{j-1})$ and $\ww(\beta_{j-1})$ are chosen such that the right end of the (homonymous) weave representative $\n_{i_j}^{\uparrow}(\beta_{j-1})$ coincides with the left end of the weave representative $\ww(\beta_{j-1})$.\\

    \item $\gamma_{j-1}$ denotes the reduced word spelled by the left end of the weave representative for $\n_{i_j}^{\uparrow}(\beta_{j-1})$, which is a reduced word for $\delta(\beta_{j-1})$ starting with $s_{i_j}$.\hfill$\Box$
\end{enumerate}
\end{definition}

\noindent Intuitively described, the weave $\mathfrak{s}(\bG_{u_{j-1},\beta_{j-1}},j;\ww(\beta_{j-1}))$ in \cref{def:weave_solid_crossing} is obtained by starting with $\ww(\beta_{j-1})$, and then trying to attach a weave $s_{i_j}$-edge (coming from above) to the top of the leftmost part of $\ww(\beta_{j-1})$, which is what the piece $\cc^{\uparrow}(\gamma_{j-1})$ does. The issue is that the top weave line on the left boundary of $\ww(\beta_{j-1})$ might not be a weave $s_{i_j}$-line, and so we need to bring up a weave line on the left boundary of $\ww(\beta_{j-1})$ to its top so that the top becomes a weave $s_{i_j}$-edge: this is what $\n_{i_j}^{\uparrow}(\beta_{j-1})$ achieves and thus it is inserted between $\cc^{\uparrow}(\gamma_{j-1})$, to its left, and $\ww(\beta_{j-1})$, to its right.


\subsubsection{Weaves when scanning hollow crossings}\label{sssec:weave_for_hollowcrossing} For hollow crossings, we need a different weave than that in Definition \ref{def:weave_solid_crossing}. \cref{rmk:weave_bring_strand_up}.(3) implies that some of the weave pieces constructed do not exist if the Demazure product increases, and indeed the required weaves for hollow crossing must be built differently. We use the same notation as in Subsection \ref{sssec:weave_for_solidcrossing}, with $s_{i_j}$ being the $j$th crossing of $\beta$ reading left-to-right, so that $\beta_j=s_{i_j}\beta_{j-1}$.

\begin{center}
	\begin{figure}[h!]
		\centering
		\includegraphics[scale=1.2]{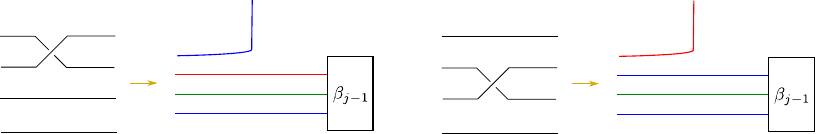}
		\caption{Examples of the weaves for hollow crossings of the 3D plabic graph, as in Definition \ref{def:weave_hollow_crossing}. (Left) The reduced word for the Demazure product of $\beta_{j-1}$ is $\delta(\beta_{j-1})=s_2s_3s_1$ and the new crossing has index $i_j=1$. (Right) The reduced word is $\delta(\beta_{j-1})=s_1s_3s_1$ and $i_j=2$. Here we use blue for $s_1$-edges of the weave, red for $s_2$-edges and green for $s_3$-edges.\\}
		\label{fig:RulesWeaves_CrossingExamples4}
	\end{figure}
\end{center}

\begin{definition}[\color{black}Weave class for a hollow crossing\color{black}]\label{def:weave_hollow_crossing}
By definition, $\mathfrak{h}(\bG_{u_{j-1},\beta_{j-1}},j;\ww(\beta_{j-1}))$ is the weave class obtained by choosing a weave representative of $\ww(\beta_{j-1})=\ww(\bG_{u_{j-1},\beta_{j-1}})$ an adding a weave edge of $s_{i_j}$-type, disjoint and on top of $\ww(\beta_{j-1})$, drawn as follows. The right end of the new weave $s_{i_j}$-edge is a vertical ray to the left of all the crossings of $\ww(\beta_{j-1})$, the left end of this new $s_{i_j}$-edge is a horizontal strand on top of the left end of $\ww(\beta_{j-1})$.

By definition, the weave class $\mathfrak{h}(\bG_{u_{j-1},\beta_{j-1}},j;\ww(\beta_{j-1}))$ is the weave equivalence class for any such (homonymous) weave $\mathfrak{h}(\bG_{u_{j-1},\beta_{j-1}},j;\ww(\beta_{j-1}))$.
\hfill$\Box$
\end{definition}

Figure \ref{fig:RulesWeaves_CrossingExamples4} depicts two instances of the weaves in Definition \ref{def:nc_weaves}. The weaves associated to hollow crossings have {\it no} trivalent vertices added to $\ww(\beta_{j-1})$, i.e.~the number of trivalent vertices of $\mathfrak{h}(\bG_{u_{j-1},\beta_{j-1}},j;\ww(\beta_{j-1}))$ and $\ww(\beta_{j-1})$ coincides. In a sense, it would {\it not} be possible to increase the number of trivalent vertices of $\ww(\beta_{j-1})$ by using the vertical weave $s_{i_j}$-line: the Demazure product increases when adding $s_{i_j}$ to the left of $\beta_{j-1}$ and thus there is no weave representative of $\ww(\beta_{j-1})$ such that the reduced word for the permutation at its left boundary has a weave $s_{i_j}$-line on top.


\subsubsection{The weave $\ww(\bG_{u,\dbr})$ }\label{sssec:weave_for_Ialphabet}  The required weave class $\ww(\bG_{u,\dbr})$ associated to $\bG_{u,\dbr}$ is defined by iteratively using Definitions \ref{def:weave_solid_crossing} and \ref{def:weave_hollow_crossing}, as $\beta$ is scanned right-to-left.

\begin{definition}[\color{black}Weave class of a 3D plabic graph\color{black}]\label{def:weave_3D_plabic} Let $\bG_{u,\dbr}$ be a 3D plabic graph, with $\dbr=(i_\ell,\ldots,i_1)$ a positive braid word in the positive alphabet of length $\ell=\ell(\dbr)$. By definition, the weave $\ww(\bG_{u,\dbr})$ is the final output of the following iterative algorithm:

\begin{enumerate}
    \item Initialize with $\ww(\beta_0):=\emptyset$, i.e.~the starting weave is the empty weave.\\

    \item For $j=1$ until $j=\ell$:

    \begin{itemize}
        \item[-] If $i_{j}$ is a solid crossing, then set $\ww(\beta_{j}):=\mathfrak{s}(\bG_{u_{j-1},\beta_{j-1}},j;\ww(\beta_{j-1}))$.

        \item[-] If $i_{j}$ is a hollow crossing, then set $\ww(\beta_{j}):=\mathfrak{h}(\bG_{u_{j-1},\beta_{j-1}},j;\ww(\beta_{j-1}))$.\\
    \end{itemize}

    \item Set $\ww(\bG_{u,\dbr}):=\ww(\beta_{\ell})$.
\end{enumerate}
The weave class $\ww(\bG_{u,\dbr})$ is said to be the weave class of the 3D plabic graph $\bG_{u,\dbr}$.
\hfill$\Box$
\end{definition}
\noindent See Figure \ref{fig:Example_3DPlabicToWeave1.pdf} for an example of Definition \ref{def:weave_3D_plabic}, drawing the weave $\ww(\bG_{u,\dbr})$ from $\bG_{u,\dbr}$.

\begin{center}
	\begin{figure}[h!]
		\centering
		\includegraphics[scale=1]{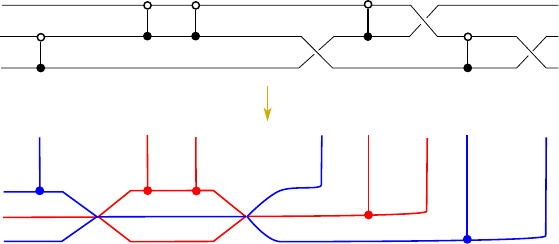}
		\caption{Example of a 3D plabic graph $\bG_{u,\dbr}$ (top) and its associated weave $\ww(\bG_{u,\dbr})$ (bottom), according to \cref{def:weave_3D_plabic}. The braid word is $\beta=(1,2,2,\mbox{\textcolor{purple}{1}},2,\mbox{\textcolor{purple}{2}},1,\mbox{\textcolor{purple}{1}})$, where we marked the hollow crossings in purple. After the third hollow crossing, counting from the left, the Demazure product is maximal (equal to $s_1s_2s_1$) and thus no more hollow crossings can appear.}
		\label{fig:Example_3DPlabicToWeave1.pdf}
	\end{figure}
\end{center}

\noindent For visual ease, we have drawn $\ww(\bG_{u,\dbr})$ horizontally, as it is built right-to-left as we scan $\bG_{u,\dbr}$ right-to-left. This is the same drawing convention we used in \cite{CasalsWeng22}, where weaves were built for grid plabic graphs. The 3D plabic graphs in \cite{GLSBS} and in the present article are drawn horizontally. In contrast, Demazure weave are often drawn vertically, see e.g.~\cite{CGGS,CGGLSS,CLSBW}, with a horizontal slice at the top spelling a braid word $\beta$ and the horizontal slice at the bottom spelling its Demazure product. The geometric reason for that is that there is directionality in the Lagrangian cobordisms that were being initially studied when weaves were first introduced in \cite{CasalsZaslow}. To bridge between these, we can rotate $\ww(\bG_{u,\dbr})$ to produce a vertical Demazure weave. In order to match conventions, so that $\ww(\bG_{u,\dbr})$ produces a right-inductive weave \cite{CGGLSS}, we always proceed as follows:
\begin{enumerate}
    \item Reflect $\ww(\bG_{u,\dbr})$ along a line of slope one disjoint from the weave and to its right,
    \item Extend the weave lines to the left of this rotated weave so that they all become vertical parallel lines pointing north.
\end{enumerate}
\noindent An example is depicted in Figure \ref{fig:Example_WeaveRotated_Revised.pdf}.

\begin{center}
	\begin{figure}[h!]
		\centering
		\includegraphics[scale=1.2]{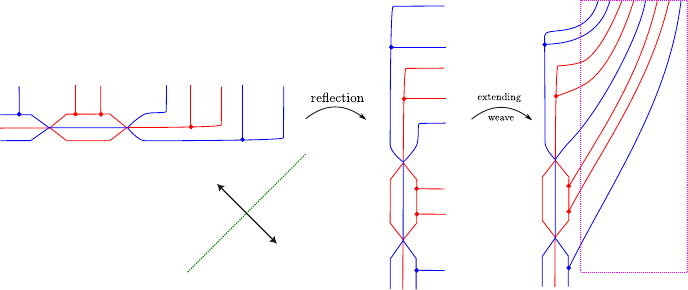}
		\caption{Reflecting the weave $\ww(\bG_{u,\dbr})$ to a right-inductive Demazure weave. The first step is the actual reflection along the dashed green line, and the second step is extending the weave lines to the right of the rotated weave so they reach the top. The extended weave lines are marked within the pink dashed box, for clarity.}
		\label{fig:Example_WeaveRotated_Revised.pdf}
	\end{figure}
\end{center}

\begin{remark} To be technically precise, depending on how one draws the vertical top weave lines in $\ww(\bG_{u,\dbr})$, the rotated weave is only equivalent to the right-inductive weave for the same word $\beta = \dbr^{(-|+)}$ after a planar isotopy: the rotated weave might not be literally Demazure when scanned via horizontal slices top-to-bottom, but it will be so after a planar isotopy. Such a minor correction is always implicitly assumed to be implemented.\hfill$\Box$
\end{remark}

By construction, after having rotated them as above, the weaves representing $\ww(\bG_{u,\dbr})$ in Definition \ref{def:weave_3D_plabic} are weave equivalent to right inductive weaves, the latter being defined in \cite[Section 4.3]{CGGLSS}. Indeed, the weaves $\ww(\bG_{u,\dbr})$ in \cref{def:weave_3D_plabic} are such that each trivalent vertex of their rotation has its incident top-right weave line go directly to the top of the weave: this characterizes right inductive weaves. We record this fact in the following statement:

\begin{lemma}\label{lem:3Dgraph_I_to_left_inductive}
Let $\dbr=\dbr^+$ be a braid word in the alphabet $I$ and $\bG_{u,\dbr}$ a 3D plabic graph. The right-inductive weave of $\dbr^{(-|+)}$ is a representative of the weave class $\ww(\bG_{u,\dbr})$. In particular, any weave representative of $\ww(\bG_{u,\dbr})$ is a Demazure weave.\hfill$\Box$
\end{lemma}

It is a consequence of the proof of \maintheoref that, under the comparison we established, the cluster seed associated to $\ww(\bG_{u,\dbr})$ in \cite{CGGLSS} coincides with the cluster seed associated to $\bG_{u,\dbr}$ in \cite{GLSBS} up to considering the opposite quiver, cf.~e.g.~ \cref{cor:inductive-weaves}.(1) and \cref{rmk:sign_convention_3D}.

\begin{remark}
For a double braid word $\dbr^-$ in the alphabet $-I$, there is an analogous procedure to build the corresponding weave class $\ww(\bG_{u,\dbr^-})$ from its associated 3D plabic graph $\bG_{u,\dbr^-}$. In this case, the strands for the weave would come from the bottom of the weave, instead of the top as in \cref{fig:Example_3DPlabicToWeave1.pdf}. The resulting weave, after reflection, is left-inductive. This is in line with \cref{cor:inductive-weaves}.(2), and the corresponding version of \cref{lem:3Dgraph_I_to_left_inductive} also holds.\qed
\end{remark}


\subsection{Comparison between weave and Deodhar cocharacters in Type A}\label{ssec:compare-cocharacters-type-A} \cref{cor:Deodhar_cocharacters_via_weaves} proved the equality
\begin{equation}\label{eq:equality_cocharacters}
\gamchp_{\dbr, \invc, \inve}=\gamWe_{\brWe, \c, \e}
\end{equation}
between the cocharacter $\gamchp_{\dbr, \invc, \inve}$, which encoded the order of vanishing of the minors $\grid_{\c, i}$ on $\tV_\e$, and the weave cocharacter $\gamWe_{\brWe, \c, \e}$, i.e.~the coweight of the Lusztig datum associated to a double inductive weave. In this section, we give a direct combinatorial argument for  \cref{eq:equality_cocharacters} in Type $A$ which does not make use of the torus element $\hp_{\invc}$ or the edge functions $u_\edge$, which were both key in the proof of \cref{cor:Deodhar_cocharacters_via_weaves}. Rather, we use \emph{monotone multicurves} on the Deodhar side, which directly encode the order of vanishing of grid minors along Deodhar hypersurfaces, and we use tropical Lusztig rules on the weave side. This type of propagation algorithm for cocharacters in terms of monotone multicurves did not quite feature in \cite{GLSBS} and, from the viewpoint of 3D plabic graphs, can be of combinatorial interest on its own. As in \cref{ssec:weaves_and_3Dplabicgraphs}, the focus now is on Lie Type A, so $\mathfrak{g}=\mathfrak{s}\mathfrak{l}_n$, and the Weyl group $W(\G)\cong S_n$ can be identified with the symmetric group $S_n$.


\subsubsection{Monotone multicurves and Lusztig data}\label{sssec:multicurves_Lusztig}
Following~\cite[Section~3.2]{GLSBS}, we define the \emph{permutation diagram} $\PD(u)\subset\Z^2$ of a permutation $u\in S_n$ as the set of dots with coordinates $(u(j), j)$ for $j\in[n]$. Thus, multiplying $u$ on the right (resp.~left) by $s_j$ corresponds to swapping rows (resp.~columns) $j$ and $j+1$, counting from the southwest corner.

\begin{definition}[{\cite[Section~3.4]{GLSBS}}]
A \emph{monotone curve} is a curve $\curve:[0,1]\to\R^2$ inside the permutation diagram $\PD(u)$ whose endpoints are dots in $\PD(u)$ and such that both coordinates of $\curve$ are strictly increasing functions on $[0,1]$. A \emph{monotone multicurve} is a collection $\bgamma = (\curve_1,\curve_2,\dots,\curve_k)$ of monotone curves inside $\PD(u)$ such that $\curve_i(1)$ is located strictly southwest of $\curve_{i+1}(0)$ for all $i\in[k-1]$. \hfill$\Box$
\end{definition}

Recall from \cref{def:lusztig-datum} that a \emph{Lusztig datum} $[\bj, f]$ is an equivalence class of weighted expressions. Our first goal is to associate a Lusztig datum to an arbitrary monotone multicurve $\bgamma$ inside the permutation diagram $\PD(u)$. For that, we consider the \emph{\dual propagation moves} as shown in \figref{fig:propag}(B). The \emph{propagation moves} of~\cite{GLSBS} are shown in \figref{fig:propag}(A).\\

\begin{figure}
\begin{subfigure}{\textwidth}
    \includegraphics[width=\textwidth]{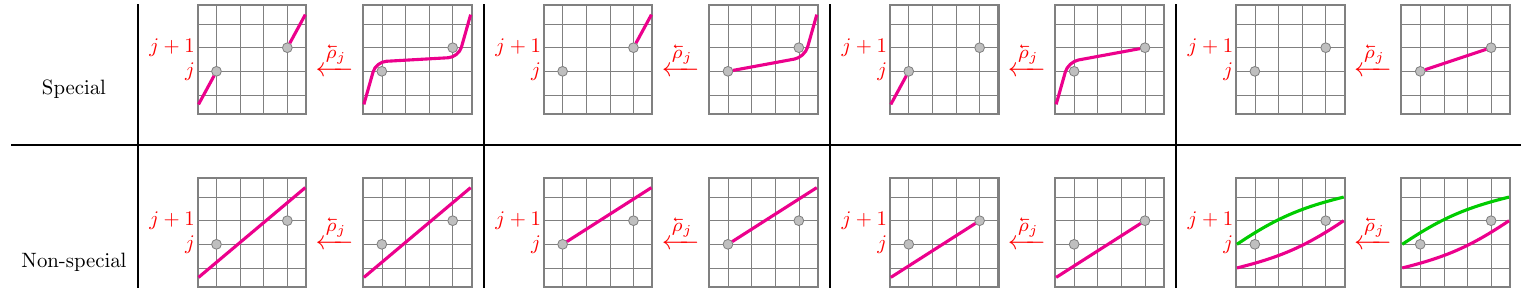}
    \caption{Propagation moves; cf.~\cite[Figure 9]{GLSBS}.}
    \label{fig:curve-propag-rules} 
    \vspace{0.5cm}
\end{subfigure}

\begin{subfigure}{\textwidth}
    \includegraphics[width=\textwidth]{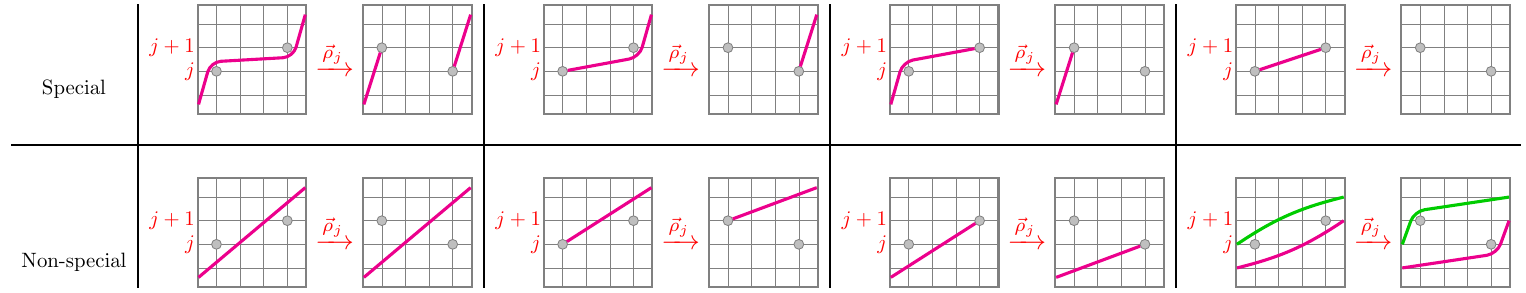}
    \caption{\Dual propagation moves, used in \cref{ssec:compare-cocharacters-type-A}.}
    \label{fig:rev}
  \end{subfigure}
\caption{\label{fig:propag} Propagation moves (A) and \dual propagation moves (B). In each case, the special moves are shown in the top row, with the non-special moves in the bottom row.} 
\end{figure}

Let $j\in I$ be such that $us_j>u$, that is, such that the dot $(u(j), j)$ is southwest of the dot $(u(j+1), j+1)$. Define another multicurve $\propagR_j(\bgamma)$ drawn inside the permutation diagram $\PD(us_j)$ to be obtained from $\bgamma$ by applying one of the moves shown in \figref{fig:propag}(B). (Exactly one of these moves applies in each case  of the shape of the intersection of curves $\gamma_j$  with the horizontal stripe $\Z \times ]j, j+1[$, 
depending on the configuration of $\bgamma$ relative to the two dots.  If $\bgamma$ does not intersect this stripe, we let $\propagR_i(\bgamma)$ be obtained via an endpoint-preserving isotopy that moves the  dot $(u(i), i)$ up, the dot $(u(i+1), i+1)$ down, and never passes a dot through a curve.) Strictly speaking, the resulting multicurve need not be monotone, but it is a collection of monotone curves, and we continue applying the \dual propagation moves to each of them separately. 

Similarly, for $j\in -I$ such that $s_{-j}u>u$, the \dual propagation moves $\propagR_j$ are obtained from those shown in \figref{fig:propag}(B) by reflecting each diagram across the line $y=x$.
Recall from \cref{def:si+-} that the two conditions ($us_j>u$ if $j\in I$, $s_{-j}u>u$ if $j\in -I$) may be combined as $s_j^-us_j^+>u$.


Let $w = \wo u$ and let $\bj = j_1j_2\dots j_\l\in (\pm I)^\l$ be a double braid word, where $\l = \ell(w)$. We say that $\bj$ is a \emph{double reduced word} for $w$ if $\delta(\bj) = w$. By~\eqref{eq:def:pos-neg-ind}, this condition may be explicitly written as
\begin{equation}\label{eq:double_red}
w = s_{j_\l^\ast}^- s_{j_{\l-1}^\ast}^- \cdots s_{j_1^\ast}^- \cdot s_{j_1}^+ s_{j_2}^+\cdots s_{j_\l}^+,
  \quad\text{i.e.,}\quad
  s_{j_1}^- s_{j_2}^- \cdots s_{j_\l}^-\cdot  u \cdot  s_{j_\l}^+ s_{j_{\l-1}}^+\cdots s_{j_1}^+ = \wo.
\end{equation}

The analogue of \cref{def:lusztig-datum} reads as follows:

\begin{definition}\label{dfn:double_red}
A \emph{double weighted expression} for $w$ is a pair $(\bj,f)$ where $\bj$ is a double reduced word for $w$ and $f:[\ell(w)]\to\Z_{\geq0}$ is an arbitrary weighting of the letters. We record the weights as exponents in angle brackets and write $(\bj,f) = j_1^{\<f_1\>}j_2^{\<f_2\>}\cdots j_\l^{\<f_\l\>}$. By definition, a \emph{double Lusztig datum} $[\bj,f]$ is an equivalence class of double weighted expressions modulo the following \emph{double braid moves}:
\begin{enumerate}[label=(B\arabic*)]
\item\label{typeA_move_B1}  $i^{\<a\>} j^{\<b\>} \leftrightarrow j^{\<b\>} i^{\<a\>}$ if $\sign(i)\neq\sign(j)$;
\item\label{typeA_move_B2}  $i^{\<a\>} j^{\<b\>} \leftrightarrow j^{\<b\>} i^{\<a\>}$ if $\sign(i)=\sign(j)$ and $|i-j|>1$;
\item\label{typeA_move_B3}  $i^{\<a\>} j^{\<b\>}i^{\<c\>} \leftrightarrow j^{\<a'\>} i^{\<b'\>} j^{\<c'\>}$ if $\sign(i)=\sign(j)$ and $|i-j|=1$, where 
\begin{equation}\label{eq:Lusztig_trop}
  a'=b+c-\min(a,c),\quad b'=\min(a,c), \quad\text{and}\quad c'=a+b-\min(a,c).
\end{equation}
\item\label{typeA_move_B4}  $j_1^{\<f_1\>} \leftrightarrow (-j_1^\ast)^{\<f_1\>}$ (changing the sign of the first letter).
\end{enumerate}
\qed
\end{definition}

\noindent See \cref{appendix:comparison_moves} for more details on the double braid moves. Note for now that all double reduced words for $w$ are related by double braid moves (B1) -- (B4). As we will see in \cref{dfn:fug} below, we will apply the moves $\propagR_j$ for each letter $j$ of $\bj$ in the reverse order, which is why \ref{typeA_move_B4} involves the first letter of $\bj$ as opposed to the last.

\begin{remark}\label{rmk:moves_double_to_single}
Each double Lusztig datum $[\bj,f]$ gives rise to a Lusztig datum similarly to \cref{def:pos-neg-ind}. Namely, let $a_1<a_2<\cdots<a_p$ and $b_1<b_2<\cdots<b_q$ be the indices of negative and positive letters in $\bj$, respectively. Define the weighted expression
$$(\bj,f)^{(-|+)}:=(-j_{a_p}^\ast)^{\<f_{a_p}\>}\cdots (-j_{a_1}^\ast)^{\<f_{a_1}\>} \cdot  j_{b_1}^{\<f_{b_1}\>} \cdots j_{b_\l}^{\<f_{b_\l}\>}$$
and let $[\bj,f]^{(-|+)}$ be the associated Lusztig datum. In this manner, the moves~\ref{typeA_move_B1} and~\ref{typeA_move_B4} do not change $(\bj,f)^{(-|+)}$ at all, while moves~\ref{typeA_move_B2} and~\ref{typeA_move_B3} translate into regular commutation and braid moves on $(\bj,f)^{(-|+)}$.\qed
\end{remark}



In the same manner that a weave defined a Lusztig datum, as in \cref{eq:coweight-of-Lusztig-datum} and \cref{def:lusztig-data-from-dbl-weave}, a monotone multicurve defines a double weighted expression as in \cref{dfn:double_red}, as follows.

\begin{definition}\label{dfn:fug}
Let $\bgamma$ be a monotone multicurve inside $\PD(u)$ and consider a double reduced word $\bj=j_1j_2\cdots j_\l$ for $w = \wo u$. The weighting $\fug:[\l]\to\Z_{\geq0}$ is defined as follows. 
We let $u^{(\l)}:=u$, $u^{(\l-1)}:= s_{j_\l}^-u s_{j_\l}^+$, \dots, $u^{(0)} := s_{j_1}^- s_{j_2}^-\cdots s_{j_\l}^- \cdot u \cdot  s^+_{j_\l}s^+_{j_{\l-1}} \cdots s_{j_1}^+ = \wo$; cf.~\eqref{eq:double_red}. Thus, $ s_{j_i}^-u^{(i)}s_{j_i}^+ > u^{(i-1)}$, so that a \dual propagation move can be applied at each step. Set $\bgamma^{(\l)}:=\bgamma$. For each $i=\l,\l-1,\dots,1$, we let $\bgamma^{(i-1)}:=\propagR_{j_{i}}(\bgamma^{(i)})$. If this \dual propagation move was special then we set $\fug(i):=1$, otherwise set $\fug(i):=0$. By definition, the double weighted expression for $w$ associated to $(\bj,\gamma)$ is $(\bj,\fug)$.\qed
\end{definition}

\begin{figure}[h!]
  \includegraphics[width=1.0\textwidth]{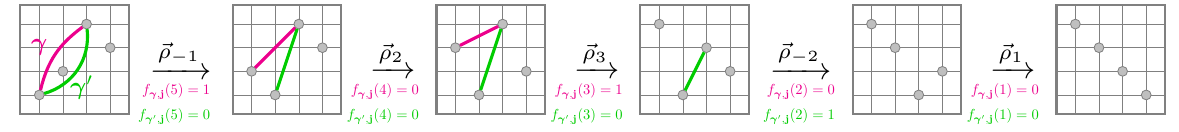}
  \caption{\label{fig:fug_ex} Computing the weighting $\fug$ for two monotone multicurves. This illustrates \cref{dfn:fug} in \cref{ex:fug}.}
\end{figure}

\begin{example}\label{ex:fug}
Let $n=4$, $u = s_3$, and $\bj = 1\,(-2)\,3\,2\,(-1)$. For the purple monotone multicurve $\bgamma$ in \cref{fig:fug_ex}, we have $(\fug(5),\dots,\fug(1)) = (1,0,1,0,0)$, while for the green monotone multicurve $\bgamma'$ in \cref{fig:fug_ex}, we have $(\fuggp(5),\dots,\fuggp(1)) = (0,0,0,1,0)$. Note that this uses both the moves in \figref{fig:propag}(B), for $\propagR_{j_{i}}$ with $j_i \in I$, and their reflections across the line $y = x$, for  $\propagR_{j_{i}}$ with $j_i \in -I$.
\qed
\end{example}

\begin{proposition}\label{prop:rev_invariant}
The double Lusztig datum $[\bj,\fug]$ does not depend on the choice of the double reduced word $\bj$ for $w=\wo u$. 

In other words, if two double reduced words $\bj,\bj'$ for $w$ are related by moves~\ref{typeA_move_B1}--\ref{typeA_move_B4}, then the weightings $\fug$ and $\fugjp$ are related by the corresponding tropical Lusztig moves from \cref{dfn:double_red}.
\end{proposition}
\begin{proof}
Consider the case of move~\ref{typeA_move_B1}: $i^{\<a\>} j^{\<b\>} \leftrightarrow j^{\<b\>} i^{\<a\>}$ with $\sign(i)\neq\sign(j)$. Suppose that, say, $i\in I$ and $j\in -I$. Let $A_i,B_i$ be the two dots involved in the move $\propagR_i$ and $C_j,D_j$ be the two dots involved in the move $\propagR_j$. Thus, $A_i$ and $B_i$ are located in adjacent rows while $C_j,D_j$ are located in adjacent columns. Since $\bj$ is reduced, we cannot have $\{A_i,B_i\}=\{C_j,D_j\}$. 
It is straightforward to check that regardless of the order in which we apply the two propagation moves, we have $b=1$ if $\bgamma$ contains a curve weakly to the right of $C_j$ and weakly to the left of $D_j$, and $b=0$ otherwise. Similarly, $a=1$ if  $\bgamma$ contains a curve weakly above $A_j$ and weakly below $B_j$, and $a=0$ otherwise.

The case of move~\ref{typeA_move_B2} is clear, since for $j\in I$, the \dual propagation move $\propagR_j$ only affects the multicurve in the neighborhood of a horizontal strip between rows $j$ and $j+1$. 


\begin{figure}
\begin{tabular}{c}
  \includegraphics[width=0.8\textwidth]{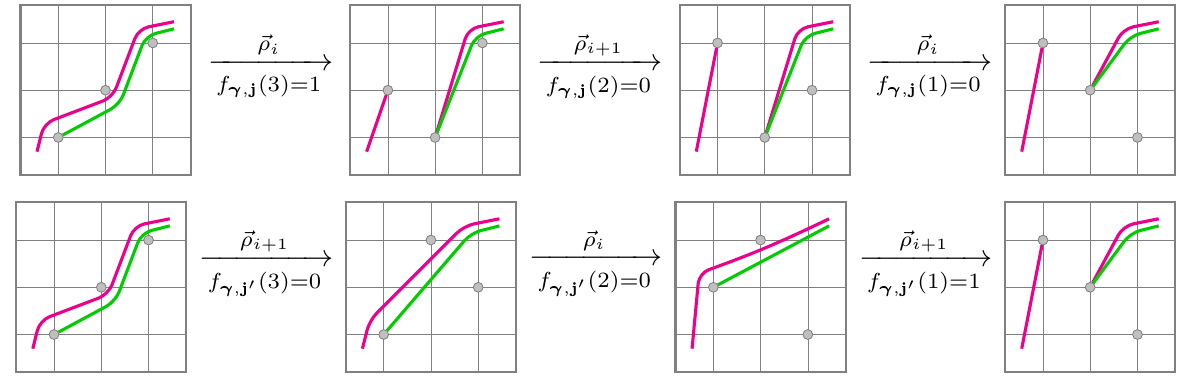}
\\\hline \\[-7pt]
  \includegraphics[width=0.8\textwidth]{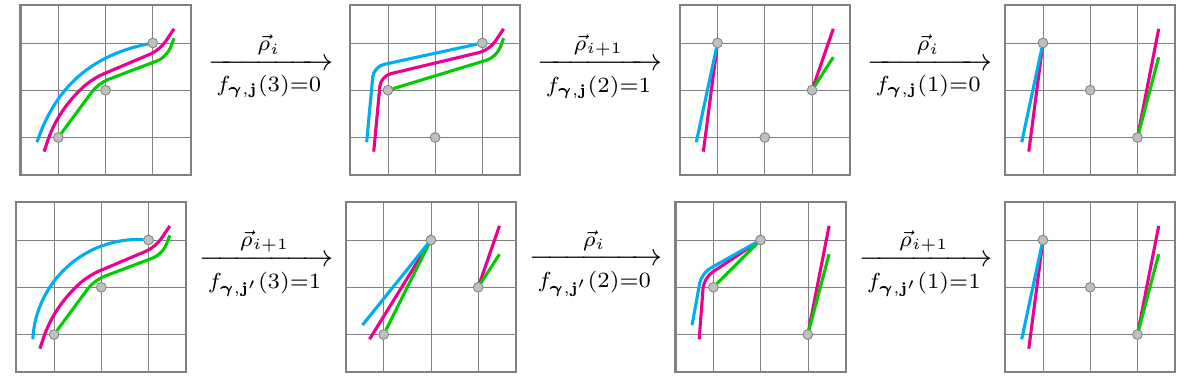}
\end{tabular}
  \caption{\label{fig:fug_check} Checking \cref{prop:rev_invariant} for move~\ref{typeA_move_B3}.}
\end{figure}

Suppose now that $\bj$ and $\bj'$ differ by a braid move~\ref{typeA_move_B3}: $i^{\<a\>} j^{\<b\>}i^{\<c\>} \leftrightarrow j^{\<a'\>} i^{\<b'\>} j^{\<c'\>}$ with, say, $i,j\in I$ and $j=i+1$. In this case, the result boils down to a case analysis. There are $26$ possibilities for how $\bgamma$ may be located relatively to the dots in rows $i,i+1,i+2$; see~\cite[Figure~20]{GLSBS}. A quick check confirms that in each case, the resulting weightings on $(i,i+1,i)$ and $(i+1,i,i+1)$ (which always take values in $\{0,1\}$) are related by the tropical Lusztig moves. For $5$ of these $26$ possibilities, the verification is performed in \cref{fig:fug_check}: in the top (resp.~bottom) row, the weightings $(\fug(3),\fug(2),\fug(1))\leftrightarrow(\fugjp(3),\fugjp(2),\fugjp(1))$ given by $(1,0,0)\leftrightarrow(0,0,1)$ (resp.~ $(0,1,0)\leftrightarrow(1,0,1)$) are indeed related by the tropical Lusztig relations~\eqref{eq:Lusztig_trop}. The remaining $21$ cases are checked similarly.

For move~\ref{typeA_move_B4}, observe that the move $\propagR_{j_1}$ is applied last, and the monotone multicurve $\bgamma^{(1)}$ from \cref{dfn:fug} is drawn inside the permutation diagram $\PD(u^{(1)})$ with $u^{(1)}=s_{j_1}^- \wo s_{j_1}^+$. Assume without loss of generality that $j_1\in I$. Then $u^{(1)}=\wo s_{j_1} = s_{j_1^*} \wo$. Therefore, regardless of whether we apply the move $\propagR_{j_1}$ or $\propagR_{-j_1^\ast}$, the resulting weight $f(1)$ will be equal to $1$ if $\bgamma^{(1)}$ contains a curve between the dots in rows (equivalently, columns) $j_1$ and $j_1+1$ and to $0$ otherwise. 
\end{proof}

\begin{definition}\label{dfn:fugmap}
Let $\bgamma$ be a monotone multicurve inside $\PD(u)$. We define $\fugmap(\bgamma):=[\bj,\fug]$, where $\bj$ is an arbitrary double reduced word for $w = \wo u$. By \cref{prop:rev_invariant}, the map $\fugmap$, from such monotone multicurves to double Lusztig datum, is well-defined.\qed
\end{definition}

Following~\cite{GLSBS}, to a collection $\bgamma$ of monotone curves we associate a \emph{red projection cocharacter} $\chip(\bgamma)$ defined as follows. Write $\bgamma=(\curve_1,\curve_2,\dots,\curve_k)$, and for each index $r\in[k]$, let $i_r,j_r$ be such that $\curve_r(0)=(u_\c(i_r),i_r)$ and $\curve_r(1)=(u_\c(j_r),j_r)$. We set
\begin{equation}\label{eq:chi_bgamma_dfn}
\chip(\bgamma):= \sum_{r=1}^k (\chi_{i_r} + \chi_{i_{r}+1}+\cdots + \chi_{j_r-1}).
\end{equation}
In other words, $\chip(\bgamma)$ is the sum of simple coroots $\chi_s$ with coefficients given by the number of times that the horizontal line at height $s+1/2$ intersects a monotone curve in $\bgamma$.

\begin{lemma}\label{lemma:chip_bgamma}
The cocharacter $\chip(\bgamma)$ coincides with the cocharacter $\charof(\fugmap(\bgamma))$ defined in~\eqref{eq:coweight-of-Lusztig-datum}.
\end{lemma}
\begin{proof}
We prove the result by induction on $\ell(\wo)-\ell(u)$. For the base case $u=\wo$, the multicurve $\bgamma$ must be empty, so $\chip(\bgamma)=\charof(\fugmap(\bgamma))=0$. 

For the induction step, let $\bj = j_1j_2\cdots j_\l$ be a reduced word for $w=\wo u$, with all letters positive. Set $i:=j_\l$ and $\bgamma':=\propagR_i(\bgamma)$. Assume that the result is known for $\bgamma'$. On the one hand, \eqref{eq:coweight-of-Lusztig-datum} implies
$$\charof(\fugmap(\bgamma)) = \fug(\l) \chi_i + s_i\cdot \charof(\fugmap(\bgamma')).$$
On the other hand, since $\propagR_i$ acts by $s_i$ on the vertical coordinates of the dots in $\PD(u)$, we obtain that $\chip(\bgamma)$ equals $s_i\cdot \chip(\bgamma')$ if the move $\propagR_i$ was not special, and equals $\chi_i + s_i\cdot \chip(\bgamma')$ if $\propagR_i$ was special. This completes the induction step.
\end{proof}


\subsubsection{Lusztig data under propagation moves}\label{sssec:propagation_Lusztig_data}
We continue to assume that $\bgamma$ is a monotone multicurve inside $\PD(u)$ and that $w=\wo u$, and recall the propagation moves from~\cite[Section 3.4]{GLSBS}, as illustrated in \figref{fig:propag}(A). We now prove \cref{lemma:fugmap_commute}, describing how the double Lusztig datum of a monotone multicurve, as introduced in \cref{dfn:fugmap}, changes when a propagation move is applied to the multicurve.

Suppose that $us_i>u$. Let the monotone multicurve $\propag_i(\bgamma)$ be obtained from $\bgamma$ by applying one of the moves in \figref{fig:propag}(A). Again, exactly one of the moves applies for each $\bgamma$. Note that unlike in the case of \dual propagation moves discussed above, the permutation diagram $\PD(u)$ is unaffected by propagation moves. Let us now show that the action of the propagation moves on monotone multicurves translates under the map $\fugmap$ from \cref{dfn:fugmap} into the action of the operators $\etop_i$ on double Lusztig data, $i\in\pm I$, introduced in \cref{eq:crystal_operator} below. These operators $\etop_i$ have been studied in relation to string parametrizations of canonical bases; see e.g.~\cite[Equation~(2.2)]{BZ93}.


Suppose that $i\in \pm I$ is such that $s_i^-us_i^+ > u$, so $s_{i^*}^-w s_i^+ < w$. Then, one can choose a double reduced word $\bj=j_1 j_2\cdots j_\l$ for $w$ that ends with $j_\l=i$.

\begin{definition}\label{eq:crystal_operator}
Given a double Lusztig datum $[\bj,f]$ with $\bj$ ending with $i$, we define
$$\etop_i[\bj,f]:=[\bj,f']$$
where for each $j\in[\l]$ the weighting $f'$ is set to be
$$f'(j):=
\begin{cases}
  f(j), &\text{if $j<\l$,}\\
  0, &\text{if $j=\l$.}
\end{cases}$$
\qed
\end{definition}

By~\cite[Theorem~2.2]{BZ93}, the output Lusztig datum $\etop_i[\bj,f]$ is well defined, i.e.~it does not depend on the choice of $\bj$ as long as $\bj$ ends with $j_\l=i$.\footnote{In~\cite{BZ93}, this result was only proved for the case when $i\in I$ and all letters in $\bj$ are positive. The extension to double reduced words follows from \cref{rmk:moves_double_to_single}.}

\begin{remark}\label{rmk:etop_neg}
Consider the case of a single reduced word $\bj$ on positive letters. Then for $i\in I$, the operator $\etop_i$ acts on $[\bj,f]$ by setting $f(\l)=0$ as above (assuming $j_\l=i$). For $i\in -I$, the operator $\etop_i$ may be described as follows. One can apply braid and commutation moves to $(\bj,f)$ until $\bj$ starts with $j_1=-i^\ast$, in which case $\etop_i$ acts by setting $f(1)=0$. Indeed, this description can be obtained from the one above by applying move~\ref{typeA_move_B4} to make the first letter $j_1=i$ of $\bj$ negative, then applying moves~\ref{typeA_move_B1} to make $i$ the last letter of $\bj$.\qed
\end{remark}


The change of Lusztig datum for monotone multicurves is precisely captured by the $\etop_i$ operators in \cref{eq:crystal_operator}:

\begin{lemma}\label{lemma:fugmap_commute}
Let $\bgamma$ be a monotone multicurve inside $\PD(u)$ and $i\in\pm I$ any index such that $s_{i}^- u s_i^+ > u$. Then the following equality holds:
\begin{equation}\label{eq:fugmap_commute}
   \fugmap(\propag_i(\bgamma)) = \etop_i ( \fugmap (\bgamma)).
\end{equation}
\end{lemma}
\begin{proof}
Suppose that $i\in I$. Then the assumption $s_{i}^- u s_i^+ > u$ translates into $u s_i > u$. Choose a reduced word $\bj=j_1j_2\cdots j_\l$ for $w$, with all letter positive and ending with $i$. Comparing the propagation moves $\propag_i$ and their duals $\propagR_i$, in \cref{fig:propag}(A) and (B), we obtain that $\propag_i(\bgamma)\neq \bgamma$ if and only if the move $\propagR_i(\bgamma)$ is special, which is equivalent to $\fug(\l)=1$. Suppose that we are in this case. Let $\bgamma':=\propag_i(\bgamma)$ and let $f':[\l]\to\Z_{\geq0}$ be such that $\etop_i[\bj,\fug]=[\bj,f']$. Thus, $f'$ differs from $\fug$ only in the value $f'(\l)=0$. Let $f'':=\fuggp$. By definition, $f''(\l)=0$ since the move $\propagR_i(\bgamma')$ is non-special. Furthermore, we have $\propagR_i(\bgamma') = \propagR_i(\bgamma)$, and thus $f''(j)=\fug(j)$ for all $j\in[\l-1]$. This implies that $f'=f''$, which is the content of~\eqref{eq:fugmap_commute}.

Suppose now that $\propag_i(\bgamma)= \bgamma$, or equivalently, $\fug(\l)=0$. Then we have $\fug=\fuggp$, so~\eqref{eq:fugmap_commute} holds in this case as well. The case $i\in -I$ is handled similarly.
\end{proof}


\subsubsection{Comparing the cocharacters}
We return to the setup of the rest of the paper. Let $\dbr = i_1i_2\cdots i_\nm \in (\pm I)^\nm$ be a double braid word and $\DW(\dbr)$ the associated double inductive weave. 

Let $\e \in \JWe$ and $\c\in[n+m]$ be such that $\c>\e$. On the one hand, from \cref{def:lusztig-data-from-dbl-weave}, the path $\Pbot_\c$ gives rise to a reduced word $\bjc = \jc_1\jc_2\cdots\jc_\l$ for $w_\invc$ and to a Lusztig datum $[\bjc, \nudc]$, where the weighting $\nudc$ is induced by the Lusztig cycle associated with the trivalent vertex located between depths $\e$ and $\e+1$.

On the other hand, the main construction of~\cite{GLSBS} gives rise to a monotone multicurve $\bgamma^{(\inve,\invc)}$ inside the permutation diagram $\PD(u_\invc)$, where $u_\invc:=\wo w_\invc$. It is defined inductively for $\invc=\inve-1,\inve-2,\dots,0$. Let $\bgamma^{(\inve,\inve-1)}$ consist of a single line segment connecting two dots located in rows $i_{\inve},i_{\inve}+1$, resp.~columns $|i_{\inve}|,|i_{\inve}|+1$, depending on the sign of $i_{\inve}$. Then, each $\bgamma^{(\inve,\invc-1)}$ is obtained from $\bgamma^{(\inve,\invc)}$ by either applying a propagation move $\propag_{i_{\invc}}$ if the crossing $\c\in\JWe$ is solid, or by applying an endpoint-preserving isotopy to $\bgamma^{(\inve,\invc)}$ that swaps two dots in adjacent rows $i_{\invc},i_{\invc}+1$, resp.~columns $|i_{\invc}|,|i_{\invc}|+1$, and never allows a monotone curve to pass through a dot. In other words, such an isotopy is obtained by reversing a non-special \dual propagation move shown in the bottom row of \figref{fig:propag}(B). See~\cite[Section~3.4]{GLSBS} for further details.

Our next result relates the weave Lusztig datum of~\cite{CGGLSS} to the monotone multicurves of~\cite{GLSBS}.
\begin{proposition}\label{prop:fugmap_vs_weave}
For each $\e \in \JWe$ and $\c\in[n+m]$ with $\c>\e$, we have 
\begin{equation}\label{eq:fugmap_vs_weave}
  [\bjc, \nudc] = \fugmap(\bgamma^{(\inve,\invc)}).
\end{equation}
\end{proposition}
\begin{proof}
Denote $\fugmap(\bgamma^{(\inve,\invc)})$ by $[\bi^{(\c)},\fdc]$. By construction, $\bgamma^{(\inve,\invc)}$ is a monotone multicurve inside $\PD(u_\invc)$ with $u_\invc=\wo w_\invc$, and thus $\bi^{(\c)}$ is a reduced word for $w_\invc$. We need to show that the (double) weighted expressions $(\bi^{(\c)},\fdc)$ and $(\bjc, \nudc)$ are related by a sequence of double braid moves~\ref{typeA_move_B1}--\ref{typeA_move_B4}.

Consider the base case $\c-\e=1$. If $i_{\inve}\in I$ (resp.~$i_{\inve}\in-I$) then $\nudc(\l)=1$ (resp.~$\nudc(1)=1$) and the remaining entries of $\nudc$ are zero. Proceeding similarly to \cref{rmk:etop_neg}, we apply double braid moves~\ref{typeA_move_B4} and~\ref{typeA_move_B1}  to $(\bjc,\nudc)$ to transform it into a double weighted expression $(\bj',\nu')$ such that $\bj'$ ends with $j'_\l=i_{\inve}$, and $\nu'(\l)=1$ and $\nu'(j)=0$ for $j<\l$. Since $\bgamma^{(\inve,\inve-1)}$ contains a single line segment between two dots in rows (resp.~columns) $|i_{\inve}|$ and $|i_{\inve}|+1$, we see that this line segment disappears after we apply the move $\propagR_{i_{\inve}}$. Thus, $f_{\bgamma^{(\inve,\inve-1)},\bj'} = \nu'$, which shows~\eqref{eq:fugmap_vs_weave} in the base case.

For the induction step, suppose that the result is known for some $\c>\e$. Assume first that $\c\in\JWe$ is a solid crossing. Recall that $\bgamma^{(\inve,\invc-1)} = \propag_{i_{\invc}}(\bgamma^{(\inve,\invc)})$. On the one hand, applying \cref{lemma:fugmap_commute} and the induction hypothesis, we obtain the equalities
\begin{equation*}
  \fugmap(\bgamma^{(\inve,\invc-1)}) = \etop_i ( \fugmap (\bgamma^{(\inve,\invc)})) 
= \etop_i [\bjc, \nudc].
\end{equation*}
On the other hand, $[\bj^{(\c+1)},\nu_\e^{(\c+1)}]$ is obtained from $[\bjc, \nudc]$ as follows. We have $\bj^{(\c+1)}=\bjc$ since the crossing $\c$ was solid. If $i_{\invc}\in I$ (resp.~$i_{\invc}\in-I$), we may assume that the last letter of $\bjc$ is $i_{\invc}$ (resp.~the first letter of $\bjc$ is $-i_{\invc}^\ast$). The weighting $\nu_\e^{(\c+1)}$ is obtained from $\nudc$ by setting the last (resp.~the first) entry to $0$ in view of the third diagram in~\eqref{fig:tropical-lusztig-rules}. By \cref{rmk:etop_neg}, this operation agrees with the action of $\etop_i$. This completes the induction step in the case $\c\in\JWe$.

Assume now that $\c\notin\JWe$ is a hollow crossing. On the one hand, in this case $[\bj^{(\c+1)},\nu_\e^{(\c+1)}]$ is obtained from $[\bjc, \nudc]$ by appending an extra letter $i_{\invc}$ with weight zero. On the other hand, $\bgamma^{(\inve,\invc-1)}$ is obtained from $\bgamma^{(\inve,\invc)}$ by an endpoint-preserving isotopy as above, so $\bgamma^{(\inve,\invc)} = \propagR_{i_{\invc}}(\bgamma^{(\inve,\invc-1)})$ are related by a non-special move $\propagR_{i_{\invc}}$. Thus, the last letter of each of $[\bj^{(\c+1)},\nu_\e^{(\c+1)}]$ and $\fugmap (\bgamma^{(\inve,\invc)})$ is $i_{\invc}$ with weight zero, which completes the induction step.
\end{proof}

We are ready to give a combinatorial proof of \cref{cor:Deodhar_cocharacters_via_weaves} in Lie Type A:
\begin{cor}
For each $\e \in \JWe$ and $\c\in[n+m]$ with $\c>\e$, we have 
\begin{equation*}
\gamchp_{\dbr, \invc, \inve} = \gamWe_{\brWe, \c, \e}.
\end{equation*}
\end{cor}
\begin{proof}
Comparing \cref{def:Deodhar-gamma-cocharacter} to~\cite[Lemma~7.13]{GLSBS}, we obtain the first equality $\gamchp_{\dbr, \invc, \inve} = \chip(\bgamma^{(\inve,\invc)})$, where $\chip(\bgamma^{(\inve,\invc)})$ is the red projection cocharacter defined in \cref{eq:chi_bgamma_dfn}. By \cref{lemma:chip_bgamma}, we have $\chip(\bgamma^{(\inve,\invc)}) = \charof(\fugmap(\bgamma^{(\inve,\invc)}))$. Then \cref{prop:fugmap_vs_weave} implies $\charof(\fugmap(\bgamma^{(\inve,\invc)})) = \charof[\bjc, \nudc]$, and the right hand side $\charof[\bjc, \nudc]$ is precisely $\gamWe_{\brWe, \c, \e}$ by \cref{def:lusztig-data-from-dbl-weave}.
\end{proof}

\section{Appendix A: Notation and conventions}\label{sec:appendix}
In this section, we set up notation and collect preliminaries on flags and their relative positions that will be useful throughout the manuscript.
\subsection{Pinnings}\label{ssec:pinnings}
We fix the following concepts and their notation for the entire article:
\begin{enumerate}

\item A simple algebraic group $\G$.
\item A pair of opposite Borel subgroups $\borel := \borel_+$ and $\borel_{-}$ of $\G$.
\item $\torus = \borel_{+}\cap\borel_{-}$, a maximal torus of $\G$.
\item $\U_+$, the unipotent radical of $\borel_+$, so that we have a decomposition $\borel_+ = \U_+H = H\U_+$. 
\item The Weyl group $W = N_{\G}(T)/T$, with simple reflections $s_i$, $i \in I$, where $I$ is the set of vertices of the Dynkin diagram of $\G$. For each element $w \in W$, we fix a lift $\dot{w} \in \G$. 
\item The longest element $\wo \in W$. 
\item Simple roots $\alpha_i$, $i \in I$.
\item Fundamental weights $\omega_i$, $i \in I$. 
\end{enumerate}

We also fix a \emph{pinning} of $(H, \borel_+, \borel_-, x_i, y_i; i \in I)$ of $\G$, cf.~\cite[Section~1.1]{Lus2}. That is, for each $i\in I$, we fix a group homomorphism 
\begin{equation*}%
	\phi_i:\SL_2\to G,\quad \begin{pmatrix}
		1 & t\\
		0 & 1
	\end{pmatrix}\mapsto x_i(t),\quad  \begin{pmatrix}
		1 & 0\\
		t & 1
	\end{pmatrix}\mapsto y_i(t),\quad
\begin{pmatrix}
	t & 0\\
	0 & t^{-1}
\end{pmatrix}\mapsto \chi_i(t),
\end{equation*} 
where $x_i(t),y_i(t)$ are the exponentiated Chevalley generators and $\chi_i: \C^* \to \H$ is the simple coroot corresponding to $i \in I$. We set
\begin{equation*}%
	\ds_i:=\phi_i \begin{pmatrix}
		0 & -1\\ 
		1 & 0
	\end{pmatrix}.
\end{equation*}
Note that this is consistent with the notation previously established, i.e., $\ds_i \in \G$ is a lift of the simple root $s_i \in W$. 

\subsection{Flags and weighted flags}\label{ssec:flags_and_weighted_flags} 

By definition, the \emph{flag variety} is the quotient $\G/\borel_+$. This is a projective algebraic variety and, since $\borel_+$ is self-normalizing, its elements can be identified with the Borel subgroups of $\G$. For each $w \in W$, we have an element $w.\borel_{+} \in \G/\borel_+$ that does not depend on the specific lift of $w$ to $\G$. Note that $\borel_- = \wo.\borel_{+}$, where $\wo \in W$ is the longest element. By definition, the variety of \emph{weighted flags} is the quotient $\G/\U_+$. It is a quasi-affine variety that admits a natural projection $\pi: \G/\U_+ \lr \G/\borel_+$.

\subsubsection{Relative position of flags} The $\G$-orbits in $\G/\B_+ \times \G/\B_+$ are parametrized by the elements of the Weyl group $W$. More precisely, for each pair of elements $(\borel_1, \borel_2) \in \G/\borel_+ \times \G/\borel_+$ there exists an element $g \in \G$ and an element $w \in W$ such that 
\[
(g.\borel_1, g.\borel_2) = (\borel_+, w.\borel_+).
\]
The element $w \in W$ is uniquely specified by the pair $(\borel_1, \borel_2)$ and we say that $\borel_1, \borel_2$ are in position $w$. We write
\[
\borel_1 \Rrel{w} \borel_2
\]
to express the fact that $\borel_1$ and $\borel_2$ are in position $w$. The following lemma summarizes basic properties of relative position between two flags, cf.~e.g.~\cite[Section 3.2]{CGGLSS}.

\begin{lemma}\label{lem:properties-rel-pos} The following properties hold:
\begin{enumerate}
\item $\borel_1 \Rrel{w} \borel_2$ if and only if $\borel_2 \Rrel{w^{-1}} \borel_1$. In particular, $\borel_1 \Rrel{s_i} \borel_2$ if and only if $\borel_2 \Rrel{s_i} \borel_1$.
\item Suppose $\borel_1 \Rrel{w} \borel_2$ and $\borel_2 \Rrel{w'} \borel_3$. If $\ell(ww') = \ell(w) + \ell(w')$, then $\borel_1 \Rrel{ww'} \borel_3$. 
\item $\borel \Rrel{w} \borel'$ if and only if there exist a reduced decomposition $w = s_{i_1} \cdots s_{i_{r}}$ and flags $\borel_1, \dots, \borel_{r-1}$ such that
\[
\borel \Rrel{s_{i_1}} \borel_1 \Rrel{s_{i_2}} \cdots \Rrel{s_{i_{r-1}}} \borel_{r-1} \Rrel{s_{i_r}} \borel'. 
\]
Moreover, the flags $\borel_1, \dots, \borel_{r-1}$ are unique.
\item Let $i \in I$. If $\borel_1 \Rrel{s_i} \borel_2$ and $\borel_2 \Rrel{s_i} \borel_3$, then either $\borel_1 = \borel_3$ or $\borel_1 \Rrel{s_i}\borel_3$.\qed
\end{enumerate}
\end{lemma}

\subsubsection{Relative position of weighted flags} For weighted flags, the notion of relative position is a bit subtler. By definition, two weighted flags $\U_1$ and $\U_2$ are said to be in weak relative position $w$ if there exist $g \in \G$ and $h \in H$ such that
\[
(g.\U_1, g.\U_2) = (\U_+, hw.\U_+). 
\]
As in the flag case, the element $w$ is uniquely determined by $\U_1$ and $\U_2$. We write $\U_1 \Rwrel{w} \U_2$ to express the fact that $\U_1$ and $\U_2$ are in {\it weak} relative position $w$. Note that
\[
\U_1 \Rwrel{w} \U_2 \qquad \text{is equivalent to} \qquad \pi(\U_1) \Rrel{w} \pi(\U_2). 
\]

By definition, two weighted flags $\U_1$, $\U_2$ are said to be in {\it strong} relative position $w$ if there exists $g \in G$ such that
\[
(g.\U_1, g.\U_2) = (\U_+, w.\U_+),
\]
and we write $\U_1 \Rrel{w} \U_2$ if this is the case. Note that $\U_1 \Rrel{w} \U_2$ implies $\U_1 \Rwrel{w} \U_2$ but the converse does not hold. In particular, $\U_1 \Rrel{\id} \U_2$ if and only if $\U_1 = \U_2$. 


\section{Appendix B: Comparison between moves}\label{appendix:comparison_moves}



The construction of the Deodhar cluster algebra structures on the coordinate rings $\C[R_\dbr]$ of double braid varieties $R_\dbr$, as developed in \cite{GLSBS,GLSB}, involves \emph{double braid moves} $\dbr\to\dbr'$ between double braid words $\dbr,\dbr'$, and the corresponding isomorphisms $\phi_{\dbr, \dbr'} : R(\dbr) \lr R(\dbr')$ between double braid varieties. Specifically, \cite[Section 4]{GLSB} analyzes the relationship between the Deodhar cluster seed $\Sigma_{\dbr}^{\rm D}$ for $\dbr$ and the pullback $\phi_{\dbr, \dbr'}^*(\Sigma_{\dbr'}^{\rm D})$ of the Deodhar cluster seed $\Sigma_{\dbr'}^{\rm D}$ for $\dbr'$ under the isomorphism $\phi_{\dbr, \dbr'}$.

Double braid moves are also studied for weave cluster algebra structures, as developed in \cite{CGGLSS}, under the different guise of weave moves and double string moves, cf.~\cite[Section 6.4]{CGGLSS}. This short appendix presents a comparison of these moves: specifically, \cref{table:move_comparison} records the results in \cite{GLSB} and \cite{CGGLSS} related to such braid moves, as well as the relationship between $\Sigma_{\dbr}^{\rm D}$ and $\phi_{\dbr, \dbr'}(\Sigma_{\dbr'}^{\rm D})$. The fact that these indeed match under the comparison is a consequence of \maintheoref and the results developed in the main body of the article.

\footnotesize
\begin{table}[h!]
\begin{center}
\caption{The comparison between the double braid moves from \cite{GLSB} and the double inductive weave moves from \cite{CGGLSS} under the cluster algebra isomorphism established in \maintheoref. By definition, a move is said to be \emph{all solid} if all letters involved are solid. In this table, we also use $c$ for the index of the rightmost letter in the move. The moves are written in black, the sections in \cite{GLSB} and \cite{CGGLSS} where the move is studied are written in \textcolor{purple}{purple}, and the effect on the seed (which is equal on both sides) is written in \textcolor{blue}{blue}.}\label{table:move_comparison}
\begin{tabular}{|c|c|}
\hline
Double braid move \cite{GLSB} & \quad Double string \& weave move  \cite{CGGLSS} \quad \\
\hline & \\
(B1): $ij \leftrightarrow ji$ if ${\rm sign}(i) \neq {\rm sign}(j)$  & $\dots, iR, jL, \dots \leftrightarrow \dots, jL, iR, \dots$ \\
\textcolor{purple}{Sects.~4.1, 4.2} & \textcolor{purple}{Theorem 6.8} \\ 
\textcolor{blue}{Mutation at $c$ if \emph{special} (all solid \& $w_c s_{|i|} = s_{|j|^*} w_c$),} & \\ \textcolor{blue}{else relabeling.}
& \\  \hline & \\ 
 (B2): $ij \leftrightarrow ji$ if ${\rm sign}(i) = {\rm sign}(j)$, $s_{|i|} s_{|j|}=s_{|j|} s_{|i|}$  & $iX, jX \leftrightarrow jX, iX$ if $s_{i} s_{j}=s_{j} s_{i}$ \\ 
\textcolor{purple}{Sect.~4.3} &  \\
\textcolor{blue}{Relabeling.} & \\
& \\ 
 \hline
  & \\ 
 (B3), short: $iji \leftrightarrow jij$ if ${\rm sign}(i) = {\rm sign}(j)$ and & $iX, jX, iX \leftrightarrow jX, iX, jX$ if \\ $s_{|i|}s_{|j|} s_{|i|}=s_{|j|}s_{|i|} s_{|j|}$ & $s_{i}s_{j} s_{i}=s_{j}s_{i} s_{j}$ \\
  \textcolor{purple}{Sects. 4.4, 4.5} & \textcolor{purple}{Proposition 4.46} \\ 
  \textcolor{blue}{Mutation at $c$ if all solid, else relabeling.} & \\
  & \\ 
  \hline & \\ 
   (B3), long: $\underbrace{iji\dots}_{\text{$m_{ij}$ letters}}  \leftrightarrow \underbrace{jij\dots}_{\text{$m_{ij}$ letters}}$ if ${\rm sign}(i) = {\rm sign}(j)$, & $\underbrace{iX,jX,iX,\dots}_{\text{$m_{ij}$ letters}}  \leftrightarrow \underbrace{jX,iX,jX,\dots}_{\text{$m_{ij}$ letters}}$ if \\ $(s_{|i|}s_{|j|})^{m_{ij}}=1$ with $m_{ij}>3$ & $(s_{i}s_{j})^{m_{ij}}=1$ with $m_{ij}\geq3$ \\
   & \\
   \textcolor{purple}{Sect.~6.2}& \\
   \textcolor{blue}{Sequence of mutations $\mu_{\rm fold}$ in \cite[Table 1]{GLSB}.}& \\
   & \\ \hline &\\
  (B4): $\dbr i \leftrightarrow \dbr (-i)^{\ast}$ & $(iL, \dots) \leftrightarrow (iR, \dots)$ \\
  \textcolor{purple}{Sect.~4.6} & \textcolor{purple}{Theorem 6.8}\\ 
  \textcolor{blue}{Seeds are equal.} & \\
  & \\ 
  \hline & \\
  (B5): $i\dbr \leftrightarrow (-i)\dbr$ & $i\beta \leftrightarrow \beta i^{\ast}$ \\ \textcolor{purple}{Sect.~4.7} & \textcolor{purple}{Sect.~5.5}\\ 
   \textcolor{blue}{Quasi-cluster transformation $x_1 \mapsto x_1^{-1} M$} & \\
   \textcolor{blue}{for $M$ Laurent monomial in frozens} &\\
  & \\ 
  \hline 
\end{tabular}
\end{center}
\end{table}
\normalsize

Note that on the weave side, applying move (B1) or (B4) preserves the braid word but changes the double inductive weave, see also \cref{rmk:compatibility-iso-moves}. Applying move (B2) or (B3) changes the braid word according to a commutation or braid move, while (B5) changes the braid word by cyclic rotation.

\begin{remark} Technically, \cite[Proposition 4.46]{CGGLSS} only treats the case of inductive weaves, as opposed to general double inductive weaves. Nevertheless, the same proof shows that the same result holds for double inductive weaves.\qed
\end{remark}

As a final comment, \cite[Section 8.3]{CGGLSS} describes geometrically the Donaldson-Thomas (DT) transformation on $X(\br)$ as the composition of two isomorphisms, cyclic rotation $\rho$ and the map $\ast$. Using \cref{table:move_comparison}, we can describe the pullback of the DT transformation $\ast \circ \rho^l$ explicitly as a sequence of mutations and rescalings in the context of \cite{GLSB}, as follows. To ease notation, if $\dbr'$ is related to $\dbr$ by a sequence of double braid moves, corresponding to an isomorphism $\phi$ of double braid varieties, we will write just $\dbr'$ for $\phi^*(\Sigma_{\dbr'}^{\rm D})$, the pullback of $\Sigma_{\dbr'}^{\rm D}$ to $\C[R(\dbr)]$. Then, given a double braid word $\dbr=i_1 \dots i_l$ with all positive letters, the pullback $(\ast \circ \rho^l)^*$ of DT applied to $\Sigma_{\dbr}^{\rm D}$ agrees with the following sequence of rescalings and mutations
    \[i_1 \dots i_l \xrightarrow{{\rm (B5)}} (- i_1) i_2 \dots i_l \xrightarrow{{\rm (B1)}} \dots \xrightarrow{{\rm (B1)}}i_2 \dots i_l (- i_1) \xrightarrow{{\rm (B4)}} i_2 \dots i_l i_1^* \xrightarrow{{\rm (B5)}} (-i_2) \dots i_l i_1^* \xrightarrow{{\rm (B1)}} \dots\]
    \[\dots \xrightarrow{{\rm (B4)}} i_3 \dots i_l i_1^* i_2^* \xrightarrow{{\rm (B5)}} \cdots \cdots \xrightarrow{{\rm (B1)}} i_1^* \dots (-i_l) \xrightarrow{{\rm (B4)}} i_1^* \dots i_l^*.\]

\section{Appendix C: Richardson varieties and braid varieties}\label{appendix:richardson_varieties}

\def\rpRich{\mathcal{R}}
\def\rmRich{\mathcal{R}^-}
\def\lpRich{{^+\mathcal{R}}}
\def\lmRich{{^-\mathcal{R}}}
\def\rpChi{\chi^+_{v,w}}
\def\lpChi{{^+\chi_{v,w}}}
\def\rmChi{\Vec{\chi}}
\def\lmChi{\cev{\chi}}
\def\lpretwist{\cev{\tau}^{\pre}}
\def\rpretwist{\vec{\tau}^{\pre}}
\def\ltwist{\cev{\tau}}
\def\rtwist{\vec{\tau}}
\def\bw{\mathbf{w}}
\def\sigL{\Sigma^{\Lec}}
\def\sigI{\Sigma^{\Ing}}
\def\sigM{\Sigma^{\Men}}

This short appendix discusses the relation between Richardson varieties and braid varieties, with comments on the relation between different cluster algebra structures in Type A. The main contribution is \cref{thm:richardson-compare}, where a precise comparison is established between the weave cluster seeds, the Deodhar cluster seeds, the cluster seeds constructed by B.~Leclerc, and the cluster seeds constructed by G.~Ingermanson.

For $v, w \in W$, the (open) Richardson variety $\rpRich_{v,w}$ is defined as the intersection
\[\rpRich_{v,w} := (\borel_- v \borel_+/\borel_+) \cap (\borel_+ w \borel_+/\borel_+) \subset G/\borel_+. \]
We also have the isomorphic varieties in different flag varieties:
\[\lmRich_{v,w} := (\borel_-\backslash \borel_- v \borel_+) \cap (\borel_- \backslash \borel_- w \borel_-) \subset \borel_-\backslash G, \] \[ \rmRich_{v,w} := (\borel_+ v \borel_-/\borel_-) \cap (\borel_- w \borel_-/\borel_-) \subset G/\borel_-.\]

\subsection{Summary on cluster algebras structures for Richardson varieties} B.~Leclerc defined a conjectural cluster structure on $\lmRich_{v,w}$ in types ADE in \cite{Leclerc}. G.~Ingermanson defined an upper cluster structure on $\rpRich_{v,w}$ in type A in \cite{Ingermanson}; this was compared to the Deodhar cluster structure in \cite[Section 10.2]{GLSBS} and to B.~Leclerc's type A conjectural cluster structure in \cite{SSB}. These comparisons respectively established that G.~Ingermanson's and B.~Leclerc's constructions give cluster structures on Richardson varieties. E.~M\'enard defined a conjectural cluster structure on $\lmRich_{v,w}$ in types ADE in \cite{menard2021algebres}, which was shown to be an upper cluster structure in \cite{CaoKeller}. E.~M\'enard's construction was compared to the weave cluster structure in \cite[Section 10]{CGGLSS}, where it was also shown to give a cluster structure on Richardson varieties.

In the following sections, we record the precise relationships between these cluster structures on Richardson varieties, and the weave and Deodhar cluster structures on braid varieties. These results are largely proved in \cite{CGGLSS,GLSBS,SSB}; we compile them here for convenience.

\subsection{Identifying Richardson varieties with braid varieties}
Comparing cluster structures on $\rpRich_{v,w}, \lmRich_{v,w}$ and (certain) braid varieties requires a choice of isomorphism between the varieties in question. In this subsection, we introduce a number of maps which we will use to define these isomorphisms.

First, we have that the projection maps $\borel_-\backslash\G\longleftarrow\G\lr\G/\borel_-$, $\borel_{-}g \mapsfrom g \mapsto g\borel_{-}$, restrict to isomorphisms
\begin{equation}
    \lmRich_{v,w} \xleftarrow{\delta_1} \dot{v} U_+ \cap U_+ \dot{v} \cap \borel_- w \borel_- \xrightarrow{\delta_2} \rmRich_{v,w}. 
\end{equation}

The \emph{right chiral map} $\rmChi_{v,w}: \lmRich_{v,w} \xrightarrow{\sim} \rmRich_{v,w}$ 
is defined as $\rmChi_{v,w} := \delta_2 \circ \delta_1^{-1}$. By definition, the inverse isomorphism is the \emph{left chiral map} $\lmChi_{v,w}: \rmRich_{v,w} \xrightarrow{\sim} \lmRich_{v,w}$.

Second, we have three involutions on $\G$, two anti-automorphisms $g \mapsto g^T$ and $g \mapsto g^\iota$ and one automorphism $g \mapsto g^\theta= (g^\iota)^T=(g^T)^\iota$. These maps are defined by:
\begin{align}
    &a^T=a\quad (a \in H), & x_i(t)^T=y_i(t),\qquad & y_i(t)^T=x_i(t)\\
    &a^\iota=a^{-1}\quad (a \in H), & x_i(t)^\iota=x_i(t),\qquad & y_i(t)^\iota=y_i(t)\\
    &a^\theta=a^{-1}\quad (a \in H), & x_i(t)^\theta=y_i(t),\qquad & y_i(t)^\theta=x_i(t).
\end{align}

Note that $\ds_i^\iota=\ds_i$ and $\ds_i^T=\ds_i^\theta=\ds_i^{-1}$ and $\borel_-^\theta=\borel_+$. It follows that $g \mapsto g^{\theta}$ descends to an involutive isomorphism
\begin{equation}
    \Theta: G/\borel_+ \xrightarrow{\sim} G/\borel_-, \quad \qquad g\borel_+ \mapsto g^{\theta}\borel_-
\end{equation}
which restricts to an involutive isomorphism $\Theta: \rpRich_{v,w} \xrightarrow{\sim} \rmRich_{v,w}$.

Finally, \cite[Definition 2.5]{galashin2022twist} introduced an isomorphism $\rpretwist_{v,w}: \rmRich_{v,w} \to \lmRich_{v,w}$. The \emph{right twist} $\rtwist_{v,w}:= \Theta^{-1} \circ \rmChi_{v,w}\circ \rpretwist_{v,w} \circ \Theta$ is then an automorphism of $\rpRich_{v,w}$.

Explicitly, we have
\[
\rpretwist_{v,w} \circ \Theta: g\borel_+ \mapsto \borel_- (v[g^{\iota }\dot{w}]_L^+),
\]
where $x_L^+ \in U_+$ is the first component in the image of $x \in \borel_+ \borel_-$ under the canonical isomorphism $\borel_+ \borel_- \overset\sim\to U_+ \times H \times U_-$. For instance, in type A this isomorphism is the UDL decomposition.

\begin{remark} \label{rem:identify-with-right-B+-cosets}
As braid varieties are defined using right $\borel_+$-cosets, our preference is to phrase
all cluster structures on $\rpRich_{v,w}$. In what follows, we use the isomorphisms
\[\Theta: \rpRich_{v,w} \xrightarrow{\sim} \rmRich_{v,w}\quad \text{and} \quad \lmChi_{v,w} \circ \Theta^{-1}: \rpRich_{v,w} \xrightarrow{\sim} \lmRich_{v,w}\]
to identify $\rpRich_{v,w}, \rmRich_{v,w}, \lmRich_{v,w}$ when necessary, and use their pullbacks to identify the appropriate coordinate rings.\qed
\end{remark}

\subsection{Comparison between cluster structures in Type A Richardson varieties} 
Fix $v \leq w \in W$ and let $\bw$ be a reduced expression for $w$. The constructions cited above of B.~Leclerc, E.~M\'enard, and G.~Ingermanson\footnote{Technically, Ingermanson only gives a seed for the \emph{unipeak} expression for $w$.} have in common that each reduced expression $\bw$ gives rise to a seed. We denote these seeds by $\sigL_{v,\bw}, \sigM_{v, \bw}$, and $\sigI_{v,\bw}$, respectively. By the convention in \cref{rem:identify-with-right-B+-cosets}, all seeds here are seeds for the Richardson variety $\rpRich_{v,w}$: e.g.~the functions of $\sigL_{v, \bw}$ differ from those defined by B.~Leclerc by the pullback in \cref{rem:identify-with-right-B+-cosets}. The cluster variables of $\sigL_{v,\bw}$ and $\sigI_{v,\bw}$ are indexed by the letters in the complement of the \emph{rightmost} subexpression for $v$ in $\bw$. 

Let $\beta=s_{i_1} \dots s_{i_\ell}$ be a reduced word for $\wo v$, and let $\gamma=s_{j_1} \dots s_{j_r}$ be the reverse of $\bw$. We define an isomorphism 
$p: \rpRich_{v,w} \lr X(\beta\gamma)$ via
\[p: g \borel_+ \mapsto (\borel_+ \xrightarrow{s_{i_1}} \cdots \xrightarrow{s_{i_{\ell}}} \dot{\wo} g^{-\iota} \borel_+ \xrightarrow{s_{j_1}} \cdots \xrightarrow{s_{j_{r}}} \wo \borel_+). \]

\begin{remark}
There are many ways to identify braid varieties and Richardson varieties. This isomorphism $p$ agrees with \cite[Eq. (10.2)]{GLSBS}, up to composition with $\iso$, and was chosen because it sends the Lusztig-positive part of $\rpRich_{v,w}$ to the cluster-positive part of $X(\beta \gamma)$. It differs from the isomorphisms of \cite[Section 3.6]{CGGLSS} and \cite[Section 2.7]{GLSBS}.\qed
\end{remark}

The relation between all the different cluster seeds in Type A, i.e.~weave seeds, Deodhar seeds, and the two types of seeds constructed by B.~Leclerc and G.~Ingermanson, is established in the following result:

\begin{thm}\label{thm:richardson-compare}
    Let $v, w \in S_n$ and let $\bw$ be a reduced expression for $w$. Let $\beta, \gamma$ be as above and define the double string $\dstr= (s_{i_1} R, \dots, s_{i_{\ell}}R,s_{j_1} R, \dots, s_{j_{r}}R)$. Let $\dbr$ be the double braid word $\dbr:=s_{j_r} \cdots s_{j_1} s_{i_\ell} \cdots s_{i_1}$. Then we have the following relation between cluster seeds for the Type A Richardson variety $\rpRich_{v,w}$:
   \[\Sigma^{\threeD}_{\dbr} \circ \iso^{-1} \circ p= \Sigma^{\We}_{\dstr} \circ p = \sigI_{v, \bw} = \sigL_{v, \bw} \circ \rtwist_{v,w}.\]
\end{thm}

A word is warranted regarding indices in the above equalities. The index set for $\Sigma^{\threeD}_{\dbr}$ is $\J3D$, the complement of the rightmost subexpression for $\wo$ in $\dbr$. All letters of the suffix $s_{i_{\ell} \dots s_{i_1}}$, a reduced word for $v^{-1} \wo$, are in the rightmost subexpression for $\wo$. The remaining letters of $\dbr$ in this subexpression form the rightmost subexpression for $v$ in $s_{j_r} \dots s_{j_1} = \bw$. Thus there is a natural bijection between $\J3D$ and the index sets for $\sigI_{v, \bw}$ and $\sigL_{v, \bw}$. The corresponding equality in \cref{thm:richardson-compare} above utilizes this bijection.


\begin{proof}[Proof of \cref{thm:richardson-compare}] The equality $\Sigma^{\threeD}_{\dbr} \circ \iso^{-1}= \Sigma^{\We}_{\dstr}$ follows from \maintheoref, which yields the first equality. As mentioned above, technically the seed $\sigI_{v, \bw}$ is defined only when $\bw$ is a unipeak expression, i.e.~ in the wiring diagram, strands never go up after they begin to go down. Every permutation has a unipeak expression. For $\bw$ not unipeak, we directly define $\sigI_{v, \bw}:= \Sigma^{\threeD}_{\dbr}$. For $\bw$ unipeak, the equality $\sigI_{v, \bw} = \sigL_{v, \bw} \circ \rtwist_{v,w}$ of the third and fourth seeds is \cite[Theorem B]{SSB} and the equality $\Sigma^{\threeD}_{\dbr} \circ \iso^{-1} \circ p = \sigI_{v, \bw}$ of the first and third seeds is \cite[Prop.~10.5 \& Rmk.~10.8]{GLSBS}. These equalities for arbitrary $\bw$ follow from the fact that commutation and braid moves on $\bw$ and $\dbr$ give rise to relabelings and mutations at the same indices in both $\sigL_{v, \bw}$ and $\Sigma^{\threeD}_{\dbr}$. See the proof of \cite[Proposition 7.1]{SSB} for the effect of commutation and braid moves on $\sigL_{v, \bw}$.
\end{proof}

\begin{remark} \cref{thm:richardson-compare} does not feature E.~M\'enard's seeds $\sigM_{v, \bw}$, but it follows from \cite[Section 10]{CGGLSS} that the quiver coincides with the quivers of the seeds in \cref{thm:richardson-compare}.\qed
\end{remark}
\bibliographystyle{alpha}
\bibliography{main}

\end{document}